\documentclass[12pt]{article}
\usepackage{amsmath,amsthm,amssymb,amsfonts,latexsym,mathrsfs,bm,tikz-cd,ytableau}
\usepackage{hyperref}
\usepackage{url}
\usepackage{setspace}
\usepackage{ytableau}
\usepackage{tikz}

\newtheorem{theorem}{Theorem}[subsection]
\newtheorem{lemma}{Lemma}[subsection]
\newtheorem{example}{Example}[subsection]
\newtheorem{definition}{Definition}[subsection]
\newtheorem{corollary}{Corollary}[subsection]
\newtheorem{proposition}{Proposition}[subsection]
\newtheorem{remark}{Remark}[subsection]
\usepackage[a4paper]{geometry}                           
\title{An automorphic description of the zeta function of the basic stratum of certain Kottwitz varieties}
\author{Yachen Liu}

\begin{document}	
\maketitle

\begin{abstract}
We derive formulas for the number of points on the basic stratum of certain Kottwitz varieties in terms of automorphic representations and certain explicit polynomials, for which we present efficient algorithms for computation. We obtain our results using the trace formula, base change, representations of general linear groups over p-adic fields, and a truncation of the formula of Kottwitz for the number of points on Shimura varieties over finite fields.

\end{abstract}
\newpage	
\tableofcontents
\newpage		
	\section{Introduction}
Let $X$ be a scheme of finite type over the finite field $\mathbb F_q$ with $q = p^r$ elements, where $p$ is a prime number. The zeta function $\zeta(X,s)$ of $X$ is by definition $\zeta(X,s) = \mathrm{exp}(\sum\limits_{m \geq 1 }^{\infty} \frac{N_m}{m}q^{-ms})$, where $N_m = |X(\mathbb F_{q^m})|$. By Grothendieck's Lefchetz trace formula, we have
$$N_m = \sum_{i = 1}^{2 \mathrm{dim}(X)} \mathrm{Tr}(Frob_{q}^m, H^i_{e t,c} (X_{\overline{\mathbb F}_q}, \overline{\mathbb Q}_{\ell})) = \sum_{i  = 1}^{2 \mathrm{dim}(X)} \alpha_{i,1}^m + \ldots \alpha^m_{i,d_i},$$ where \begin{itemize}
\item The number $\ell \not = p$ is a prime number.
\item The morphism $Frob_{q}:X \rightarrow X$ is the $q^{th}$ power Frobenius automorphism.
\item The numbers $\alpha_{i,1}, \ldots, \alpha_{i,d_i}$ are eigenvalues of the action of $Frob_q$ on $H^i_{e t,c} (X_{\overline{\mathbb F}_q}, \overline{\mathbb Q}_{\ell})$, where $d_i = dim_{\overline{\mathbb Q}_{\ell}}(H^i_{e t,c} (X_{\overline{\mathbb F}_q}, \overline{\mathbb Q}_{\ell}))$.
\end{itemize}
The zeta functions of schemes over finite fields are in general hard to compute. However, if $X$ is a geometrically irreducible smooth projective $n$-dimensional variety over $\mathbb F_q$, the Weil conjectures holds (see \cite{Gro66} and \cite{Del74})

\begin{itemize} 
\item For all $i,j$, the eigenvalues $\alpha_{i,j}$'s are algebraic integers.
\item In some ordering, we have $(\alpha_{i,1}, \ldots, \alpha_{i,d_i}) = (q^n/\alpha_{2n-i, 1}, \ldots, q^n/\alpha_{2n-i, d_{2n - i}})$.
\item For all $i,j$, we have $|\alpha_{i,j}| = q^{i/2}$. In particular, we have $d_0 = d_{2n} = 1$ and $\alpha_{0,1} =1$, $\alpha_{2n,1 } = q^n$.
\end{itemize}

If $X$ is singular, the Weil conjectures do not hold for $X$, and we do not know much about $\zeta(X,s)$. 
When $X$ is the reduction $\mathrm{mod}$ $p$ of the integral canonical model of a Shimura variety (in the sense of \cite{Kis10}, which is singular in general), it is conjectured that $\zeta(X,s)$ has an automorphic description, that is, the number of points $N_m = |X(\mathbb F_{q^m})|$ can be described using automorphic forms (see \cite{Kot90}). When $X$ is the reduction $\mathrm{mod}$ $p$ of Kottwitz varieties (which are certain compact Shimura varieites introduced in \cite{Kot92}),  In this article, we give an automorphic description of the zeta function of the basic stratum of certain Kottwitz varieties, using the Langlands - Kottwitz method (see \cite{Kot92}, which is adapted to study the basic stratum in \cite{Kret11}).  

	\subsection{Background}
	Let $(G,X)$ be a PEL-type Shimura datum (as defined in $\S4$ of \cite{kot92p}, for the defintion of Shimura varieties, see \cite{Del79}), let $p$ be a prime number such that the group $G_{\mathbb Q_p}$ is unramified. We choose a flat model $\mathcal G$ of $G_{\mathbb Q_p}$ over $\mathbb Z_p$, and we write $K_p = \mathcal G(\mathbb Z_p)$. Let $K^p$ be a neat compact open subgroup of $G(\mathbb A_f^p)$ and let $E$ be the reflex field of $(G,X)$. From the above data together with some other datum, Kottwitz constructed a PEL-type moduli problem of abelian varieties over $\mathcal O_E \otimes_{\mathbb Z} \mathbb Z_{(p)}$, denoted by $\mathcal M$, which is represented by a smooth quasi-projective scheme over $\mathcal O_E \otimes_{\mathbb Z} \mathbb Z_{(p)}$ (see $\S 5$ of \cite{kot92p}). We use $S$ to denote the moduli space for $\mathcal M$. Let $M_{K_pK^p}$ denote the canonical model of $Sh_{K_pK^p}(G,X)$ over $E$ defined by Deligne (see $\S$2.2 of \cite{Del79}), then $S_E$ is isomorphic to $|\mathrm {ker}^1(\mathbb Q, G)|$ copies of $M_{K_pK^p}$ (see $\S$8 of \cite{kot92p}). Let $v$ be a place of $E$ above $p$, we extend $S$ over $\mathcal O_{E_v}$ via the integral canonical model (as defined in \cite{Kis10}). Note that for abelian type Shimura varieties (which inludes PEL type Shimura varieties), integral canonical model exists in the case of good reduction (see \cite{Kis10}).
 
 Let $v$ be a place of $E$ above $p$. We write $q = |\mathcal O_{E_v} / (\pi_{E_v})|$ and let $\mathbb F_q = \mathcal O_{E_v} / (\pi_{E_v})$, where $\pi_{E_v}$ is a uniformizer in $\mathcal O_{E_v}$. We use $\bar S$ to denote $S_{\mathbb F_q}$. Let $B(G_{\mathbb Q_p})$ denote the set of isocrystals with additional $G$-structure (as introduced in \cite{Kot85}), which is the set of $\sigma$-conjugation class in $G(L)$, where $L$ is the completion of the maximal unramified extension of $\mathbb Q_p$, and $\sigma: G(L) \rightarrow G(L)$ is the Frobnius automorphism. 
 For each geometric point $x \in \bar S(\overline {\mathbb F}_q)$, we consider the isocrystal with additional $G$-structure $b(\mathcal A_x)$ associated to the abelian variety $\mathcal A_x$ corresponding to $x$ (see pp.170-171 of \cite{RaRi96}). The set of isocrystals with additional $G$-structure $\{b(\mathcal A_x): x \in \bar S(\overline {\mathbb F}_q)\}$ 
 is contained in certain finite set $B(G_{\mathbb Q_p} , \mu) \subset B(G_{\mathbb Q_p})$ (see $\S 6$ of \cite{Kot97}). For $b \in B(G_{\mathbb Q_p} )$, we define $S_b$ to be $\{x \in \bar S(\overline {\mathbb F}_q): b(\mathcal A_x) = b\}$. Then $S_b$ is a locally closed subset of $\bar S_{\overline {\mathbb F}_q}$ (see $\S 1$ of \cite{Ra02}), and we equip $S_b$ with the structure of an algebraic variety induced from $S_b$. The collection $\{S_b\}_{b\in B(G_{\mathbb Q_p},\mu)}$ of subvarieties of $\bar S$ is called the Newton stratification of $\bar S$.

We take $b$ to be the unique basic element of $B(G_{\mathbb Q_p},\mu)$ (see $\S 6.4$ of \cite{Kot97}), and we call $S_b$ the basic stratum (in this case, the basic stratum $S_b$ is closed in $\bar S_{\overline{\mathbb F}_q}$, see $\S 1$ of \cite{Ra02}). The basic stratum of PEL-type Shimura varieties (in the case of good reduction) has been studied by many people using geometric methods, before we explain the case considered in our paper, we first present some examples. We first consider the case $(G,X) = (GSp_{2g}, \mathcal H_g^{\pm})$ (the Siegel modular varieties case, as defined in $\S 3.1 $ of \cite{Lan17}): If $g = 1$, then $S_b$ is a finite set. If $g = 2$, then irreducible components of $S_b$ are isomorphic to $\mathbb P^1$, which intersect pairwise transversally at the superspecial points (i.e., the points where the underlying abelian variety is isomorphic to a power of a supersingular elliptic curve), see \cite{Kob} and \cite{KO85}. If $g = 2$, then every irreducible component of $S_b$ is birationally equivalent to a $\mathbb P^1$-bundle over a Fermat curve, see \cite{LO98}. For general $g>0$, the
dimension and the number of irreducible components of $S_b$ are computed in \cite{LO98}. 
For the baisc stratum of a Hilbert-Blumenthal variety associated
with a totally real field extension of degree $g$ over $\mathbb Q$ (as defined in $\S 1$ of \cite{Yu03}): See \cite{RK99} for the case where $g = 2$ and $p$ is an inert prime. The case where $g = 3$ and $p$ is an inert prime is studied in \cite{Go}. For the case where $g = 4$ and $p$ is an inert prime, see \cite{Yu034}.

For $\mathrm{GU}(r,s)$ Shimura varieties associated to quasi-split unitary groups (associated to an imaginary quadratic field extension of $\mathbb Q$, as defined in p670 of \cite{Vol10}), if $r = 1$ and $p$ is an inert prime, the basic stratum is pure of dimension $[s/2]$ and locally of complete intersection (see \cite{Vol11}). 
In the case of $(r,s) = (1,2)$ or $(r,s) = (1,3)$, and $p$ is an inert prime, the irreducible components of the basic stratum are Fermat curve in $\mathbb P^2$ given by equation $x_0^{p+1} + x_1^{p+1} + x_2^{p+1} = 0$, the singular points are the superspecial points and there are $p^3 +1 $ superspecial points on any irreducible components, each superspecial point is the
pairwise transversal intersection of $p + 1$ irreducible components if $s = 2$, and each superspecial point is the
pairwise transversal intersection of $p^3 + 1$ irreducible components if $s = 3$, see \cite{Vol10} for $s = 2$ and \cite{Vol11} for $s = 3$. In the case of $(r,s) = (2,2)$ and $p$ is an inert prime, the basic stratum is pure of dimension $2$, each irreducible component of the basic stratum is isomorphic to the Fermat surface in $\mathbb P^3$ given by equation $x_0^{p+1} + x_1^{p+1} + x_2^{p+1} + x_3^{p+1} = 0$, and if two irreducible components intersect nontrivially, the reduced scheme underlying
their scheme-theoretic intersection is either a point or a projective line, see \cite{Ho14}.
If $s = 2$ and $p$ is an inert prime, the irreducible components of the basic stratum are realised as closed subschemes of flag schemes over Deligne–Lusztig
varieties defined by explicit conditions, see \cite{Fox21}. 
 
In this article, we consider certain compact PEL-type unitary Shimura varieties of signature $(r,s)$, introduced by Kottwitz in $\S 1$ of \cite{Kot92}, which we refer to as Kottwitz varieties. We assume that $p$ is an inert prime (the case that $p$ is a split prime is studied in \cite{Kret11}). Following \cite{Kret11}, we use automorphic forms and representation of $p$-adic groups to study the basic stratum of certain Kottwitz varieties. In the case of good reduction, we express the trace of the Frobenius operator and the Hecke operator of $G(\mathbb A^p)$ on the $\ell$-adic cohomology of $\bar S$ in terms of orbital integrals, and then in terms of elliptic
parts of the stable Arthur–Selberg trace formula for the endoscopic groups (see \cite{kot90}, \cite{kot92p} for PEL-type Shimura varieties, and see \cite{SKZ21} for abelian type Shimura varieties). Such method to study the Frobenius-Hecke trace is called the Langlands-Kottwitz method.   

Now let $S_b$ be the basic stratum of Kottwitz varieties, we use $\iota:B \hookrightarrow \bar S$ to denote the closed subscheme of $\bar S$ (defined over $\mathbb F_q$), such that $B_{\overline{\mathbb F}_q}=S_b$ (See Theorem 3.6 of \cite{RaRi96}). Note that in this case, the Shimura variety $Sh_{K^pK_p}(G,X)$ is compact, then $\bar S$ is proper, and $B \subset \bar S$ is proper. In \cite{Kret11}, the Langlands-Kottwitz method is adapted to study the Frobenius-Hecke trace of the $\ell$-adic cohomology of $S_b$, and we can express the Frobenius-Hecke trace in terms of orbital integrals, and then in terms of traces of automorphic representations. 
In this article, we study the Frobenius-Hecke action on the $\ell$-adic cohomology of $B_{\overline{\mathbb F}_q}$ for certain Kottwitz varieties, and we derive a pointing counting formula for $B_{\mathbb F_{q^\alpha}}$ for $\alpha \in \mathbb Z_{>0}$ in terms of automorphic representations. Our method is totally different from geometric methods (no use of automorphic forms is made in geometric methods), and the advantage of our method is that we can study the basic stratum of Kottwitz varieties of arbitrary signature $(r,s)$ (where the geometric methods are restricted to the case where $min\{r,s\} \leq 2$). Our method is of both global and local nature: On the global side, we apply a truncation of the formula of Kottwitz for the number of points on Shimura varieties over finite fields (introduced in \cite{Kret11}), the Arthur - Selberg trace formula, base change of certain automorphic forms and the twisted trace formula. On the local side, we apply twisted compact traces and representations of general linear groups over $p$-adic fields. The difference between our method and that of \cite{Kret11} is that \cite{Kret11} studies compact traces, and in this article we study twisted compact traces. 

\subsection{Main results}

We now give a precise statement of the main results: Let $D$ be a division algebra over $\mathbb Q $ equipped with an anti-involution $*$. Let $\overline{\mathbb Q} $ be the algebraic closure of $\mathbb Q$ inside $\mathbb C$. We assume $F$ is an imaginary quadratic extension of $\mathbb Q$, and we assume $*$ induces the complex conjugation on $F$. 
	Let $n$ be the positive integer such that $n^2$ is the dimension of $D$ over $F$. Let $G$ be the  $\mathbb Q$-group such that for each commutative $\mathbb Q$-algebra $R$ the set $G(R)$ is defined as follows:
	$$G(R):=\{x \in D \otimes_{\mathbb Q} R: xx^* \in R^{\times}\}.$$
Let $h_0: \mathbb C \rightarrow D_{\mathbb R}$ be an algebra morphism, such that $h_0(z)^* = h_0(\overline z)$ for all $z \in \mathbb C$. We assume that the involution of $D_{\mathbb R}$ defined by $x \mapsto h_0(i)^{-1}xh_0(i)$ ($x\in D_{\mathbb R}$) is positive. We restrict $h_0$ to $\mathbb C^{\times}$ to obtain a morphism $h'$ from Deligne's torus $\mathrm{Res}_{\mathbb C/\mathbb R}\mathbb G_{m,{\mathbb C}}$ to $G_{\mathbb R}$. Let $h$ denote $h'^{-1}$ and let $X$ denote the $G(\mathbb R)$-conjugacy class of $h$. The pair $(G,X)$ is a Shimura datum. We assume $G_{\mathbb Q_p}$ is unramified and let $S_b$ be the basic stratum at $p$. We assume that $p$ is inert in extension $F/\mathbb Q$.

By the classification of unitary groups over real numbers, we have $G_{\mathbb R} \cong GU_{\mathbb C/\mathbb R}(s, n-s)$ for some integer $s$ with $0 \leq  s \leq n $. If $n-s \not = s $, then the reflex field is $E$, and if $n-s = s $, then the reflex field is $\mathbb Q$. The basic stratum $B$ is defined over $\mathbb F_q$ where $q = p^2$ if $n-s \not = s $, and the basic stratum $B$ is defined over $\mathbb F_p$ if $n-s  = s $. Then we have the following theorem: 

\begin{theorem}
    (See Theorem 5.2.1)
\begin{enumerate}
    \item If $n-s \not = s$, then for an arbitrary $\alpha \in \mathbb Z_{> 0}$, we have $$|B(\mathbb F_{q^{\alpha}})| =| \mathrm{ker}^1(\mathbb Q, G)| \sum_{\pi} N_{\pi} \mathrm{Tr}_{\theta}(C_{\theta c} \phi_{n \alpha s}, \pi)$$ 
$$=   | \mathrm{ker}^1(\mathbb Q, G)| \sum_{\pi}  \sum\limits_{Q \in \mathcal P^{\theta}, Q = LU}\sum\limits_{w \in G_{P,Q}^{\theta}}N_{\pi} \varepsilon_{Q,\theta}  \varepsilon_{Q, w}  \mathcal S(\hat{\chi}^{\theta}_U  \phi_{n \alpha s}^{(Q)}) (q^{t_{\pi,w(1)}}, \ldots, q^{t_{\pi,w(n)}}). $$ Where $\pi $'s are representations of $GL_n(F_p)$ of the form $\mathrm{Ind}_P^G(\otimes_i \mathrm{St}_{GL_{n_i}}(\varepsilon_i))$, such that $\varepsilon_i$ is either the trivial character or the unramified quadratic character (note that $\varepsilon_{Q,w} = \pm 1$ is introduced before Proposition 5.2.1 and $\varepsilon_{Q, \theta} = \pm 1$ is introduced in Proposition 2.2.3), and $$N_{\pi} = \pm \sum\limits_{\substack{\chi \otimes \tau \subset \mathcal A(F^{\times} \times D^{\times}): \\  \tau_p \text{ is of }\theta \text{ type}, \\ \pi\leq \tau_p  }}  \mathrm{Tr}_{\theta}( \varphi_{ep}^{\theta},(\chi \otimes \tau)_{\infty} )  \mathrm{Tr}_{\theta}(\varphi^{p\infty} ,(\chi \otimes \tau )^{p \infty}) M_{\tau_p, \pi}$$ (the partial ordering is defined before Proposition 3.3.2, and $M_{\tau_p, \pi}$ is introduced in Theorem 3.3.1). And the space $\mathcal A(F^{\times} \times D^{\times})$ is the space of automorphic forms on $(F^{\times} \times D^{\times}) (\mathbb A)$. 

\item If $n-s = s$, then for an arbitrary even $\alpha \in \mathbb Z_{> 0}$, we have $$|B(\mathbb F_{p^{\alpha}})| =| \mathrm{ker}^1(\mathbb Q, G)| \sum_{\pi} N_{\pi} \mathrm{Tr}_{\theta}(C_{\theta c} \phi_{n \frac \alpha2 s}, \pi) $$ 
$$   =   | \mathrm{ker}^1(\mathbb Q, G)| \sum_{\pi}  \sum\limits_{Q \in \mathcal P^{\theta}, Q = LU}\sum\limits_{w \in G_{P,Q}^{\theta}}N_{\pi} \varepsilon_{Q,\theta}  \varepsilon_{Q, w}  \mathcal S(\hat{\chi}^{\theta}_U  \phi_{n \frac{\alpha}2 s}^{(Q)}) (q^{t_{\pi,w(1)}}, \ldots, q^{t_{\pi,w(n)}}) $$ with $N_{\pi}$ as defined above, and $\pi $'s are representations of $GL_n(F_p)$ of the form $\mathrm{Ind}_P^G(\otimes_i \mathrm{St}_{GL_{n_i}}(\varepsilon_i))$, such that $\varepsilon_i$ is either the trivial character or the unramified quadratic character. 

\item If $n-s = s$, then for an arbitrary odd $\alpha \in \mathbb Z_{> 0}$, we have $$|B(\mathbb F_{p^{\alpha}})| =| \mathrm{ker}^1(\mathbb Q, G)| \sum_{\pi} N_{\pi} p^{-\frac \alpha2 s(n-s)}\mathrm{Tr}_{\theta}(C_{\theta c} \phi_{n  \alpha s}, \pi) $$ 
$$ = | \mathrm{ker}^1(\mathbb Q, G)| \sum_{\pi}  \sum\limits_{Q \in \mathcal P^{\theta}, Q = LU}\sum\limits_{w \in G_{P,Q}^{\theta}}N_{\pi}  p^{-\frac \alpha2 s(n-s)} \varepsilon_{Q,\theta}  \varepsilon_{Q, w}  \mathcal S(\hat{\chi}^{\theta}_U  \phi_{n \frac{\alpha}2 s}^{(Q)}) (q^{t_{\pi,w(1)}}, \ldots, q^{t_{\pi,w(n)}}) $$ with $N_{\pi}$ as defined above, and $\pi $'s are representations of $GL_n(F_p)$ of the form $\mathrm{Ind}_P^G(\otimes_i \mathrm{St}_{GL_{n_i}}(\varepsilon_i))$, such that $\varepsilon_i$ is either the trivial character or the unramified quadratic character. 
\end{enumerate} 

\end{theorem} 

We make some remarks on the notations appearing in the above theorem: 
\begin{itemize}
\item The traces $\mathrm{Tr}_{\theta}$ are twisted traces which we define in $\S2.2$, where $\theta$ is certain involution of $GL_n(F_p)$ which comes from the Galois action (see Remark 4.1.1).

\item The functions $\phi_{n\alpha s} $ (resp. $C_{\theta c} \phi_{n \alpha s}$) are certain spherical functions on $GL_n(F_p)$ (resp. truncated spherical functions on $GL_n(F_p)$) which will be defined in $\S2.2$ and $\S 2.4$. 

\item The group $Q \in \mathcal P^{\theta}, Q = LU$ ranges over $\theta$-stable standard parabolic subgroups of $GL_n$ (standardised with respect to the Borel subgroup of upper triangular matrices), where $L$ is a Levi subgroup and $U$ is a unipotent subgroup. 

\item The constants $\varepsilon_{Q,\theta},  \varepsilon_{Q, w}$ are equal to $\pm 1$ with $\varepsilon_{Q,\theta}$ as defined in Definition 2.2.3 and $\varepsilon_{Q, w}$ as defined before Proposition 5.2.1.  

\item We make a few more remarks on the constant $$N_{\pi}= \pm \sum\limits_{\substack{\chi \otimes \tau \subset \mathcal A(F^{\times} \times D^{\times}): \\ \chi \otimes \tau \text{ is } \theta \text{-stable,} \\ \tau_p \text{ is of }\theta \text{ type}, \\ \pi\leq \tau_p  }}  \mathrm{Tr}_{\theta}( \varphi_{ep}^{\theta},(\chi \otimes \tau)_{\infty} )  \mathrm{Tr}_{\theta}(\varphi^{p\infty} ,(\chi \otimes \tau )^{p \infty}) M_{\tau_p, \pi},$$ appearing in our theorem: The $\pm 1$ term comes from intertwining operators and certain traces against characters of $\mathbb A_F^{\times}$ (see Proposition 4.3.1). The term $\mathrm{Tr}_{\theta}( \varphi_{ep}^{\theta},(\chi \otimes \tau)_{\infty} )  \mathrm{Tr}_{\theta}(\varphi^{p\infty} ,(\chi \otimes \tau )^{p \infty}) $ comes from twisted traces away from $p$, the constant $M_{\tau_p, \pi}$ and the partial ordering come from ring of Zelevinsky (see $\S 3.2$ and $\S 3.3$). 
We have $$\mathrm{Tr}_{\theta}(\varphi^{\theta}_{ep},(\chi \otimes \tau)_{\infty} ) = \sum_{i  =1}^{\infty} \mathrm Tr(\theta, H^i(\mathfrak h, K_{\infty}; (\chi \otimes \tau)_{\infty}))$$ where $K_{\infty}$ is a maximal compact subgroup of $H(\mathbb R)$ and $\mathfrak h$ is the Lie algebra of $H(\mathbb R)$. 
    Representations of $GL_n(F_p)$ of $\theta$-type (Definition 3.1.1) are certain $\theta$-stable irreducible admissible representations of $GL_n(F_p)$, which are local components of discrete automorphic representations and have an Iwahori fixed vector. 

\item The sets $G_{P,Q}^{\theta}$ is a certain subset of the symmetric group $S_n$ (see $\S$5.2) with a combinatorial interpretation using Young tableau (see $\S$5.3). 

\item The functions $\mathcal S(\hat{\chi}^{\theta}_U  \phi_{n \alpha s}^{(Q)})$ are certain spherical functions on $GL_n(F_p)$ (see $\S$2.6 for an explicit description of $\mathcal S(\hat{\chi}^{\theta}_U  \phi_{n \alpha s}^{(Q)})$).
\end{itemize}

\subsection{Proof strategies}

We now comment on the strategy of proof of Theorem 1.2.1. First, we reduce the computation of the traces of the Frobenius operator and the Hecke operator on the $\ell$-adic cohomology of $B_{\bar {\mathbb F}_q}$ to the computation of certain traces of automorphic representations of $G(\mathbb A)$. 

Let $f^{\infty p} \in \mathcal H(G(\mathbb A^{\infty p})//K^p)$ be a $K^p$-bi-invariant Hecke operator away from $p$ and $\infty$, and Let $\zeta$ be an irreducible algebraic representation over $\overline{\mathbb Q}$ of $G_{\overline{\mathbb Q}}$, associated with an $\ell$-adic local system on $\bar S$ (where $\ell \not = p$ is a prime number), denoted by $\mathcal L$. The formula of Kottwitz for Shimura varieties of PEL-type gives an expression for the Hecke-Frobenius traces on $\ell$-adic cohomology of these varieties at primes of good reduction (see $\S 19$ of \cite{kot92p}). In \cite{Kret11} Kret truncates the formula of Kottwitz and adapts the argument in \cite{Kot92} (see Proposition 9 in \cite{Kret11}), then we have the following Proposition: 

\begin{proposition}
(See Proposition 4.2.1)
If $n-s \not = s$, the reflex field is $F$. Let $q = |\mathcal O_F/(p)| = p^2$, then we have $$\mathrm{Tr}(f^{\infty p}\times \Phi_{\mathfrak p}^{\alpha}, \sum_{i = 1}^{\infty}(-1)^iH^i_{et}(B_{\bar {\mathbb F}_q}, \iota^*\mathcal L)) =  | \mathrm{ker}^1(\mathbb Q, G)| \mathrm{Tr}(C_cf_{n\alpha s} f^{\infty p} f_{\infty} , \mathcal A(G))$$ for $\alpha$ sufficiently large, where $\mathcal A(G)$ is the space of automorphic forms on $G(\mathbb A)$. 

If $n-s = s$, the reflex field is $\mathbb Q$. We have $$\mathrm{Tr}(f^{\infty p}\times \Phi_{\mathfrak p}^{\alpha}, \sum_{i = 1}^{\infty}(-1)^iH^i_{et}(B_{\bar {\mathbb F}_p}, \iota^*\mathcal L)) =  | \mathrm{ker}^1(\mathbb Q, G) |\mathrm{Tr}(C_cf_{n \frac \alpha2 s} f^{\infty p} f_{\infty}, \mathcal A(G))$$ for $\alpha$ even and sufficiently large, and $$\mathrm{Tr}(f^{\infty p}\times \Phi_{\mathfrak p}^{\alpha}, \sum_{i = 1}^{\infty}(-1)^iH^i_{et}(B_{\bar {\mathbb F}_p}, \iota^*\mathcal L)) = | \mathrm{ker}^1(\mathbb Q, G) |\mathrm{Tr}(C_c(p^{-\frac \alpha2 s(n-s)}f_{n\alpha  s} )f^{\infty p} f_{\infty}, \mathcal A(G))$$ for $\alpha$ odd and sufficiently large.
\end{proposition} 

We make some remarks on the notations appearing in the above theorem: 
\begin{itemize}

\item At $p$: The functions $f_{n \alpha s}$ (resp. $C_cf_{n \alpha s}$) are certain spherical functions on $G(\mathbb Q_p)$ (resp. truncated spherical functions on $G(\mathbb Q_p)$) which will be defined in $\S2.2$ and $\S 2.4$. 

\item Away from $p$ and $\infty$: The function $f^{\infty p} \in \mathcal H(G(\mathbb A_p^{\infty})//K^p)$ is an arbitrary $K^p$-bi-invariant Hecke operator.  

\item At $\infty$: The function $f_{\infty}$ is an Euler-Poincar{\'e} function associated to $\zeta$ (as introduced in \cite{Clo85}). 

\end{itemize}

\begin{remark}
We assume that $n-s = s $ and we take $k$ to be a sufficiently large odd integer. By applying Proposition 1.3.1, we obtain $|B(\mathbb F_{p^{2k}})| = p^{\frac k2 s(n-s)}|B(\mathbb F_{p^{k}})|$, from which we obtain results relating the dimension of $B$ and the rational points of $B_{\mathbb F_{p^k}}$ (as well  as results relating the dimension of $B$ and the number of irreducible components), see Corollary 4.2.1. 
\end{remark}

The main difficulty in computing $\mathrm{Tr}(C_cf_{n\alpha s} f^{\infty p} f_{\infty} , \mathcal A(G))$ is that $G(\mathbb Q_p)$ is isomorphic to the quasi-split unitary group over $\mathbb Q_p$, denoted by $GU_{F_p/\mathbb Q_p}(n)(\mathbb Q_p)$ (which we will define in $\S2.3$), and that unitary representations of $GU_{F_p/\mathbb Q_p}(n)(\mathbb Q_p)$ is not well understood. To address this difficulty, we use base change to reduce the computation of $\mathrm{Tr}(C_cf_{n\alpha s} f^{\infty p} f_{\infty} , \mathcal A(G))$ to the computation of certain twisted traces of automorphic representations of $(F^{\times} \times D^{\times})(\mathbb A)$ (note that $(F^{\times} \times D^{\times})(\mathbb Q_p) \cong F_p^{\times} \times GL_n(F_p)$, and unitary representations of general linear groups over $p$-adic fields is understood much better, see \cite{Tad86}). 

We define $H $ to be $\mathrm{Res}_{\mathbb Q}^F(G) \cong \mathrm{Res}_{\mathbb Q}^F (F^{\times} \times D^{\times})$, we take $f  = f^{\infty} \otimes f_{ep} \in \mathcal H(G(\mathbb A))$ and $\varphi  = \varphi^{\infty} \otimes \varphi^{\theta}_{ep} \in \mathcal H(G(\mathbb A))$ such that $ f$ and $\varphi$ are associated in the sense of base change (i.e., the functions $f_v$ and $\varphi_v$ are associated for all places $v$, as defined in $\S 3.1 $ of \cite{Lab99}), where $f_{ep}$ is an Euler-Poincar{\'e} function associated to the trivial representation (as introduced \cite{Clo85}) and $\varphi_{ep}^{\theta}$ is a Lefschetz fuction for $\theta$ (as defined in p120 of \cite{Lab99}). By modifying the proof of Theorem A3.1 in \cite{Clo99} (see the proof of Proposition 4.2.2), we have the following proposition:

\begin{proposition}
(See Proposition 4.2.2) We have $$\mathrm{Tr} (r (f), \mathcal A(G)) = \mathrm{Tr}(R(\varphi)A_\theta  , \mathcal A( F^{\times} \times D^{\times} ))$$ where $r, R$ are right regular representations, and $A_{\theta}$ is defined by $\alpha \mapsto \alpha \circ \theta$ for $\alpha \in \mathcal A  (F^{\times} \times D^{\times} ))$, and $\theta$ is the involution arising from the Galois action (see Remark 4.1.1).
\end{proposition}

In the following, we assume that $n-s \not = s$, the argument for the case $n-s = s$ is almost identical. 
We fix the following datum:

\begin{itemize}

\item Local system: We take $\zeta$ to be the trivial representation $G_{\overline{\mathbb Q}}$, then the associated local system $\mathcal L$ is $\overline{\mathbb Q}_{\ell}$. 
\item At $p$:  We take $C_{\theta c}  \varphi_{n \alpha s} \in \mathcal H(H(\mathbb Q_{p}))$ (with $\varphi_{n \alpha s}$ as defined in $\S$2.4) and $ C_c f_{n \alpha s }\in \mathcal H(G(\mathbb Q_{p}))$. Then $C_{\theta c}  \varphi_{n \alpha s}$ and $ C_c f_{n \alpha s }$ are associated in the sense of base change as defined in $\S 3.1 $ of \cite{Lab99}(see the proof of Proposition 10 of \cite{Kret11}). 
\item Away from $p$ and infinite: We take $1 _{K^{p}} \in \mathcal H(G(\mathbb A^{p \infty}))$ to be the characteristic function of $K^p$, then there exists $ \varphi^{p\infty} \in \mathcal H(H(\mathbb A^{p \infty}))$ associated to $1 _{K^{p \infty}}$ if $K^p$ is small enough (for unramified places, see Proposition 3.3.2 in \cite{Lab99}, for ramified places, see the proof of Theorem A5.2 in \cite{Clo99})
\item At $\infty$: We take $\varphi_{ep}^{\theta} \in \mathcal H(H(\mathbb A_{ \infty}))$ to be a Lefschetz function for $ \theta$ (as defined in p120 of \cite{Clo99}), and $ f_{ep} \in \mathcal H(G(\mathbb A_ \infty)$ to be an Euler-Poincar{\'e} function associated to the trivial representation, after multiplying $\varphi^{\theta}_{ep}$ by some $c \in \mathbb R_{>0}$, we may choose $\varphi^{\theta}_{ep}$ and $f_{ep}$, such that $\varphi^{\theta}_{ep}$ and $f_{ep}$ are associated (see Corollary A1.2 of \cite{Clo99}). 
\end{itemize}

Applying Proposition 1.3.1 and Proposition 1.3.2, we have (Corollary 4.2.2 of $\S 4.2$):

\begin{corollary}
If $n-s \not = s$, we define $q = |\mathcal O_F/(p)| = p^2$, then for an arbitrary $\alpha \in \mathbb Z_{>0}$ sufficiently large, we have $$|B(\mathbb F_{q^{\alpha}})| = \mathrm{Tr}(1 _{K^{p}}\times \Phi_{\mathfrak p}^{\alpha}, \sum_{i = 1}^{\infty}(-1)^iH^i_{et}(B_{\bar {\mathbb F}_q}, \iota^*\mathcal L)) $$
$$=  | \mathrm{ker}^1(\mathbb Q, G)| \mathrm{Tr}(C_cf_{n\alpha s} f^{\infty p} f_{\infty} , \mathcal A(G)) = | \mathrm{ker}^1(\mathbb Q, G)|\mathrm{Tr}_{\theta}(C_{\theta c}  \varphi_{n \alpha s}\varphi^{p\infty} \varphi_{ep}^{\theta}, \mathcal A( F^{\times} \times D^{\times} )).$$
\end{corollary}
We then reduce the computation of the twisted trace $$\mathrm{Tr}_{\theta}(C_{\theta c}  \varphi_{n \alpha s}\varphi^{p\infty} \varphi_{ep}^{\theta}, \mathcal A( F^{\times} \times D^{\times} ))$$ to the computation of the twisted trace $Tr_{\theta}(C_{\theta c}\phi_{n \alpha s}, \rho)$, where $\rho$ is a representation of $GL_n(F_p)$ of $\theta$-type (as defined in $\S 3.1$). In particular, we have (see the proof of Proposition 4.3.1):

\begin{proposition}
The twisted trace $$\mathrm{Tr}_{\theta}(C_{\theta c}  \varphi_{n \alpha s}\varphi^{p\infty} \varphi_{ep}^{\theta}, \mathcal A( F^{\times} \times D^{\times} ))= \sum_{\pi} M_{\pi} \mathrm{Tr}_{\theta}(C_{\theta c} \phi_{n \alpha s}, \rho) $$ where $\rho $'s are representations of $GL_n(F_p)$ of $\theta$-type and $$N_{\pi} = \pm \sum\limits_{\substack{\chi \otimes \tau \subset \mathcal A(F^{\times} \times D^{\times}): \\ \chi \otimes \tau \text{ is } \theta \text{-stable,} \\  \tau_p \text{ is of }\theta \text{ type}, \\ \pi\leq \tau_p  }}  \mathrm{Tr}_{\theta}( \varphi_{ep}^{\theta},(\chi \otimes \tau)_{\infty} )  \mathrm{Tr}_{\theta}(\varphi^{p\infty} ,(\chi \otimes \tau )^{p \infty}) $$ with the partial ordering as defined before Proposition 3.3.2, and $$\mathrm{Tr}_{\theta}(\varphi^{\theta}_{ep},(\chi \otimes \tau)_{\infty} ) = \sum_{i  =1}^{\infty} \mathrm Tr(\theta, H^i(\mathfrak h, K_{\infty}; (\chi \otimes \tau)_{\infty})).$$ 
\end{proposition}

To compute the twisted trace $Tr_{\theta}(C_{\theta c}\phi_{n \alpha s}, \rho)$ (for which we use a twisted version of Clozel's formula, see Proposition 2.2.3), where $\rho$ is a representation of $GL_n(F_p)$ of $\theta$-type (as defined in $\S 3.1$), we need to understand the Jacquet modules of $\rho$ (which is described in \cite{Kret12}). Additionally, we need to determine how the intertwining operator acts on Jacquet modules of $\rho$, and this is the main difficulty for us.   

To address the difficulty, we use Zelevinsky's segment construction (especially Proposition 3.3.2 in $\S 3.3$) and Proposition 2.7.1 (which concerns twisted traces of $C_{\theta c}\phi_{n \alpha s}$ against certain induced representation). These reduce the 
computation of the twisted trace $Tr_{\theta}(C_{\theta c}\phi_{n \alpha s}, \rho)$, where $\rho$ is a representation of $GL_n(F_p)$ of $\theta$-type, to the computation of the twisted traces $\mathrm{Tr}_{\theta}(C_{\theta c}f_{n\alpha s} , \rho' )$, where $\rho'$ is of the form $\mathrm{Ind}_P^G(\otimes_i \mathrm{St}_{GL_{n_i}}(\varepsilon_i))$, with $\varepsilon_i$ being either the trivial character or the unramified quadratic character. Notably, we have a clear understanding of how the intertwining operator acts on Jacquet modules of $\rho'$ (see \cite{Nguyen04}). 
In particular, we establish Proposition 1.3.4 and Theorem 1.3.1. 

\begin{proposition} (See Proposition 3.3.2)
We have $$\mathrm{Tr}_{\theta}(C_{\theta c}\phi_{n\alpha s} ,  \rho ) = \mathrm{Tr}_{\theta} (C_{\theta c}\phi_{n\alpha s} , \rho_{\mathrm{St}}) + \sum\limits_{\rho' \in \mathcal C_{\rho}^{\theta},\rho'<\rho} N_{\rho, \rho'}  \mathrm{Tr}_{\theta}(C_{\theta c}\phi_{n\alpha s} , \rho') $$ where 
\begin{itemize}
\item $ \rho_{\mathrm{St}} $ is of the form $\mathrm{Ind}_P^G(\otimes_i \mathrm{St}_{GL_{n_i}}(\varepsilon_i))$, where $\varepsilon_i$ is either the trivial character or the unramified quadratic character (note that $ \rho_{\mathrm{St}} $ is a representation of $\theta$-type).

 \item $N_{\rho,\rho'} \in \mathbb Z $ and it depends on $\rho$, $\rho'$ and an intertwining operator which we introduced later in the proof (which is constructed from $\rho$). 

 \item $C^{\theta}_{\rho}$ is the set consisting of $\theta$-stable representation $\rho'$ of $GL_n(F_p)$ such that $\rho' < \rho$, where $<$ is a certain partial ordering in Ring of Zelevinsky as defined at the end of $\S 3.2$ (which is introduced in \cite{Zel80}).

\end{itemize}
\end{proposition}

Note that the minimal element in $C_{\theta}^{\rho}$ is of the form $\mathrm{Ind}_P^G(\otimes_i \mathrm{St}_{GL_{n_i}}(\varepsilon_i))$, where $\varepsilon_i$ is either the trivial character or the unramified quadratic character, keep applying Proposition 1.3.4, we have:

\begin{theorem} (See Theorem 3.3.1)
  The twisted trace $$\mathrm{Tr}_{\theta}(C_{\theta c}\phi_{n\alpha s} ,  \rho )  = \sum \limits_{\rho' \in \mathcal C_{\rho}^{\theta}}M_{\pi,\rho'}  \mathrm{Tr}_{\theta}(C_{\theta c}f_{n\alpha s} , \rho' )$$ for some $M_{\rho, \rho'} \in \mathbb Z $, where $\rho'$ is of the form $\mathrm{Ind}_P^G(\otimes_i \mathrm{St}_{GL_{n_i}}(\varepsilon_i))$, such that $\varepsilon_i$ is either the trivial character or the unramified quadratic character.  
\end{theorem}

Now let $ \pi$ be a representation of $GL_n(F_p)$ of the form $\mathrm{Ind}_P^G(\otimes_i \mathrm{St}_{GL_{n_i}}(\varepsilon_i))$, such that $\varepsilon_i$ is either the trivial character or the unramified quadratic character. Applying a twisted version of Clozel' formula (see Proposition 2.2.3, which is introduced in p263 of \cite{Clo90}), we have 
\begin{equation}
\mathrm{Tr}_{\theta}(C_{\theta c} \phi_{n \alpha s}, \pi) = \sum\limits_{Q \in \mathcal P^{\theta}, Q = LU}\varepsilon_{Q,\theta} \mathrm{Tr}(\hat{\chi}^{\theta}_U  \phi_{n \alpha s}^{(Q)}, \pi_U(\delta_{Q}^{-\frac 12})A_{\theta}). 
\end{equation}
where $\pi \mapsto \pi_U$ (resp. $\pi \mapsto J_U(\pi)$) is the Jacquet module (resp. normalized Jacquet module) as defined at the beginning of $\S$2, and $\varepsilon_{Q,\theta}, \hat{\chi}^{\theta}_U  \phi_{n \alpha s}^{(Q)}$ as introduced after Theorem 1.2.1. By choosing an explicit intertwining operator on $\pi$ (denoted by $\theta$, as introduced in $\S5.1$), we have an explicit description of the action of $A_{\theta}$ on $ \pi_U$, while $ \pi_U$ is computed using Bernstein-Zelevinsky filtration (see VI1.5.2 of \cite{Renard10}). More precisely, we have

\begin{proposition} (See Proposition 5.1.3)
In $\mathcal R$, we have $$ J_	U(\mathrm{Ind}_P^G(\otimes_i \mathrm{St}_{GL_{n_i}}(\varepsilon_i)))= \bigoplus_{w \in G_{P,Q}  \subset S_n} \mathrm{Ind}_{L \cap w^{-1}P}^L (w \circ J_{w Q \cap M} (\otimes_i \mathrm{St}_{GL_{n_i}}(\varepsilon_i))$$ where $G_{P,Q}  \subset S_n$ is the set of minimal length representatives of $S_P\setminus S_n/ S_Q$

The induced intertwining operator $A_{\theta}$ acts on $$\bigoplus\limits_{w \in G_{P,Q}  \subset S_n} \mathrm{Ind}_{L \cap w^{-1}P}^L (w \circ  J_{w Q \cap M}(\otimes_i \mathrm{St}_{GL_{n_i}}(\varepsilon_i)))$$ as follows: the intertwining operator $A_{\theta}$ maps $$ \mathrm{Ind}_{L \cap w^{-1}P}^L (w \circ  J_{w Q \cap M}(\otimes_i \mathrm{St}_{GL_{n_i}}(\varepsilon_i)))$$ onto $$\mathrm{Ind}_{L \cap (w^{\theta})^{-1}P}^L (w^{\theta} \circ  J_{w^{\theta} Q \cap M}(\otimes_i \mathrm{St}_{GL_{n_i}}(\varepsilon_i))) $$ where $w^{\theta}$ is the minimal length representative in $S_P \theta (w) S_Q$, and $\theta(w)$ is defined to be $A_n^{-1} w A_n$ (recall that $A_n$ is the anti-diagonal matrix with entries $1$).
\end{proposition}

Note that $\hat{\chi}^{\theta}_U  \phi_{n \alpha s}^{(Q)} $ is an unramified and $\mathrm{Ind}_{L \cap w^{-1}P}^L (w \circ J_{w Q \cap M} (\otimes_i \mathrm{St}_{GL_{n_i}}(\varepsilon_i))$ is ramified unless $w Q \cap M$ is a Borel subgroup. By applying (1) and Proposition 1.3.5, we have 
\begin{equation}
\begin{split}
\mathrm{Tr}_{\theta}(\hat{\chi}^{\theta}_U  \phi_{n \alpha s}^{(Q)}, \pi_U(\delta_{Q}^{-\frac 12})A_{\theta})  = \mathrm{Tr}(\hat{\chi}^{\theta}_U  \phi_{n \alpha s}^{(Q)},  \bigoplus_{w \in G_{P,Q}  \subset S_n} \mathrm{Ind}_{L \cap w^{-1}P}^L (w \circ  J_{w Q \cap M}(\otimes_i \mathrm{St}_{GL_{n_i}}(\varepsilon_i))) A_{\theta}) 
\\=  \sum_{w \in G^{\theta}_{P,Q}  \subset S_n} \mathrm{Tr}_{\theta}(\hat{\chi}^{\theta}_U  \phi_{n \alpha s}^{(Q)},   \mathrm{Ind}_{L \cap w^{-1}P}^L (w \circ  J_{w Q \cap M}(\otimes_i \mathrm{St}_{GL_{n_i}}(\varepsilon_i)))
\end{split}
\end{equation} where $G^{\theta}_{P,Q}  \subset S_n$ is defined to be $\{w \in G_{P,Q}  \subset S_n: w^{\theta} = w, w S_{Q} \cap S_P = \{id_{S_n}\}\}$. Let $(q^{t_{\pi,1}}, \ldots, q^{t_{\pi,n}})$ denote the Hecke matrix (see Definition 2.7.1) of $J_{N_0}(\otimes_i \mathrm{St}_{GL_{n_i}}(\varepsilon_i)) = \otimes_i (\delta_{B_{n_i}}^{\frac 12}(\varepsilon_i))$. For $w \in G^{\theta}_{P,Q}$, we have: 
\begin{equation}
\begin{split}
\mathrm{Tr}_{\theta}(\hat{\chi}^{\theta}_U  \phi_{n \alpha s}^{(Q)},   \mathrm{Ind}_{L \cap w^{-1}P}^L (w \circ  J_{w Q \cap M}(\otimes_i \mathrm{St}_{GL_{n_i}}(\varepsilon_i)))) =   \mathrm{Tr}_{\theta}(\hat{\chi}^{\theta}_U  \phi_{n \alpha s}^{(Q)},   \mathrm{Ind}_{M_0}^L (w \circ  (\otimes_i \delta_{B_{n_i}}^{\frac 12}))
\\=  \mathrm{Tr}_{\theta}(\hat{\chi}^{\theta}_U  \phi_{n \alpha s}^{(Q)},   w \circ  (\otimes_i \delta_{B_{n_i}}^{\frac 12}(\varepsilon_i))) = \varepsilon_{Q, w} \mathcal S(\hat{\chi}^{\theta}_U  \phi_{n \alpha s}^{(Q)}) (q^{t_{\pi,w(1)}}, \ldots, q^{t_{\pi,w(n)}}) 
\end{split}\end{equation} where $ \varepsilon_{Q, w} = \pm 1$ depends on $Q$ and $w$ (See Proposition 2.2.2), and we have the second equality because of a twisted version of van Dijk's formula (see Proposition 2.2.4). 

Finally, by applying Corollary 1.3.1, Proposition 1.3.3, Theorem 1.3.1 and (1), (2), (3), we obtain Theorem 1.2.1.

\subsection{Structure of the paper}

In $\S$2 we carry out the necessary local computations at $p$: In $\S2.1$, we introduce compact elements and compact traces. In $\S2.2$ we study truncated twisted traces. In $\S2.3$, we introduce quasi-split unitary groups associated to a quadratic field extension of $p$-adic fields, and we give explicit descriptions of unramified Hecke algebras associated to them. Then in $\S2.4$-$\S2.6$, we introduce and study the Kottwitz functions $f_{n\alpha s}, \phi_{n\alpha s}$, $\varphi_{n \alpha s}$ and certain truncated Kottwitz functions $C_{\theta c}(\phi_{n \alpha s}^{(P)}), \hat{\chi}^{\theta}_{N} \phi_{n \alpha s} ^{(P)} $, where $P$ is a $\theta$-stable parabolic subgroup of $GL_n$. In $\S2.7$, we study twisted traces of $C_{\theta c}\phi_{n \alpha s}$ against certain induced representations, and the proposition we obtain (Proposition 2.7.1) will be crucial in reducing the computation of twisted traces of $C_{\theta c}\phi_{n \alpha s}$ against representations of $\theta$-type to the computation of twisted traces of $C_{\theta c}\phi_{n \alpha s}$ against induced representations of Steinberg representations (See Theorem 3.3.1). 

 In $\S$3 we study twisted traces of $C_{\theta c}\phi_{n \alpha s}$ against representations of $\theta$-type: In $\S3.1$, we introduce semi-stable rigid representations, which are local components of discrete automorphic representations of general linear groups, and representations of $\theta$-type, which are certain $\theta$-stable semi-stable rigid representations. In $\S3.2$ we define ring of Zelevinsky and Zelevinsky's segments construction (as introduced in \cite{Zel80}), and we collect some important facts about them. In $\S3.3$, we reduce the computation of twisted traces of $C_{\theta c}\phi_{n \alpha s}$ against representations of $\theta$-type to the computation of twisted traces of $C_{\theta c}\phi_{n \alpha s}$ against induced representations of Steinberg representations (See Theorem 3.3.1). Note that Theorem 3.3.1 is crucial to establish Proposition 4.3.1.
 
 In $\S$4, we use the Langlands-Kottwitz method and base change to study the Frobenius-Hecke trace of the $\ell$-adic cohomology of the basic stratum: We introduce Kottwitz varieties and functions of Kottwitz in $\S 4.1$. Then in $\S4.2$, we express the Frobenius-Hecke trace of the $\ell$-adic cohomology of the basic stratum in terms of traces of automorphic representations of certain unitary group associated to division algebra (see Proposition 4.2.1), and further we express the number of points of the basic stratum in terms of twisted traces of automorphic representations of general linear groups (see Corollary 4.2.2), and then in terms of twisted traces of $C_{\theta c}\phi_{n \alpha s}$ against certain induced representations of Steinberg representations of general linear groups (see Proposition 4.3.1). 
 
In $\S$5, we compute twisted traces of $C_{\theta c}\phi_{n \alpha s}$ against certain induced representations of Steinberg representations (of general linear groups): In $\S 5.1$, we provide explicit descriptions of intertwining operators on certain induced representations of Steinberg representations, and we provide explicit descriptions of the induced intertwining operators on their Jacquet modules (see Proposition 5.1.3). Then in $\S 5.2$, we compute twisted traces of $C_{\theta c}\phi_{n \alpha s}$ against certain induced representations of Steinberg representations (see Proposition 5.2.1), using a twisted version of Clozel's formula (see Proposition 2.2.3), Proposition 5.1.3, and a twisted version of van Dijk's formula (see Proposition 2.2.4). Then we obtain a point-counting formula for the basic stratum of Kottwitz varieties (see Theorem 5.2.1). In $\S 5.3$, we give a combinatorial interpretation of certain subsets (which appear in Theorem 5.2.1) of symmetric groups in terms of Young tableau, which allows us to develop a a more efficient algorithm compute point-counting formulas for the basic stratum of Kottwitz varieties.

\section*{Acknowledgements}
	This paper is a part of the author’s PhD research at the
University of Amsterdam, under the supervision of Arno Kret. I thank my supervisor Arno Kret for all his help, for suggesting the approach of using a truncation of the formula of Kottwitz and the twisted trace formula to study the basic stratum of Kottwitz varieties, and for his valuable comments on the preliminary drafts of the paper, which greatly enhance the writing of this paper. I also thank Dhruva Kelkar, Kaan Bilgin and Reinier Sorgdrager for helpful discussions.

\section{Local computations}

In this section we compute the twisted compact traces of the functions of Kottwitz (which is introduced in $\S$2.1 in \cite{Kot84}, for which we also give the definition in $\S4.1$) against certain representations of the general linear group which occur in the pointing counting of the basic stratum of some compact unitary Shimura varieties.

Given an algebraic group $G$ over a field $F$, we use $X^*(G)$ (resp. $X_*(G)$) to denote the set of rational characters (resp. the set of rational cocharacters). We regard $X_*(G) $ as $\mathrm{Hom}_{\mathbb Z}(X^*(G), \mathbb Z)$ via pairing $X_*(G) \times X^*(G) \rightarrow Hom(\mathbb G_{m}, \mathbb G_{m}) \cong \mathbb Z$ defined by $(\alpha, \beta) \mapsto \beta \circ \alpha$. 

Now let $G$ be a quasi-split reductive group over a $p$-adic field $F$ with a chosen Borel subgroup $P_0 $ and a chosen Levi decomposition $P_0 = M_0 N_0$. If $P = MN \subset G$ is a standard parabolic subgroup, then we define the modular character $\delta_{G,P} (m)$ to be $|\mathrm{det}(m,\mathfrak n)|$, where $m \in M(F)$ and $\mathfrak n$ is the Lie algebra of $N$ (we also use $\delta_{P}$ to denote $\delta_{G,P}$ if there is no ambiguity). The induction $\mathrm{Ind}^G_P$ is normalized parabolic induction (which sends unitarty representations to unitary representations, as defined in $\S$VI.1.2 in \cite{Renard10}, note that $\delta_{P} (m)$ is defined to be $|\mathrm{det}(m,\mathfrak n)|^{-1}$ in \cite{Renard10}). The Jacquet module $\pi_N$ of a smooth representation $\pi$ is not normalized, it is the space of coinvariants for the unipotent subgroup $N \subset G$. We use $J_N(\pi)$ to denote the normalized Jaquet module, which is $\pi_N(\delta_P^{\frac 12})$.

\subsection{Compact traces}

	 Let $F$ be a $p$-adic field, let $G$ be an unramified smooth reductive group $G$, and we choose a flat model $\underline G$ over $\mathcal O_F$ and write $K = \underline G(\mathcal O_F)$. Let $\pi_F$ be a normalizer of $\mathcal O_F$. We normalize the $p$-adic norm on $F$ so that $| p | = p^{-Card(\mathcal O_F/(\pi_F))}$ (the convention as in pp.491-492 \cite{Kret11}) and we normalize the Haar measure on $G(F)$ so that $\mu(K) = 1$. We choose a minimal parabolic subgroup $P_0$ of $ G$ and we standardize the parabolic subgroups of $G$ with respect to $P_0$. Let $\mathcal P$ be the set of standard parabolic subgroups containing $P_0$. 
	 
	 We have a Borel subgroup $\underline{P_0} \subset \underline{G}$ such that $\underline{P_0}(F)= P_0(F)$. Let $I \subset \underline{G}(\mathcal O_F)$ be the group of elements $g \in \underline{G}(\mathcal O_F)$ that reduce to an element of the group $\underline {P_0}(\mathcal O_F /\pi_F )$ modulo $\pi_F$. The group $I$ is called the standard Iwahori subgroup of $G(F)$. A smooth representation $\pi$ of G is called semi-stable if it has a non-zero invariant vector under the subgroup $I$ of $G(F)$.

	\begin{definition}
	Let $g $ be an element of $G(F)$, we use $\bar g$ to denote the image of $g$ in $G(\bar F)$ and we use $g_{sm}$ to denote the semisimple part of $\bar g$.
	Then $g$ is called compact if for some (any) maximal torus $T$ in $G(\bar F) $ containing $g_{sm}$, the absolute value $|\alpha(g_{sm})|$ is equal to $1$ for all roots $\alpha$ of $T$ in $Lie(G(\bar F))$. 
	\end{definition}

 \begin{remark}
     If $G = GL_n$ ($n \in \mathbb Z_{>0}$), then $g \in G(F)$ is compact if and only if the norms of all eigenvalues are equal.
 \end{remark}

	Let $C_c^{G(F)}$ be the characteristic function of $G(F)$ of the set of compact elements $G_c \subset G$ (We also use $C_c$ to denote $C_c^{G(F)}$ if there were no ambiguity). The subset $G_c \subset G$ is open, closed and stable under stable conjugation. 
	
	\begin{remark}
	 Let $M$ be a Levi subgroup of $G$ and let $g$ be an element of $M(F) \subset G(F)$. Note that the condition that $g$ is compact for the group $M(F)$ is \textbf{not} equivalent to the condition that $g$ is compact for the group $G(F)$.  
\end{remark}
	
	Let $f$ be a locally constant, compactly supported function on $G(F)$ and let $(\pi,V)$ be a smooth admissible representation of $G(F)$. For $v \in V$, we define $\int_{G(F)} f(g) \pi(g) v dg $ as follows: We define $f_V: G(F) \rightarrow V$ by $g \mapsto f(g) \pi(g) v dg$, the function $f_V$ is locally constant, let $L \subset G(F)$ be a compact open subgroup such that $f_V$ right $L$-invariant, then we define $\int_{G(F)} f(g) \pi(g) v dg $ to be $\mu(L)\sum_{g \in G(F)/L} f(g) \pi(g) v  $ (this is independent of the choice of $L$). We suppose that $f $ is $L'$-bi-invariant for some compact open subgroup $L' \subset G(F) $, we define $\mathrm{\mathrm{Tr}}(f, \pi)$ to be $\mathrm{\mathrm{Tr}}(v \mapsto \int_{G(F)} f(g) \pi(g) vdg )|_{W}$ where $W \subset V$ satisfies that $V^{L'} \subset W \subset V$ and that $W$ is finite dimensional (this is independent of the choice of $W$ and $L'$).
 
 The compact trace of $f$ on the representation $\pi$ is defined by $\mathrm{\mathrm{Tr}}(C_c^{G(F)} f,\pi)$ where $C_c^{G(F)}f$ is the point-wise product of $C_c^{G(F)}$ and $f$. We define $\bar f$ to be the conjugation average of $f$ under the maximal compact subgroup $K$ of $G(F)$: More precisely, for all elements $g$ in $G$, the value $\bar f (g)$ equals $\int_K f(kgk^{-1}) dk $ where the Haar measure is normalized so that $K$ has volume $1$ (then $\bar f = f$ if $f $ is spherical with respect to $K$). We use $\mathcal H(G(F))$ to denote the Hecke algebra of locally constant, compactly supported functions on $G(F)$, where the product on $\mathcal H(G(F))$ is the convolution integral defined by $f_1\ast f_2 (h)=\int_{G(F)}f_1(h)f_2(gh^{-1}) dg $ with respect to the Haar measure which gives the group $K$ measure $1$. When $L$ is a compact open subgroup of $G(F)$, we use $\mathcal H(G(F)//L)$ to denote the Hecke algebra of locally constant, compactly supported functions on $G(F)$ which are bi-invariant under the action of $L$. Let $\mathcal H^{\mathrm{unr}}(G(F))$ denote $\mathcal H(G(F)//K)$, elements of $\mathcal H(G(F)//K)$ are called unramified Hecke operators. 

Let $P = MN \in \mathcal P$ and let $f \in \mathcal H(G(F))$, we define the constant terms $f^{( P )} $ by $f^{(P)}(m) = \delta_P^{\frac 12}(m)\int_{N(F)} f(mn) dn$, note that the constant terms $f \mapsto f^{(P)}$ is injective [\cite{Kot80}, $\S$5].  	

  We write $P_0 = TU$ where $T \subset P_0$ is a maximal torus of $G$, let $S \subset  T$ be a maximal split torus and we write $d$ for the dimension of $S$. We use $$\mathcal S_G: \mathcal H(G(F)//K) \rightarrow  \mathcal H(T(F)//T(F) \cap K) \xrightarrow{\sim} \mathcal H(S(F)//S(F) \cap K) \xrightarrow{\sim} \mathbb C[X_1^{\pm 1}, \ldots, X_d^{\pm 1}]$$ to denote the Satake transform, where $\mathcal H(G(F)//K) \rightarrow  \mathcal H(T(F)//T(F) \cap K)$ is defined to be $f \mapsto f^{(P_0)}$, the isomorphism $\mathcal H(T(F)//T(F) \cap K) \xrightarrow{\sim} \mathcal H(S(F)//S(F) \cap K)$ is given by $f  \mapsto f|_{S(F)}$ and $\mathcal H(S(F)//S(F) \cap K) \xrightarrow{\sim} \mathbb C[X_1^{\pm 1}, \ldots, X_d^{\pm 1}]$ as described in 2.2 of \cite{Minguez11}. Note that $\mathcal S_G$ has image in $\mathcal H(T(F)//T(F) \cap K)^{W(G,T)(F)}$ where $W(G,T) = N_G(T)/T$ is the Weyl group and $N_G(T)$ is the normalizer of $T$ in $G$.
 
	\begin{definition}
	Let $f$ be an element of $\mathcal H(G(F))$. And let $(\pi,V) $ be a smooth, admissible representation of $G$. The compact trace of $f$ on the representation $\pi$ (denoted by $\mathrm{\mathrm{Tr}}_c(f,\pi)$) is defined to be $\mathrm{\mathrm{Tr}} (v \mapsto \int_{G(F)} C_c^{G(F)}f(g) \pi (g)v dg)$. 
	\end{definition}
	
We choose a Levi decomposition $P_0 = M_0 N_0$ where $M_0$ is a maximal torus. Let $P = MN$ be a standard parabolic subgroup of $G$ and let $A_P$ be the split centre of $P$ (i.e., the maximal split torus in the centre of $P$),
we write $\varepsilon_P = (-1)^{\mathrm{dim}(A_P/A_G) }$. Let $\Delta \subset X^*(A_{M_0})$ be the set of roots assocaited to $A_{M_0}$. We define $\mathfrak a_P$ to be $X_*(A_P)_{ \mathbb R}$ and define $\mathfrak a^G_P$ to be the quotient of $\mathfrak a_P$ by $\mathfrak a_G$. 

We define $\Delta_P \subset \Delta$ to be the set of roots which act non-trivially on $A_P$. We write $\mathfrak a_0 = \mathfrak a_{P_0}$ and $\mathfrak a_0^G =\mathfrak a^G_{P_0}$. For each root $\alpha$ in $\Delta$, we have a coroot $\alpha^{\vee}$ in $\mathfrak a^G_0$ defined by $\alpha^{\vee}(\beta) = \frac{2(\beta,\alpha)}{(\alpha,\alpha)} $ for $\beta \in \Delta$.

For $\alpha \in \Delta_P \subset \Delta$, we still use $\alpha^{\vee}$ to denote the image of $\alpha^{\vee} \in \alpha_0^G$ in $\mathfrak a^G_P$ via the surjection $\mathfrak a_0^G \twoheadrightarrow \mathfrak a^G_P$. The set $\{\alpha ^{\vee}|_{\mathfrak a^G_P}:  \alpha \in \Delta_P\}$ form a basis of the vector space
$\mathfrak a^G_P$. The set of fundamental roots $\{\omega_{\alpha}^G \in ({\mathfrak a^G_P)}^*| \alpha \in \Delta_P \}$ 
is the basis of $(\mathfrak a^G_P)^*= Hom(\mathfrak a^G_P , \mathbb R)$ dual to the basis $\{\alpha ^{\vee}|_{\mathfrak a^G_P}:  \alpha \in \Delta_P\}$ of coroots. We let $\tau^G_P$ be the characteristic function of the acute Weyl chamber,
$$(\mathfrak a_P^G)^+= \{x \in \mathfrak a_P^G| \forall \alpha\in \Delta_P, \langle\alpha,x\rangle>0\}.$$
Let $\hat{\tau}_P^G$ be the characteristic function of $\mathfrak a^G_P$ of the obtuse Weyl chamber,
$$^+(\mathfrak a_P^G)= \{x \in \mathfrak a_P^G| \forall \alpha\in \Delta_P, \langle\omega_{\alpha}^G , x\rangle>0\}.$$

 The Harish-Chandra map $H_M$ of $M$ is defined to be unique map from $M$ to $\mathfrak a_P = Hom_{\mathbb Z}(X^*(M),\mathbb R) $ such that $q^{- \langle\chi.H_M(m)\rangle}=|\chi(m) |_p$ for all elements $m$ of $M$ and rational characters $\chi$ in $X^*(M)$. Let $\chi_{G,N}$ be the composition $\tau _P^G \circ (\mathfrak a_P \twoheadrightarrow \mathfrak a^G_P ) \circ  H_M$, and let $\hat{\chi}_{G,N}$ to be the composition  $\hat{\tau} _P^G \circ (\mathfrak a_P \twoheadrightarrow \mathfrak a^G_P ) \circ  H_M$ . Note that the functions $\chi_{G,N}$ and $\hat{\chi}_{G,N}$ are locally constant and $K_M$-invariant, where $K_M = \underline G(\mathcal O_F) \cap M(F)$. We also use $\chi_{N}$ or $\hat{\chi}_{N}$ to denote $\chi_{G,N}$ and $\hat{\chi}_{G,N}$ if there is no ambiguity. 

\begin{lemma} 
Let $P = M N$ be a standard parabolic subgroup of $G(F)$. Let $m$ be an element of $M$. Then we have:\begin{itemize}
 \item $\hat \chi_N(m)=1$ if and only if for all roots $\alpha$ in the set $\Delta_P$ we have $|\omega_{\alpha}(m)| < 1$,
 
 \item $ \chi_N(m)=1$ if and only if for all roots $\alpha$ in the set $\Delta_P$ we have $|\alpha(m)| < 1$.
 \end{itemize}
\end{lemma}
  
	For $f\in \mathcal H(G(F))$, we use $f^{(P)} \in H(P(F))$ to denote the constant term of $f$ with respect to $P$, which defined by $f^{(P)}(m) = \delta_P(m)^{\frac 12} \int_{N(F)}f(mn) dn $ for $m \in M(F)$.  Note that $f \mapsto f^{(P)}$ induces an injective morphism from $\mathcal H(G(F)//K)$ to $\mathcal H(M(F)//K \cap M(F))$ (see Theorem 2.4 of \cite{Minguez11}). And we have the following proposition:
	 
	 \begin{proposition}
Let $\pi$ be an admissible representation of $G(F)$ of finite length, then we have $\mathrm{\mathrm{Tr}}(C_cf, \pi) = \sum\limits_{P \in \mathcal P, P = MN}\varepsilon_P \mathrm{\mathrm{Tr}}(\hat{\chi}_N \bar f^{(P)}, \pi_N(\delta_{P}^{-\frac 12}) )$. 
	 \end{proposition}

\begin{proof}
See p259 of \cite{Clo90}. 
\end{proof}

\begin{proposition}
Let $\pi$ be a smooth admissible representation of G such that the trace $\mathrm{\mathrm{Tr}}(C_c^{G(F)} f,\pi)$ does not vanish for some unramified Hecke operator $f \in \mathcal H^{unr}(G(F))$. Then $\pi$ is semi-stable.

\end{proposition}
\begin{proof}
See p494 of \cite{Kret11}.

\end{proof}

\subsection{Twisted traces}

	Let $F$ be a $p$-adic field. Let $G$ be an unramified connected reductive group over $F$ with an automorphism $\theta$ of order 2. Let $\mathcal G$ be a flat model of $G$ over $\mathcal O_F$, and we define $K = \mathcal G(\mathcal O_F)$. We fix a Haar measure $\mu$ on $G(F)$ so that $\mu(K) = 1$. Let $P_0$ be a Borel subgroup, and let $T$ be a maximal split torus of $ G$ contained in $P_0$. We suppose that $\theta$ fixes $( T,B)$ and we standardize parabolic subgroups of $G$ with respect to $P_0$.
	
\begin{definition}
	Let $(\pi,V)$ be a smooth admissible representation of $ G(F)$. We say that $(\pi, V)$ is \textbf{$\theta$-stable} if there exists a non-zero $A_{\theta} \in \mathrm{End}_{\mathbb C}(V)$ such that for all $g \in G(F)$ we have $ \pi(g) \circ A_{\theta} = A_{\theta} \circ  \pi(\theta(g)) $, such $A_{\theta}$ is then unique up to scalar and such $A_{\theta}$ is said to be an intertwining operator of $(\pi,V)$ (Note that $A_{\theta}$ is an isomorphism from $(\pi \circ \theta,V) $ to $(\pi, V)$). 
\end{definition}

\begin{remark}
Note that $ \pi(g) \circ A_{\theta} = A_{\theta} \circ  \pi(\theta(g)) $ implies that $ \pi(\theta(g)) \circ A_{\theta} = A_{\theta} \circ  \pi(\theta(\theta(g)) = A_{\theta} \circ  \pi(g) $ for all $g \in G(F)$, and therefore $A_{\theta}$ is also an isomorphism from $(\pi ,V) $ to $(\pi \circ \theta, V)$. Then $A_{\theta} \circ A_{\theta}:(\pi ,V) \rightarrow (\pi \circ \theta ,V) \rightarrow (\pi ,V) $ is an automorphism of $(\pi ,V)$. When $\pi$ is irreducible, then by Schur's Lemma (see III.1.8 of \cite{Renard10}), we have that $A_{\theta} \circ A_{\theta}$ is multiplication by some $\lambda \in \mathbb C^{\times}$.
We say that $A_{\theta} \in \mathrm{End}_{\mathbb C}(V) $ is normalised if $A_{\theta} \circ A_{\theta} $ is the identity.
\end{remark}

Let $f$ be a locally constant, compactly supported function on $G(F)$. Let $(\pi,V)$ be a $\theta$-stable smooth admissible representation of $G(F)$. For $v \in V$, we define $\int_{G(F)} f(g) \pi(g) v A_{\theta} dg $ as follows: We define $f_{V,A_{\theta}}: G(F) \rightarrow V$ by $g \mapsto f(g) \pi(g) A_{\theta} v dg$ is, 
let $L \subset G(F)$ be a compact open subgroup such that $f_{V, A_{\theta}}$ right $L$-invariant, then we define $\int_{G(F)} f(g) \pi(g) A_{\theta} v  dg $ to be $\mu(L) \sum_{g \in G(F)/L} f(g) \pi(g)  A_{\theta} v  $ (this is independent of the choice of $L$). Let $L' \subset G(F)$ be a compact open subgroup such that $f$ is left $L'$- invariant, then for an arbitrary $ l \in L' $ and for an arbitrary $v\in V$, we have $$\pi(l)\int_{G(F)} f(g) \pi(g) A_{\theta} v  dg  = \pi(l)(\mu(L) \sum_{g \in G(F)/L}  f(g) \pi(g)  A_{\theta} v) $$ 
$$= \mu(L) \sum_{g \in G(F)/L} f(g) \pi(lg)  A_{\theta} v =  \mu(L) \sum_{g \in G(F)/L} f(lg) \pi(lg)  A_{\theta} v = \mu(L) \sum_{g \in G(F)/L} f(g) \pi(g)  A_{\theta} v.$$
Thus for an arbitrary $v\in V$, we have $\int_{G(F)} f(g) \pi(g) A_{\theta} v  dg$ is $L'$-invariant.

\begin{definition}
	Let $(\pi,V)$ be a smooth admissible representation of $ G(F)$ with an intertwining operator $A_{\theta} \in \mathrm{End}_{\mathbb C}(V)$. Let $f \in \mathcal H(G(F))$, we suppose that $f $ is $L'$-bi-invariant for some compact open subgroup $L' \subset G(F) $, we define $\mathrm{\mathrm{Tr}}(f,\pi A_\theta)$ to be $\mathrm{\mathrm{Tr}} (v \mapsto \int_{G(F)} f(g) \pi (g) \circ A_{\theta} v dg) |_W$, where $W \subset V$ satisfies that $V^{L'} \subset W \subset V$ and $W$ is finite dimensional (this is independent of the choice of $W$ and $L'$).
	\end{definition}
	
\begin{remark}
Let $(\pi,V)$ be a smooth admissible representation of $ G(F)$ with an intertwining operator $A_{\theta} \in \mathrm{End}_{\mathbb C}(V)$, then there exists a Harish-Chandra character $\Theta_{\pi 
A_{\theta}}:G(F) \rightarrow \mathbb C $,
such that
$\mathrm{\mathrm{Tr}}(f, \pi A_\theta) = \int_{G(F)} \Theta_{\pi A_{\theta}} (g) f(g) dg $ for $f \in \mathcal H(G(F))$ and $\Theta_{\pi 
A_{\theta}}$ is stable under $\theta$-conjugation
(see $\S 2.3$ in Part I of \cite{BG17}).

\end{remark}

\begin{lemma}
Let $\pi_1, \pi_2$ be two isomorphic admissible representations (not necessarily irreducible). Let $A_{\theta}: (\pi_1 \circ \theta, V_1) \rightarrow (\pi_1, V_1)$ be a normalized intertwining operator and let $I:  (\pi_1 , V_1) \rightarrow (\pi_2, V_2)$ be a  isomorphism. Then for an arbitrary $f \in \mathcal H(G(F)) $, we have $\mathrm{\mathrm{Tr}}(f, \pi_1 A_{\theta}) = \mathrm{\mathrm{Tr}}(f,\pi_2  IA_{\theta} I^{-1})$. 
\end{lemma}

\begin{proof}
For an arbitrary $v \in V_2$ and for an arbitrary $g \in G(F)$, we have $\pi_2(g) I A_{\theta} I^{-1} v= I \pi_1(g) A_{\theta} I^{-1} v$. 
Then we have $$\mathrm{\mathrm{Tr}}(f,\pi_2  IA_{\theta} I^{-1})  = \mathrm{\mathrm{Tr}}(v \mapsto \int_{G(F)} f(g) \pi_2(g)I A_{\theta} I^{-1} v dg) $$
$$= \mathrm{\mathrm{Tr}}(v \mapsto \int_{G(F)} f(g)I \pi_1(g) A_{\theta} I^{-1} v dg) $$
$$ =\mathrm{\mathrm{Tr}}(v \mapsto \int_{G(F)} f(g) \pi_1(g) A_{\theta} v dg)  =\mathrm{\mathrm{Tr}}(f, \pi_1 A_{\theta}) . $$ 
\end{proof}

We fix the notations above, further we let $\pi_1, \pi_2$ be irreducible and let $B_{\theta}:  (\pi_2 \circ \theta, V_2) \rightarrow (\pi_2, V_2)$ be a normalized intertwining operator, then we have the following:

\begin{corollary}
$\mathrm{\mathrm{Tr}}(f, \pi_2 B_{\theta}) = \pm \mathrm{\mathrm{Tr}}(f, \pi_1 A_{\theta})$.
\end{corollary}

\begin{proof}
Firstly, we have $\mathrm{\mathrm{Tr}}(f, \pi_1 A_{\theta}) = \mathrm{\mathrm{Tr}}(f,\pi_2  IA_{\theta} I^{-1})$ from the above lemma. Note that $  IA_{\theta} I^{-1} \circ B_{\theta}$ is an isomorphism of irreducible admissible representation $\pi_2$, by Schur's Lemma (see III.1.8 of \cite{Renard10}), we have that $ IA_{\theta} I^{-1} \circ B_{\theta} $ is multiplication by some $\lambda \in \mathbb C^{\times}$. Also note that both $A_{\theta}$ and $B_{\theta}$ are normalized, then we have $$ \lambda^2 =  IA_{\theta} I^{-1} \circ IA_{\theta} I^{-1} \circ B_{\theta} \circ B_{\theta} = I (A_{\theta} \circ A_{\theta})I^{-1} \circ B_{\theta} \circ B_{\theta} = 1.$$ Thus $\lambda = \pm 1$, and the corollary follows.

\end{proof}

\begin{remark}
The above corollary shows that twisted traces of irreducible representations up to $\pm1$ is well defined up to isomorphism.
\end{remark}

Let $P=MN \subset G$ be a parabolic subgroup and let $ \pi$ be an admissible representation of $M(F)$, then we have that $\mathrm{Ind}_{M}^G(\pi) \circ \theta \cong \mathrm{Ind}_{\theta(M)}^G (\pi \circ \theta) $ with isomorphism given by $f \mapsto f \circ \theta$ for $f \in \mathrm{Ind}_{M}^G(\pi) \circ \theta$. 
                                           
We assume that $M$ is $\theta$-stable and let $\pi$ be a $\theta$-stable representation $M(F)$ with an intertwining operator $A_{\theta}:(\pi\circ \theta,V) \rightarrow (\pi,V)$, then $\mathrm{Ind}_{M}^G(\pi)$ is also $\theta$-stable, with an intertwining operator $\mathrm{Ind}_M^G(\pi) \circ \theta \rightarrow \mathrm{Ind}_M^G(\pi \circ \theta) \rightarrow \mathrm{Ind}_M^G(\pi)$ (also denoted by $A_{\theta}$ ) given by $ f \mapsto f \circ \theta \mapsto A_{\theta } (f \circ \theta)$ where $A_{\theta } (f \circ \theta)$ is defined by $A_{\theta } (f \circ \theta) (g) = A_{\theta } (f \circ \theta (g))$.

\begin{remark}
Note that twisted traces (with normalized intertwining operators) of an irreducible representation are unique up to $\pm 1$ modulo different normalized intertwining operators, we also use $\mathrm{\mathrm{Tr}}_{\theta}(f, \pi)$ to denote $ \mathrm{\mathrm{Tr}}(f,\pi A_{\theta} )$ when the intertwining operator is not specified.  
\end{remark}
	
Let $\mathcal P$ denote the set of standard parabolic subgroups of $G$, and let $\mathcal P^{\theta}$ denote the subset of $\mathcal P$ consisting of $ P \in \mathcal P$ which are invariant under $\theta$. For $g \in G(F)$, we say that $g$ is $\theta$-compact if $g \theta(g)$ is compact.

	Recall that we use $C_c^{G(F)}$ to denote be the characteristic function of the set of compact elements of $ G(F)$, and that we use $C_{\theta c}^{G(F)}$ to denote be the characteristic function of the set of $\theta$-compact elements of $ G(F)$. The subset $G_c \subset G(F), G_{\theta c}\subset G(F)$ are open and closed. Note that we also use $C_c$ or $C_{ \theta c}$ to denote $C_c^{G(F)}$ or $C_{\theta c}^{G(F)}$ if there is no ambiguity. 

\begin{definition}
If $P = MN \in \mathcal P^{\theta}$, then we define $\mathfrak a_P^{\theta}$ to be the subspace of $\theta$-invariants of $\mathfrak a_P$. We use $a_{P}^{\theta}$ to denote the dimension of $\mathfrak a_P^{\theta}$, and we use $\varepsilon_{P, \theta}$ to denote $(-1)^{a_{P}^{\theta}-a_{G}^{\theta}}$. 

Recall that we define $\tau _P^G$ and $\hat{\tau} _P^G$ in $\S2.1$ after Definition 2.1.2. We define the function $\chi^{\theta}_{G,N}$ to be the composition $\tau _P^G \circ (\mathfrak a_P \twoheadrightarrow \mathfrak a^G_P ) \circ  (H_M+ \theta \circ H_M)$,
and we define the function $\hat{\chi}^{\theta}_{G,N}$ to be the composition  $\hat{\tau} _P^G \circ (\mathfrak a_P \twoheadrightarrow \mathfrak a^G_P ) \circ  (H_M+ \theta \circ H_M)$.

We also use $\chi^{\theta}_{N}$ (or $\hat\chi^{\theta}_{N}$) to denote $\chi^{\theta}_{G,N}$ (or $\hat\chi^{\theta}_{G,N}$) if there is no ambiguity. 
\end{definition}

\begin{lemma} 
Let $P = M N$ be a $\theta$-stable standard parabolic subgroup of $G$. Let $m$ be an element of $M$. Then we have: \begin{itemize}
\item $\hat \chi^{\theta}_N(m)=1$ if and only if for all roots $\alpha$ in the set $\Delta_P$ we have $|\omega_{\alpha}(m\theta(m))| < 1$,

\item $\chi^{\theta}_N(m)=1$ if and only if for all roots $\alpha$ in the set $\Delta_P$ we have $\alpha(m\theta(m))| < 1$ (Recall that we define $\omega_{\alpha}$ in $\S2.1$ after Definition 2.1.2).
\end{itemize}
\end{lemma}

\begin{proposition}
Let $K$ be a hyperspecial subgroup of $G(\mathbb Q_p)$. Let $(\pi, V)$ be a ramified admissible $\theta$-stable representation of $G(\mathbb Q_p)$ (i.e., $V^K = \{0\}$) with an intertwining operator $A_{\theta}$. Let $f $ be an element in $\mathcal H^{\mathrm{unr}}(G(\mathbb Q_p))$, then we have $\mathrm{\mathrm{Tr}}(f, \pi A_{\theta}) = 0$. 
\end{proposition}

\begin{proof}
Note that $f$ is $K$-bi-invariant, by Definition 2.2.2, then we have $$\mathrm{\mathrm{Tr}}(f,\pi A_\theta) =\mathrm{\mathrm{Tr}} (v \mapsto \int_{G(F)} f(g) \pi (g)  A_{\theta} v dg) |_{V^K} = 0.$$ 
\end{proof}

\begin{proposition}

Let $(\pi, V, A_{\theta})$ be a $\theta$ -stable irreducible unramified representation of $G(F)$ such that $A_{\theta}$ is normalized. Let $K \subset G(F)$ be a hyperspecial subgroup, and we assume that $\theta(K) = K$. We write $V^K = \mathbb C v$, then we have:

\begin{enumerate}
    \item $A_{\theta}v =cv$ for some $c \in \mathbb C^{\times}$, 
    \item For an arbitrary $f \in \mathcal H(G(F))$, we have $\mathrm{\mathrm{Tr}}(f,\pi) = \pm \mathrm{\mathrm{Tr}}(f, \pi A_{\theta})$.
\end{enumerate}

\end{proposition}

\begin{proof}
 
The proof of (1): For every $k \in K$, we have that $\pi(k) A_{\theta} v = A_{\theta} \pi(\theta(k))v  =   A_{\theta} v $ (here we use the assumption that $\theta(K) = K$). Thus $A_{\theta} v \in \pi^K$ and we have $A_{\theta}v =cv$ for some $c \in \mathbb C^{\times}$.

The proof of (2):  We write $A_{\theta} v = cv $ where $c \in \mathbb C^{\times}$, then we have $v =  A_{\theta}(A_{\theta} v) = A_{\theta} cv = c^2 v$, and therefore $c = \pm 1$. The twisted trace $\mathrm{\mathrm{Tr}}(f, \pi A_\theta)$ is $$\mathrm{\mathrm{Tr}} (v \mapsto \int_{G(F)} f(g) \pi(g) A_\theta v)|_{V^K}  =  \mathrm{\mathrm{Tr}} (v \mapsto \int_{G(F)} f(g) \pi(g)  cv) |_{V^K}$$
$$= c\mathrm{\mathrm{Tr}} (v \mapsto \int_{G(F)} f(g) \pi(g)  v) |_{V^K} = \pm \mathrm{\mathrm{Tr}} (v \mapsto \int_{G(F)} f(g) \pi(g) v)|_{V^K},$$ then we have (2).

\end{proof}

Now let $E/F $ be an unramified extension of degree $2$, in the following, let $\theta$ be the involution of $G(E)$ arising from the non-trivial element of $\mathrm{Gal}(E/F)$ (therefore $\theta(K ) = K$). Then we have the following twisted version of Clozel's formula for compact traces (p263 of \cite{Clo90} and see $\S2.1$ in \cite{Nguyen04} for the proof).

\begin{proposition}
Let $\pi$ be a $\theta$-stable (with an intertwining operator $A_{\theta}$), admissible representation of $G(F)$ of finite length, then we have $$\mathrm{\mathrm{Tr}}(C_{\theta,c}f, \pi A_{\theta}) = \sum\limits_{P \in \mathcal P^{\theta}, P = MN}\varepsilon_{P, \theta} \mathrm{\mathrm{Tr}}(\hat{\chi}^{\theta}_N \bar f^{(P),\theta}, \pi_N(\delta_{P}^{-\frac 12})A_{\theta}),$$ where $$\bar f^{(P), \theta}(m) := \delta_P^{\frac 12}(m) \int_{N \times K} f(kmn\theta(k^{-1})) dk dn$$ for $m \in M(F)$ and $\varepsilon_{P, \theta}$ as defined in Definition 2.2.3.
\end{proposition}

Let $\mathrm{St}_{G(E)}$ denote the Steinberg representation of $G(E)$ (see \cite{Rod86}), then we have: 

\begin{corollary}
Let $f \in \mathcal H^{unr}(G(E))$ be a spherical Hecke operator, then we have $$\mathrm{\mathrm{Tr}}(C_{\theta, c}f, \mathrm{St}_{G(E)} A_{\theta}) = \mathrm{\mathrm{Tr}}(\hat{\chi}^{\theta}_{N_0}f^{(P_0)}, 1(\delta_{P_0}^{\frac 12})A_{\theta}).$$ 
\end{corollary}

\begin{proof}
 (cf. Proposition 6 in \cite{Kret11}): We have $$\bar f^{(P), \theta}(m) = \delta_P^{\frac 12}(m) \int_{N \times K} f(kmn\theta(k^{-1})) dk dn = \delta_P^{\frac 12}(m) \int_{N } f(mn) dn = f^{(P)}(m)$$ because $f $ is spherical and we normalize the Haar measure on $G(F)$ so that $\mu(K) = 1$. By applying Proposition 2.2.3, we obtain
$$\mathrm{\mathrm{Tr}}(C_{\theta, c}f, \mathrm{St}_{G(E)} A_{\theta}) = \sum\limits_{P \in \mathcal P^{\theta}, P = MN}\varepsilon_{P,\theta} \mathrm{\mathrm{Tr}}( \hat{\chi}^{\theta}_N f^{(P)}, (\mathrm{St}_{G(E)} )_N(\delta_{P}^{-\frac 12})A_{\theta})$$
The Jacquet module $(\mathrm{St}_{G(E)} )_N$ equals $\mathrm{St}_{M(E)}(\delta_P)$(Proposition 2 of \cite{Rod86}). If the parabolic subgroup $P = M N \subset G$ is not the Borel subgroup $P_0 = M_0 N_0$, we have $$\mathrm{\mathrm{Tr}}( \hat{\chi}^{\theta}_N f^{(P)}, (\mathrm{St}_{G(E)} )_N(\delta_{P}^{-\frac 12})A_{\theta}) = \mathrm{\mathrm{Tr}}( \hat{\chi}^{\theta}_N f^{(P)}, \mathrm{St}_{M(E)}(\delta_P^{\frac 12})A_{\theta}) = 0$$ because $\mathrm{St}_{M(E)}(\delta_P^{\frac 12})$ is ramified and the function $\hat{\chi}^{\theta}_N f^{(P)}$ is spherical. The Jacquet module $(\mathrm{St}_{G(E)})_{N_0}$ is equal to $1(\delta_{P_0})$, then the corollary follows. 
\end{proof}

We define $\tilde G = G (E) \rtimes \langle \theta \rangle$. Let $\pi$ be a $\theta$-stable representation of $ G (E)$ with an intertwining operator $A_{\theta}$, we define $\tilde \pi $ to be a representation of $\tilde G(E)$ given by $\tilde \pi(g,1) = \pi(g)$ and $\tilde \pi(g,\theta) = \pi(g) A_{\theta}$. Let $f$ be an element of $\mathcal H(G(E))$, we define $\tilde f \in \mathcal H(\tilde{ G})$ by $\tilde f(g, 1) = 0$ and $\tilde f(g,\theta) = f(g)$. Then we have $\mathrm{\mathrm{Tr}}(f, \pi I_{\theta}) = \mathrm{\mathrm{Tr}}(\tilde f, \tilde \pi )$. We fix a hyperspecial subgroup  $K \subset G(E)$, we define $\bar f^{\theta, K} \in \mathcal H(G(E))$ by $\bar f^{\theta, K}(g) = \int_K  f(kg\theta(k^{-1}))dk$ (The Haar measure is normalized so that $K$ has volume $1$). 

\begin{proposition}
Let $\Omega$ be an open and closed subset of $G(E)$ which is invariant under $\theta$-conjugation by $ G(E)$ (we normalize the Haar measure on $G(E)$ so that $\Omega$ has volume $1$). Let $P = M N$ be a $\theta$-stable standard parabolic subgroup of $G$ (we normalize the Haar measure on $M(E)$ so that $\Omega \cap M(E)$ has volume $1$). Let $\pi$ be an admissible representation of $M(E)$ of finite length, and let $A_{\theta}$ is an intertwining operator of $\pi$. Note that $A_{\theta}$ induces an intertwining operator of $\mathrm{Ind}_P^G(\pi)$, which we also denote by $A_{\theta}$. 
 Then for all $f$ in $\mathcal H(G(E))$, we have
 $$\mathrm{\mathrm{Tr}}_{G(E)}(\chi_\Omega f,\mathrm{Ind}_P^G(\pi) A_{\theta})= \mathrm{\mathrm{Tr}}_{M(E)}(\chi_{\Omega}((\bar f^{\theta, K}) ^{(P)}),\pi A_{\theta}) = \mathrm{\mathrm{Tr}}_{M(E)}((\chi_{\Omega}  \bar f^{\theta, K})^{ (P)},\pi A_{\theta}) .$$ 
 In particular, if $f \in \mathcal H^{\mathrm{unr}}(G(E))$, then we have $$\mathrm{\mathrm{Tr}}_{G(E)}(\chi_\Omega f,\mathrm{Ind}_P^G(\pi) A_{\theta})=\mathrm{\mathrm{Tr}}_{M(E)}(\chi_\Omega f ^{( P )} , \pi A_{\theta} ) = \mathrm{\mathrm{Tr}}_{M(E)}((\chi_{\Omega}   f)^{ (P)},\pi A_{\theta}).$$ 
\end{proposition}

\begin{proof}
(cf. Proposition 3 in \cite{Kret11}).
By Theorem 5.9.2 of \cite{BG17} (p107) we have
$$\mathrm{\mathrm{Tr}}_{G(E)}(\chi_{\Omega} f,\pi A_{\theta})=\mathrm{\mathrm{Tr}}_{ M(E)}((\overline{\chi_{\Omega} f}^{\theta, K})^{(P)},\rho A_{\theta}).$$

 we first prove that the functions $\chi_{\Omega}((\bar f^{\theta, K}) ^{(P)})$ and $(\chi_{\Omega}  \bar f^{\theta, K})^{ (P)}$ in $\mathcal H(M(F))$ have the same twisted orbital
integrals. Note that $\Omega$ is invariant under $\theta$-conjugation by $G(E)$, we have $$\overline{\chi_{\Omega} f}^{\theta, K} (g)=  \int_K  \chi_{\Omega}(kg\theta(k^{-1}))f(kg\theta(k^{-1}))dk = \int_K  \chi_{\Omega}(g)f(kg\theta(k^{-1}))dk = \chi_{\Omega}\bar f^{\theta, K}(g).$$ 
Applying Proposition 7.2.2 in Part I of \cite{BG17}, for $\gamma \in \Omega \cap M(E)$, we have
$$\mathrm{TO}_{\gamma}^ {M(E)} ((\chi_{\Omega}\bar f^{\theta, K})^{(P)})= \mathrm{TO}_{\gamma}^{M(E)}((\overline{\chi_{\Omega} f}^{\theta, K})^{(P)})= D_{G/M}(\gamma) \mathrm{TO}_{\gamma}^{G(E)}(\chi_{\Omega} f),$$
where $D_{G/M}(\gamma)$ is a certain Jacobian factor defined in $\S7.2$, Part I of \cite{BG17}), and for $\gamma \not\in \Omega \cap M(E)$ we have $\mathrm{TO}_{\gamma}^ {M(E)} ((\chi_{\Omega}\bar f^{\theta, K})^{(P)})=0$.

For $\gamma \in M(E)$, we have $\mathrm{TO}_{\gamma} ^{M(E)} (\chi _{\Omega}  ( (\bar f^{\theta, K}) ^{(P)}))$ equals $\mathrm{TO}_{\gamma}^{M(E)} ((\bar f ^{\theta, K})^{(P)})$ if $\gamma \in \Omega$ and equals $0$ if $\gamma \not \in \Omega$. Applying Proposition 7.2.2 in Part I of \cite{BG17} again, for $\gamma \in \Omega \cap M(E)$ we have $$\mathrm{TO}_{\gamma}^{M(E)} (\chi_{\Omega}((\bar  f^{\theta, K}) ^{(P)})) = \mathrm{TO}_{\gamma}^{M(E)} ((\bar  f^{\theta, K}) ^{(P)}) = D_{G/M}(\gamma )\mathrm{TO}_{\gamma}^{G(E)} ( f ) =D_{G/M}(\gamma )\mathrm{TO}_{\gamma} ^{G(E)} (\chi_{\Omega} f ) $$ and $TO_{\gamma}^{M(E)} (\chi_{\Omega}((\bar  f^{\theta, K}) ^{(P)}))= 0$ if $\gamma \not \in \Omega \cap M(E)$. Thus the twisted orbital integrals of the functions  $\chi_{\Omega}((\bar f^{\theta, K}) ^{(P)})$ and $(\chi_{\Omega}  \bar f^{\theta, K})^{ (P)}$ agree.

By a twisted version of Weyl’s integration formula for the group $M(E)$ (Corollary 7.3.8 of \cite{BG17}), we know that for $h \in \mathcal H(M(E))$, the twisted trace $\mathrm \mathrm{Tr}(h, \pi A_{\theta})$ only depends on the Harish-Chardra character $\Theta_{\pi A_{\theta}}$ and $\mathrm{TO}_{\gamma}^{M(E)} (h)$ (for $\gamma \in M(E)$). Therefore $\mathrm{\mathrm{Tr}}_{M(E)}(\chi_{\Omega}((\bar f^{\theta, K}) ^{(P)}),\pi A_{\theta}) = \mathrm{\mathrm{Tr}}_{M(E)}((\chi_{\Omega}  \bar f^{\theta, K})^{ (P)},\pi A_{\theta}) = \mathrm{\mathrm{Tr}}_{ M(E)}((\overline{\chi_{\Omega} f}^{\theta, K})^{(P)},\rho A_{\theta}) = \mathrm{\mathrm{Tr}}_{G(E)}(\chi_{\Omega} f,\pi A_{\theta})$. 

\end{proof}

\subsection{Some unitary groups}	
	
	Let $E/F$ be a quadratic field extension where $F$ is a $p$-adic field, a number field or $\mathbb R$. Let $\sigma $ be the nontrivial element of $\mathrm{Gal}(E/F)$. In the following, we introduce some unitary groups and their parabolic subgroups following $\S$2.2 of \cite{Morel08}.

	For an arbitrary $a \in E$, we also write $\sigma(a) = \bar a$. For $n \in \mathbb Z_{>0}$ and $s \in \mathbb Z_{>0}$ with $s < n$, let \begin{displaymath}I_n =\begin{pmatrix}1& & 0\\ & \ddots & \\ 0& & 1\end{pmatrix} \in GL_n(F) \end{displaymath}be the diagonal matrix with entries $1$, and let $$A_n=\begin{pmatrix}0& & 1\\ & \ddots & \\ 1& & 0\end{pmatrix} \in GL_n(F)$$ be the anti-diagonal matrix with entries $1$. Write $I_{s, n -s} = \begin{pmatrix} I_s & 0 \\ 0&-I_{n-s} \end{pmatrix} \in GL_n(F)$.

	Let $GU_{E/F}(s,n-s)$ denote the algebraic group over $F$ whose $A$-points is given by $$GU_{E/F}(s, n-s)(A) = \{g \in GL_n(A \otimes_{F}E) : g^*I_{s, n -s}g = c(g), c(g)\in A^{\times}\}$$ where $g^*$ denote the conjugate transpose and $A$ is an $F$-algebra. We consider the algebraic group $GU_{E/F}(A_{n})$ for $n \geq 0$. Let $GU_{E/F}(n)$ denote $GU_{E/F}(A_n)$. We define $GU_{E/F}(A_0)$ to be $ \mathbb G_m$ and we define $c :GU_{E/F}(A_0) \rightarrow \mathbb G_m $ to be $\mathrm{id}_{\mathbb G_m}$. In the following, let $E/F$ be an unramified quadratic field entension of $p$-adic fields.

	 A maximal torus of $GU_{E/F}(n) $ is the diagonal torus $T_{E/F}(n) $ wich is defined to be $$  \{diag(\lambda_1,\lambda_2,\cdots,\lambda_n):\lambda_1, \lambda_2,\ldots, \lambda_n \in \mathrm{Res}_{\mathbb 	Q_p}^E(\mathbb G_m), \lambda_1\bar\lambda_n=\lambda_2\bar\lambda_{n-1} = \ldots = \lambda_{n-1}\bar\lambda_{1}  \in \mathbb G_m\}.$$  The maximal split torus in $T_{E/F}(2k+1)$ is $$S_{E/F}(2k+1) := \left\{\begin{pmatrix}\lambda\lambda_1&\cdots&0&0&0&\cdots&0\\0&\ddots&0&0&0&\cdots&0\\0&\cdots&\lambda\lambda_k&0&0&\cdots&0\\0&\cdots&0&\lambda&0&\cdots&0\\0&\cdots&0&0&\lambda\lambda_{k}^{-1}&\cdots&0\\0&\cdots&0&0&0&\ddots&0\\0&\cdots&0&0&0&\cdots&\lambda\lambda_{1}^{-1} \end{pmatrix}:\lambda, \lambda_1,\ldots, \lambda_{k} \in \mathbb G_m\right\}.$$ And the maximal split torus of $T_{E/F}(2k)$ is $$S_{E/F}(2k) := \left\{\begin{pmatrix}\lambda\lambda_1&\cdots&0&0&\cdots&0\\0&\ddots&0&0&\cdots&0\\0&\cdots&\lambda\lambda_k&0&\cdots&0\\0&\cdots&0&\lambda_{k}^{-1}&\cdots&0\\0&\cdots&0&0&\ddots&0\\0&\cdots&0&0&\cdots&\lambda_{1}^{-1} \end{pmatrix}:\lambda, \lambda_1,\ldots, \lambda_{k} \in \mathbb G_m\right\}.$$
	A minimal parabolic subgroup of $GU_{E/F}(n)$ containing $T_{E/F}(n)$ is $$B_{E/F}(n) = B_n \cap GU_{E/F}(n)$$ where $B_n \subset \mathrm{Res}_{F}^E (GL_n)$ is the subgroup of upper triangular matrices. 

 We write $k_n = [\frac n2]$ if $n$ is odd and $k_n = [\frac n2]- 1$ if $n$ is even. All standard parabolic subgroups of $GU_{E/F}(n)$ (with respect to $P_{n,0}$) are of the form $$ \begin{pmatrix}\mathrm{Res}_{F}^E (GL_{r_1})& & & & & &  *\\ &\ddots& & & & \\ & &\mathrm{Res}_{F}^E(GL_{r_k})& & & \\ & & &GU_{E/F}(n- 2r)& & & \\ & & & & \mathrm{Res}_{F}^E (GL_{r_1}) & &\\ & & & & &\ddots& \\ 0& &  & &  & &\mathrm{Res}_{F}^E (GL_{r_k})&  \end{pmatrix} \cap GU_{E/F} (n).$$
Where $r  = r_1 + \ldots +r_k \leq k_n$. Write $J= \{r_1, \ldots, r_m\}$, we use $P_{n,J}$ to denote the above parabolic subgroup of $GU_{E/F}(n)$. Let $ N_{n,J}$ denote the unipotent radical of $P_{n,J}$, and let $M_{n,J}$ denote the Levi subgroup corresponding to $P_{n,J}$. 
Note that there is an isomorphism $$f_J: M_{n,J} \rightarrow \mathrm{Res}_{F}^E(GL_{r_1}) \times \ldots \times \mathrm{Res}_{F}^E(GL_{r_k}) \times GU_{E/F}( n-2r ) $$ given by 

$$diag(g_1, \ldots, g_k,  g, h_k, \ldots, h_1) \mapsto (c(g^{-1})g_1, \ldots ,c(g)^{-1}g_k, g). $$

    We give an explicit description of the unramifed Hecke algebra $\mathcal H^{\mathrm{unr}}(GL_n(F))$ and $\mathcal H^{\mathrm{unr}}(GU_{E/F}(n)(F))$:

Let $B_n =  TU$ be a Borel subgroup of $GL_n$ consisting of upper triangular matrices where $T$ is the diagonal torus. The unramifed Hecke algebra $\mathcal H^{\mathrm{unr}}(GL_n(F))$ is isomorphic to $\mathbb C[X_1^{\pm1}, \ldots, X_n^{\pm1}]^{S_n} \hookrightarrow \mathbb C[X_1^{\pm1}, \ldots, X_n^{\pm1}] \cong \mathcal H^{\mathrm{unr}}(T(F))$ such that $X_k$ corresponds to the characteristic function of $T(\mathcal O_F) a_k  T(\mathcal O_F)$, where $a_k = \mathrm{diag}(a_{k,i}) \in T(F)$ such that $a_{k,k} = \pi_F $ and $a_{k,i} = 1 $ if $i \not = k$. 
    
Let $B_{E/F}(n)$ be a Borel subgroup of $GU_{E/F}(n)$ as defined in this section. We take a maximal torus $T_{E/F}(n) \subset GU_{E/F}(n)$ and a maximal split torus $S_{E/F}(n) \subset T_{E/F}(n)$ as defined in this section. We define $k = \frac n2$ if $n$ is even, and we define $k = \frac{n-1}2$ if $n$ is odd. The unramifed Hecke algebra $\mathcal H^{\mathrm{unr}}(GU_{E/F}(n)(F))$ is isomorphic to $$\mathbb C[X^{\pm 1}, X_1^{\pm1}, \ldots, X_k^{\pm1}]^{\{\pm 1\}^k \rtimes S_k } \hookrightarrow \mathbb C[X^{\pm 1}, X_1^{\pm1}, \ldots, X_k^{\pm1}] \cong \mathcal H^{\mathrm{unr}}(T_{E/F}(n)(F)) \cong \mathcal H^{\mathrm{unr}}(S_{E/F}(n)(F))$$ with $\mathcal H^{\mathrm{unr}}(T_{E/F}(n)(F)) \cong \mathcal H^{\mathrm{unr}}(S_{E/F}(n)(F))$ given by $f \mapsto f|_{S_{E/F}(n)(F)}$ for $f \in \mathcal H^{\mathrm{unr}}(T_{E/F}(n)(F))$ (see 2.4 of \cite{Minguez11}), and where:

\begin{itemize}
    \item $X_l$ corresponds to the characteristic function of $S_{E/F}(n)(\mathcal O_F) a_l  S_{E/F}(n)(\mathcal O_F)$ where $a_l = \mathrm{diag}(a_{l,i}) \in S_{E/F}(n)(F)$, such that $a_{l,l} = a_{l,n-l}= \pi_F $ and $a_{l,i} = 1 $ if $i \not = l$ and $i \not = n-l$. 

     \item $X$ corresponds to the characteristic function of $S_{E/F}(n)(\mathcal O_F) a_l  S_{E/F}(n)(\mathcal O_F)$ where $a_l = \mathrm{diag}(a_{l,i}) \in S_{E/F}(n)(F)$ such that $$\begin{cases}
   a_{l,i} =  \pi_F  \text{ for all i when n is odd} \\
     a_{l,i} =  \pi_F  \text{ if } 1 \leq i \leq \frac n2, a_{l,i} =  1  \text{ if } \frac n2 \leq i \leq n \text{ when n is even} 
    \end{cases}.$$ 

    \item $\sigma \in S_k$ acts by permutation on $(\alpha_1, \ldots, \alpha_k) \in \{\pm 1\}^k$. 

    \item $\sigma \in S_k$ acts by permutation on $X_1, \ldots, X_k$ and acts trivially on $X$.  

    \item $(\alpha_1, \ldots, \alpha_k) \in \{\pm 1\}^k$ acts by $(\alpha_1, \ldots, \alpha_k) X_i = X_i^{\alpha_i}$ and $(\alpha_1, \ldots, \alpha_k) X= X \prod\limits_{i: \alpha_i = -1} X_i^{-1}$,

\end{itemize}  see 4.2 of \cite{Morel08}

 \subsection{The Kottwitz functions $f_{n\alpha s}, \phi_{n\alpha s}$ and $\varphi_{n \alpha s}$}

In this section, we introduce the Kottwitz functions $f_{n\alpha s}$ and $\varphi_{n \alpha s}$ which play an important role in point-counting of Shimura varieties. 

Let $n \in \mathbb Z_{\geq 0}$ be a non-negative integer. A \textbf{composition} of $n$ is an element $(n_i ) \in \mathbb Z^k_{\geq 1}$ for some $k \in \mathbb Z_{\geq 1}$ such that $n = \sum_{i=1}^k n_i$. We write $\ell((n_i))$ for $k$ and we refer to $\ell((n_i))$ as the length
of the composition. Let $P_{0,n} = M_{0,n}N_{0,n}\subset GL_n$ be a Borel subgroup consisting of upper triangular matrices. The set of compositions $(n_i)$ of $n$ is in bijection with the set of
standard parabolic subgroups of $GL_n$ (with respect to $P_{0,n}$): A composition $(n_i)$ of $n$ corresponds to the standard parabolic subgroup 
$$P(n_i) := \left\{\begin{pmatrix}
g_1 &  & * \\
 & \ddots &  \\
0 &   & g_k \\
\end{pmatrix} \in GL_n: g_i \in GL_{n_i}\right\}$$
An \textbf{extended composition} of $n$ is an element $ (n_i ) \in \mathbb Z^k_{\geq 0}$ for some $k \in \mathbb Z_{\geq 1}$  such that $n = \sum_{i=1}^k n_i$. 

In the following, let $E/F$ be a quadratic extension of $p$-adic fields, we assume that $[\mathcal O_E/(\pi_E):\mathcal O_F/(\pi_F)] = 2$ where $\pi_F$ (resp. $\pi_E$) is a uniformizer of $\mathcal O_F$ (resp. $\mathcal O_E$). We write $q = |\mathcal O_F/(\pi_F)|$. Let $\sigma $ be the nontrivial element of $\mathrm{Gal}(E/F)$. We assume that $\pi_F$ remains a uniformizer in $E$, and we normalize the $p$-adic norm on $E$ so that $|\pi_F| = q^{-1}$.

\begin{definition}

We define $\phi_{n\alpha s} \in \mathcal H^{\mathrm{unr}}(GL_n(E)) $ to be the unramified Hecke operator whose Satake transform is given by $$q^{\alpha s(n-s)}\sum\limits_{\substack{I \subset {1,2, \ldots, n}\\|I| = s}}(Y^I)^{\alpha}$$ where $Y^I : = \prod_{i \in I}Y_{i}$. 

We define $\varphi_{n\alpha s} \in \mathcal H^{\mathrm{unr}}(E^{\times} \times GL_n(E)) $ to be the unramified Hecke operator whose Satake transform is given by $ Y^{\alpha} \otimes \phi_{n \alpha s} $.

 We define $$\begin{cases} X_i = X_{n-i}^{-1} \text{ if $n$ is even and $i > [n/2]$ }\\ X_i = X_{n-i}^{-1} \text{ if n is odd and $i > [n/2] + 1$} \\X_i = 1 \text{ if $n$ is odd and $i = [n/2] + 1$} \end{cases}.$$ And we define $f_{n \alpha s} \in  \mathcal H^{\mathrm{unr}}(GU_{E/F}(n)(F))  $ to be the unramified Hecke operator whose Satake transform is given by $$ q^{\alpha s(n-s)} X^{\alpha}\sum\limits_{\substack{I \subset \{1,2, \ldots, n\}\\|I| = s}}(X^I)^{\alpha}$$ where $X^I : = \prod_{i \in I}X_{i}$. 
 
\end{definition}	
	
\begin{proposition}
The spherical Hecke functions $f_{n\alpha s} \in  \mathcal H^{\mathrm{unr}}(GU_{E/F}(n)(F))$ and $\varphi_{n \alpha s} \in \mathcal H^{\mathrm{unr}}(GL_n(E))$ are associated in the sense of base change (i.e., the functions $f_{n\alpha s}$ and $\varphi_{n \alpha s}$ have matching orbital integral, as defined in p239 of \cite{Kot86unr}). 
\end{proposition}

\begin{proof}
See $\S 4.1$ and $\S 4.2$ of \cite{Morel08}.
\end{proof}

\begin{example}
 The function $f_{n \alpha 1} \in \mathcal H^{\mathrm{unr}}(GU_{E/F}(n))$ is the unramified Hecke operator whose Satake transform is given by $$\mathcal S_{GU_{E/F}(n)}(f_{n\alpha 1}) = \begin{cases}q^{(n-1) \alpha} X^{\alpha}(\sum_{i = 1}^{[\frac n2]}(X_i^{\alpha}+ X_i^{-\alpha})) \text{ if n is even}\\ q^{(n-1) \alpha} X^{\alpha}(\sum_{i = 1}^{[\frac n2]}(X_i^{\alpha}+ X_i^{-\alpha}) +1) \text{ if n is odd} \end{cases}.$$ 
\end{example}

\begin{lemma}

\begin{enumerate}
    
\item  The function $\phi_{n\alpha s} $ is supported on the set of elements $g \in GL_n(E)$ with $| \mathrm{det} g| =q^{-\alpha s}$. 

\item The function $f_{n\alpha s} $ is supported on the set of elements $g \in GU_{E/F}(n)(F)$ with $| c (g)| = q^{-s\alpha}$ if $n$ is even and $| c (g)| = q^{-2s\alpha}$ if $n$ is odd. 

\end{enumerate}
\end{lemma}

\begin{proof}
The proof of (1): see Lemma 3 in \cite{Kret11}. 

The proof of (2): Assume first that $n$ is even. Let $\chi$ be the characteristic function of the set $\{g \in GU_{E/F}(n): |c(g)| = q^{-s\alpha} \}$. Let $B_{E/F}(n)$ be a Borel subgroup of $GU_{E/F}(n)$ as defined in $\S2.2$, we write $B_{E/F}(n) = T_{E/F}(n) U$ where $T_{E/F}(n)$ is a maximal torus as defined in $\S2.2$ and $U$ is the unipotent part of $B_{E/F}(n)$, let $S_{E/F}(n ) \subset T_{E/F}(n)$ be a maximal split torus (as defined in $\S2.2$). We have $$(\chi f_{n \alpha s})^{B_{E/F}(n)} (m )  = \delta_{B_{E/F}(n)}^{\frac 12}(m)\int_{U(F)} \chi(mn) f(mn) dn = \delta_{B_{E/F}(n)}^{\frac 12}(m)\int_{U(F)} \chi(m)  f(mn) dn,$$ thus $$(\chi f_{n \alpha s})^{B_{E/F}(n)} = \chi|_{T_{E/F}(n)(F)} \cdot f_{n \alpha s}^{B_{E/F}(n)}.$$ By the definition of $f_{n \alpha s}$ and $S_{E/F}(n) \subset T_{E/F}(n)$, 
we have $$\chi|_{S_{E/F}(n)(F)} \cdot f_{n \alpha s}^{B_{E/F}(n)}|_{S_{E/F}(n)(F)} = f_{n \alpha s}^{B_{E/F}(n)}|_{S_{E/F}(n)(F)}.$$
Also note that the map $f \mapsto f^{B_{E/F}(n)} \mapsto f^{B_{E/F}(n)}|_{S_{E/F}(n)} $ is injective for $f \in \mathcal H^{\mathrm{unr}}(GU_{E/F}(n)(F))$ (see 2.2 and 2.4 of \cite{Minguez11}), then we may conclude that $\chi f_{n \alpha s} = f_{n \alpha s}$. 
Thus the function $f_{n\alpha s} $ is supported on the set of elements $g \in GU_{E/F}(n)(F)$ with $| c (g)| = q^{-s\alpha}$ if $n$ is even. For odd $n$, a similar argument follows.

\end{proof}

The following proposition indicates that the trace of $C_cf_{n \alpha s}$ may not vanish against induced representation, even when $s = 1$, as a result, some induced representation might contribute to the cohomology of basic stratum, unlike the cases considered in \cite{Kret11}.

\begin{proposition}
Let $P_{n,J} = M_{n,J}N_{n,J}$ be a standard parabolic subgroup of $GU_{E/F}(n)$ with $J=\{r_1,\ldots,r_k\}$, then we have:\begin{itemize}
\item Then function $C_c^{GU_{E/F}(F)}f_{n\alpha 1}^{(P_{n,J})}\not = 0$ only if $\begin{cases}k= 1, r_1 = 1 \text{ if n is odd}\\k = 1, r_1 = 2\text{ if n is even} \end{cases}$.
\item If $\begin{cases}k= 1, r_1 = 1 \text{ and n is odd, or}\\k = 1, r_1 = 2, n \geq 4\text{, n is even and } \alpha \text{ is even} \end{cases}$, then $C_c^{GU_{E/F}(F)}f_{n\alpha 1}^{(P_{n,J})}\not = 0$. 
\end{itemize}
\end{proposition}

\begin{proof}
Let $ m = (n_1,\ldots,n_k,g,m_k,\ldots,m_1) $ be a semi-simple element of $M_{n,J}(F) \cong (GL_{r_1} \times \ldots \times GL_{r_k} \times GU_{E/F}(n-r) )(F)$, with image $(c(g)^{-1}n_1,\ldots, c(g)^{-1}n_k,g)$ in  $(GL_{r_1} \times \ldots \times GL_{r_k} \times GU_{E/F}(n-r) )(F)$, and such that $C_c^{GU_{E/F}(n)(F)}f_{n\alpha 1}^{P_{n,J}}(m) \not = 0 $. Note that $n_{i}^*m_i =c(g) A_{r_i}$, which implies that $|\mathrm{det}(n_i)\mathrm{det}(m_i)|_{\mathfrak p} = |c(g)\mathrm{det}(A_{r_i})|_{\mathfrak p} = |c(g)|_{\mathfrak p}^{ r_i}$.  Also note that $m$ is compact in $GU_{E/F}(F)$, then we have $|\mathrm{det}(n_i)|_{\mathfrak p} = |\mathrm{det}(m_i)|_{\mathfrak p}$. Thus $ |\mathrm{det}(n_i)|_{\mathfrak p} = |c(g)|_{\mathfrak p}^{\frac{r_i}2}$ and we have $|(\mathrm{det}(c(g)^{-1}n_i))|_{\mathfrak p} = |(c(g)|_{\mathfrak p} ^{-\frac{r_i}2} = \begin{cases}q^{\alpha\frac{r_i}2}\text{ if n is even}\\q^{\alpha r_i}\text{ if n is odd}\end{cases}.$

We define $r_0 = 0$, then we have: $$f_{n\alpha 1}^{(P_{n,J})} =q^{(n-1)\alpha} \sum_{i = 1}^k (f_{1,i} \otimes f_{2,i}\ldots\otimes f_{k,i} \otimes X^{\alpha}) + \otimes_{i=1}^k h_i \otimes X^{\alpha}h$$ where 

$$f_{j,i} =\begin{cases}
    \sum_{j = r_0 + r_1 + \ldots + r_{i-1} + 1}^{r_0 + \ldots r_i } (X_j^{-\alpha}+ X_j^{\alpha}) \in \mathcal {H}^{\mathrm{unr}}(GL_{r_i}(F))\text{, if } j = i, \\ f_{j,i} = 1 \in \mathcal {H}^{\mathrm{unr}}(GL_{r_i}(F)) \text{, otherwise.}
\end{cases}. $$

And $h_i = 1 \in  \mathcal H^{unr}(GL_{r_i}(F)), h \in  \mathcal H^{unr}(GU_{E/F}(n-r)(F))$ defined by $$ h = \begin{cases}\sum_{ j = r+ 1}^{[\frac n2]} X_j^{-\alpha}+ X_j^{\alpha}\text{, if n is even and }\frac n2 <r\\ 0 \text{, if n is even and }\frac n2 = r \\ \sum_{ j = r+ 1}^{[\frac n2]} X_j^{-\alpha}+ X_j^{\alpha}+ 1\text{, if n is odd and }[\frac n2] <r \\ 1 \text{, if n is odd and }[\frac n2] = r \end{cases}.$$

By applying Lemma 2 of \cite{Kret11}, we obtain that all $f_{i,i}$'s are supported on elements $a$ with $|\mathrm{det}(a) |_{\mathfrak p} = q^{\pm\alpha}$ , and $h_i$'s, $f_{i,j}$'s ($i \not = j$) are supported on elements $a$ with $|\mathrm{det}(a) |= 1$.  Then $C_c^{GU_{E/F}(n)(F)}f_{n\alpha 1}^{(P_{n,J})}\not = 0$ only if $\begin{cases}k= 1, r_1 = 1 \text{ if n is odd}\\k = 1, r_1 = 2\text{ if n is even} \end{cases}: $ We first consider the case that $n$ is odd, if $k > 1$, then there exists $1 \leq \alpha, \beta, \gamma\leq k$ such that $f_{\alpha,\beta} = 1 \in \mathcal {H}^{\mathrm{unr}}(GL_{r_{\beta}}(F)), h_{\gamma} = 1 \in \mathcal {H}^{\mathrm{unr}}(GL_{r_{\gamma}}(F))$. In this case, we have 
$$C_c^{GU_{E/F}(n)(F)}f_{n\alpha 1}^{(P_{n,J})} (n_1,\ldots,n_k,g,m_k,\ldots,m_1)  $$
$$=q^{(n-1)\alpha} C_c^{GU_{E/F}(n)(F)}(n_1,\ldots,n_k,g,m_k,\ldots,m_1) \times  $$
$$\sum_{i = 1}^k (f_{1,i} \otimes f_{2,i}\ldots\otimes f_{k,i} \otimes X^{\alpha}) (c(g)^{-1}n_1,\ldots, c(g)^{-1}n_k,g) + \otimes_{i=1}^k (h_i \otimes X^{\alpha}h) (c(g)^{-1}n_1,\ldots, c(g)^{-1}n_k,g)$$ equals $0$, because and $f_{i,j} = 1, h_{i'} = 1$ are supported on elements $a$ with $|\mathrm{det}(a) |= 1$, and $|(\mathrm{det}(c(g)^{-1}n_i))|_{\mathfrak p} = |(c(g)|_{\mathfrak p} ^{-\frac{r_i}2} =q^{\alpha\frac{r_i}2}= q^{\alpha r_i}$ if n is odd. If $k = 1$, then $$f_{n\alpha 1}^{(P_{n,J})} =q^{(n-1)\alpha} (\sum_{j = 1}^{r_1 } (X_j^{-\alpha}+ X_j^{\alpha}) \otimes X^{\alpha}) +1  \otimes X^{\alpha}h$$
are supported on elements $a$ with $|\mathrm{det}(a) |= 1$ or $|\mathrm{det}(a) |= q^{\pm \alpha}$, we must then have $r_1 = 1$. For even $n$, a similar argument follows. 

If $k = 1$, $r_1 = 1$ and $n$ is odd, we have $C_c^{GU_{E/F}(F)}f_{n\alpha 1}^{(P_{n,J})} (\mathrm{diag}(\pi_F^{\alpha} )) \not = 0$. If $k = 1$, $r_1 = 2, n \geq 4$ and $n$ is even, we have $$C_c^{GU_{E/F}(F)}f_{n\alpha 1}^{(P_{n,J})} ( \begin{pmatrix} \begin{pmatrix} 0&\pi_F^{\alpha}\\1&0 \end{pmatrix}&0_{2 , n-4} & 0_{2, 2}\\ 0_{n-4 , 2}& \mathrm{diag}(\pi_F^{\frac \alpha2}) & 0_{n-4, 2}\\ 0_{2,2}& 0_{2,n-4}& \begin{pmatrix} 0&1\\\pi_F^{\alpha}&0 \end{pmatrix}\end{pmatrix})\not = 0$$ where $0_{i,j}$ is the $i \times j$ zero matrix.
\end{proof}

\subsection{The functions $(C_{\theta c} \phi_{n \alpha s})^{(P)}$}

In this section, we study constant term of the truncated Kottwitz function $  (C_{\theta c} \phi_{n \alpha s})^{(P)} $, where $P \subset GL_n$ is a $\theta$-stable standard parabolic subgroup. 
  
\begin{proposition}
Let $P = M N$ be a standard parabolic subgroup of $GL_{n,E}$ corresponding to the composition $(n_i)$ of $n$. Let $k$ be the length of this composition. The constant term of $\phi_{n\alpha s}$ at $P$ is given by $\sum\limits_{(s_i)} q^{C(n_i,s_i)} \phi_{n_1\alpha s_1} \otimes  \phi_{n_2\alpha s_2} \otimes \ldots \otimes  \phi_{n_k\alpha s_k}$
where the sum ranges over all extended composition $(s_i)$ of $s$ of length $k$. The constant $C(n_i, s_i)$ is equal to $s(n-s)-\sum_{i = 1}^k s_i(n_i-s_i)$.

\end{proposition}

\begin{proof}
See Proposition 4 of \cite{Kret11}. 

\end{proof}

We define involution $\theta:GL_n(E) \rightarrow GL_n(E)$ by $g \mapsto A_n \bar g^{-t} A_n$, where $A_n$ is the anti-diagonal matrix with entries $1$. Let $P = M N$ be a standard parabolic subgroup of $GL_{n,E}$ corresponding to the composition $(n_i)$ of $n$ with length $k$, then $P$ is $\theta$-stable if and only if $n_i = n_{k+1 -i} $ for $1 \leq i \leq k$. For $A = (a_{ij}) \in GL_n$, it follows that $A_n A A_n = (a'_{ij})$, where $a'_{ij} = a_{n + 1-i, n+1 -j } $. 

\begin{proposition}
Let $P = M N$ be a $\theta$-stable standard parabolic subgroup of $GL_{n,E}$ corresponding to the composition $(n_i)$ of $n$, and let $k$ be the length of this composition. Let $\pi$ be a $\theta$-stable representation of $M(E)$ of finite length with an intertwining operator $A_{\theta}$, then the twisted trace $\mathrm{Tr}((C^{GL_n(E)}_{\theta c}\phi_{n\alpha s})^{(P)}, \pi A_{\theta})$ equals to $\mathrm{Tr}(f, \pi A_{\theta})$, with $f$ defined as follows: \begin{itemize}\item$ f = \sum\limits_{(s_i)} q^{C(n_i,s_i)} (C^{GL_{n_1} \times GL_{n_k}(E)}_{\theta c}(\phi_{n_1\alpha s_1} \otimes  \phi_{n_k\alpha s_k} ))\otimes \ldots  (C^{GL_{n_{\frac{k-1}2}} \times GL_{n_{\frac{k+3}2}}(E)}_{\theta c}(\phi_{n_{\frac{k-1}2}\alpha s_{\frac{k-1}2}} \otimes \phi_{n_{\frac{k+3}2}\alpha s_{\frac{k+3}2}})) \otimes (C^{GL_{\frac{k+1}2}(E)}_{\theta c}\phi_{n_{\frac{k+1}2}\alpha s_{\frac{k+1}2} })$ if $k$ is odd,
\item $f = \sum\limits_{(s_i)} q^{C(n_i,s_i)} (C^{GL_{n_1} \times GL_{n_k}(E)}_{\theta c}(\phi_{n_1\alpha s_1} \otimes  \phi_{n_k\alpha s_k} ))\otimes \ldots  (C^{GL_{n_{\frac{k}2}} \times GL_{n_{\frac{k+2}2}}(E)}_{\theta c}(\phi_{n_{\frac{k}2}\alpha s_{\frac{k}2}} \otimes \phi_{n_{\frac{k+2}2}\alpha s_{\frac{k+2}2}}))$ if $k$ is even, 

\end{itemize}
where the sum ranges over all extended composition $(s_i)$ of $s$ of length $k$ such that 
\begin{itemize}
\item  $s_{i} = s_{k +1 -i}$ for all $i$ if $k$ is even. 
\item $s_{i} = s_{k +1 -i}$ for all $i \not = \frac{k+1}2$ if $k$ is odd. 
\end{itemize} And the constant $C(n_i, s_i)$ is equal to $s(n-s)-\sum_{i = 1}^k s_i(n_i-s_i)$. 

\end{proposition}

\begin{proof}

We first consider the case that $k$ is even: By Proposition 2.2.4, we have $$\mathrm{Tr}((C^{GL_n(E)}_{\theta c}\phi_{n\alpha s})^{(P)}, \pi A_{\theta}) = \mathrm{Tr}(C^{GL_n(E)}_{\theta c}(\phi_{n\alpha s})^{(P)}, \pi A_{\theta}).$$
By Proposition 2.5.1, the twisted trace $\mathrm{Tr}(C^{GL_n(E)}_{\theta c}(\phi_{n\alpha s})^{(P)}, \pi A_{\theta})$ is a sum of terms of the form $\mathrm{Tr}(C^{GL_n(E)}_{\theta c}( \phi_{n_1\alpha s_1} \otimes \ldots \otimes  \phi_{n_1\alpha s_1}  ), \pi A_{\theta})$, where $(s_i)$ ranges over extended compositions of $s$. We assume the term $C^{GL_n(E)}_{\theta c}( \phi_{n_1\alpha s_1} \otimes \ldots \otimes  \phi_{n_1\alpha s_1}  )$ corresponding to the extended composition $(s_i)$ of $s$ is non-zero. Let $ m$ be a semisimple point in
$M(E)$ at which the term $C^{GL_n(E)}_{\theta c}( \phi_{n_1\alpha s_1} \otimes \ldots \otimes  \phi_{n_1\alpha s_1}  )$ does not vanish. Let $m_i \in GL_{n_i}$ be the $\mathrm{i}$-th block of $m$, and let $m'_i $ be the $\mathrm{i}$-th block of $\theta(m)$. By the definition of $\theta$, we have that $|\mathrm{det}(m'_{i})| = |\mathrm{det}(m_{k+1 -i})| ^{-1}$. By the definition of $\theta$-compact, we also have $|\mathrm{det}(m_i m'_{i})| = |\mathrm{det}(m_{k+1 -i} m'_{k+1-i}) | = |\mathrm{det}(m_i m'_{i})|^{-1},$
thus $|\mathrm{det}(m_i m'_{i})| =1$. By applying Lemma 2.4.1, we have $s_{i} = s_{k +1 -i}$ for all $i$. Because $|\mathrm{det}(m_i m'_{i})| =1$, we have $C^{GL_n(E)}_{\theta c}(m) \not = 0  $ if and only if the norms of all eigenvalues of $m \theta(m)$ are equal to $1$, this occurs if and only if the norms of all eigenvalues of $m_i m'_i$ are equal to $1$ for all $i$, which occurs if and only if $(C^{GL_{n_1} \times GL_{n_k}(E)}_{\theta c}\otimes \ldots  C^{GL_{n_{\frac{k-1}2}} \times GL_{n_{\frac{k+3}2}}(E)}_{\theta c} \otimes C^{GL_{\frac{k+1}2}(E)}_{\theta c}) (m_i) \not = 0$. Therefore, the function $C^{GL_n(E)}_{\theta c}( \phi_{n_1\alpha s_1} \otimes \ldots \otimes  \phi_{n_1\alpha s_1}  )$ equals to $(C^{GL_{n_1} \times GL_{n_k}(E)}_{\theta c}(\phi_{n_1\alpha s_1} \otimes  \phi_{n_k\alpha s_k} ))\otimes \ldots  $
$\otimes(C^{GL_{n_{\frac{k-1}2}} \times GL_{n_{\frac{k+3}2}}(E)}_{\theta c}(\phi_{n_{\frac{k-1}2}\alpha s_{\frac{k-1}2}} \otimes \phi_{n_{\frac{k+3}2}\alpha s_{\frac{k+3}2}})) \otimes (C^{GL_{\frac{k+1}2}(E)}_{\theta c}\phi_{n_{\frac{k+1}2}\alpha s_{\frac{k+1}2} }).$

If $k$ is odd, we repeat the above argument. 
\end{proof}

\subsection{The functions $\hat{\chi}^{\theta}_{N} \phi_{n \alpha s} ^{(P)}$}

In this section, we study the spherical functions $\hat{\chi}^{\theta}_{N} \phi_{n \alpha s} ^{(P)}$ (where $P \subset GL_n$ is a $\theta$-stable standard parabolic subgroup) which appear in Proposition 2.2.3.

\begin{proposition}
Let $P = MN$ be a $\theta$-stable standard parabolic subgroup of $GL_n$, and let $(n_i)$ be the corresponding composition of $n$. Write $k$ for the length of the composition $(n_i)$. Then the function $\hat{\chi}^{\theta}_{N} \phi_{n \alpha s} ^{(P)}$ is equal to $$\sum\limits_{(s_i)} q^{C(n_i,s_i)} \phi_{n_1\alpha s_1} \otimes  \phi_{n_2\alpha s_2} \otimes \ldots \otimes  \phi_{n_k\alpha s_k}$$ where the sum ranges over all extended compositions $(s_i)$ of $s$ of length $k$ satisfying

$$ (s_1-s_k) + (s_2-s_{k-1}) + \cdots +  (s_{i}-s_{k+1-i})  > 0 $$

 for all indices $i$ strictly smaller than $k$.

\end{proposition}

\begin{proof}

We use $H_a$ for $a \in \{1, 2, \ldots , n\} $ to denote the $a$-th vector of the canonical basis of the vector space $\mathfrak a_0 = \mathbb R^n$. 
The subset $\Delta_P$ of $\Delta$ is the subset consisting of the roots $\alpha_{n_1} + \alpha_{n_2}+ \cdots +\alpha_{n_i}$ 
for $i \in \{1,2, \ldots, k-1\}$. 
For any root $\alpha  = \alpha_i|_{\mathfrak a_P}$ in $\Delta_P$ we have:

$$\bar \omega_{\alpha}^{GL_n(E)} = (H_1 + \cdots + H_i-\frac in (H_1 + \cdots H_n ))|_{\mathfrak a_P}.$$
Let $m = (m_i)$ be an element of the standard Levi subgroup $M$, we write $\theta((m_i)) = (m'_i)$. The element $m$ satisfies that $\hat{\chi}_N(m)  = 1$ if and only if the absolute 
value $|\bar \omega_{\alpha}^{GL_n(E)}(m)|$ is smaller than $1$ for all roots $\alpha$ in $\Delta_P$. The evaluation at $m = (m_i)$ of the characteristic function $\hat{\chi}_N^{\theta}$ is 
equal to $1$ if and only if $$|\mathrm{det}(m_1 m'_k)  | |\mathrm{det}(m_2 m'_{k-1})  | \ldots  |\mathrm{det}(m_im'_{k +1-i})  |  < |\mathrm{det}(m \theta(m))| ^{\frac{ n_1 + \ldots n_i}n} = 1$$
for $i \in \{1,2, \ldots, k-1\}$. 

Now we determine the function $ \hat\chi^{\theta}_N \phi_{n\alpha s}^{(P)}$: Let $m = (m_i)$ be an element of $M$ and we write $\theta((m_i)) = (m'_i)$. 
We assume $\hat{\chi}^{\theta}_N(m)  = 1$. Let $(s_i)$ be an
extended composition of $s$. We assume $\hat\chi_N^{\theta}(m)\not = 0$ and $\phi_{n_1 \alpha s_1} \otimes\phi_{n_1 \alpha s_1}  \otimes \cdots \phi_{n_k \alpha s_k} (m) \not = 0$. By Lemma 2.4.1, the absolute value $|\mathrm{det}(m'_{k +1-i})|$ is equal to $|\mathrm{det}(m_{k +1-i})|^{-1}$ for all indices $i$, and we have $$|\mathrm{det}(m_1 m'_k)  | |\mathrm{det}(m_2 m'_{k-1})  | \ldots  |\mathrm{det}(m_im'_{k +1-i})  |  < |\mathrm{det}(m \theta(m))| ^{\frac{ n_1 + \ldots n_i}n} = 1$$
for $i \in \{1,2, \ldots, k-1\}$ if and only if $$ (s_1-s_k) + (s_2-s_{k-1}) + \cdots +  (s_{i}-s_{k+1-i})  > 0 $$ for all indices $i \in \{1,2, \ldots, k-1\}$. Conversely, if the extended composition $(s_i)$ satisfies the above condition for all indices $i < k 	$, then any element $ m = (m_i	)$ with $|\mathrm{det}(m_i m'_{k+1-i})| = q^{-\alpha(s_i-s_{k+1 -i})}$ satisfies $ \hat\chi_N^{\theta}(m)= 1$. 
\end{proof}

More generally, let $M = GL_{n_1} \times \ldots \times GL_{n_t }$ be a $\theta$ -stable Levi subgroup of $GL_n$ and let $P=LU$ be a standard parabolic subgroup of $M$ with corresponding composition $(m_i)$. We write $k$ for the length of $(m_i)$, for $1\leq i \leq t$, we define $k_i \in \{1,2, \ldots, k\}$ such that $\sum_{j = k_{i-1}+1}^{j = k_i} m_j= n_i$ (we define $m_0 = 0$, then $k_i$ is the number of blocks of $L = GL_{n_1} \times \ldots \times GL_{n_i}$). 
We consider $\phi_{n_1\alpha s_1} \otimes \ldots \otimes \phi_{n_s \alpha s_t} \in \mathcal H^{unr}( GL_{n_1}) \otimes \ldots \otimes \mathcal H^{unr}(GL_{n_t})$,  for $1\leq i \leq k $, let $(s_{1,j}),\ldots,(s_{t,j})$ be extended compositions of $s_1,\ldots, s_t$ respectively. If $k_{j-1} < i \leq k_j$ (we define $k_0=0$), then we write $s'_i = \sum_{p =1}^{j-1}\sum_{q = 1}^{k_p-k_{p-1}} s_{p,q} +  \sum_{p=1}^{i-k_{j-1}}$. Repeat the above argument, we have

\begin{proposition}

 The function $\hat{\chi}^{\theta}_{M,U} \phi_{n \alpha s} ^{(P)}$ is equal to $$\sum\limits_{(s_{1,j}),\ldots,(s_{t,j})} q^{C(n_i,s_i)} \otimes_{i=1}^k(\otimes_{j = 1}^{k_{i}-k_{i-1}} \phi_{m_i\alpha s_{i,j}} )$$ where the sum ranges over extended compositions $(s_{1,j}),\ldots,(s_{t,j})$ of $s_1,\ldots, s_t$ of $s$ of length $k_{i}-k_{i-1}$ respectively, satisfying

$$ (s'_1-s'_k) + (s'_2-s'_{k-1}) + \cdots +  (s'_{i}-s'_{k+1-i})  > 0 $$
 for all indices $i \in \{1,2,\ldots,k-1\}$ with $i \not = k_1,\ldots,k_t$.

\end{proposition}

\subsection{Computation of some twisted traces}

In this section, we compute the twisted traces of some $\theta$-stable representations, the proposition we get is crucial to establish Theorem 3.3.1 in the next section. 

\begin{definition}
Let $G$ be a split reductive group over p-adic field  $F$ with a fixed Borel subgroup $B = TU$, let $\pi$ be an unramified representation of some Levi subgroup $M$ of $G$ then we write $\epsilon_{M,\pi} \in \hat M$ for the Hecke matrix of this representation (well-defined up to the action of the rational Weyl-group of $T$ in $M$). 
\end{definition}
We recall the definition of the Hecke matrix. Let $\pi$ be an unramified representation of $G(F)$, then there exists a unique unramified character $\sigma$ of $T(F)$ such that $\pi$ is the unique irreducible quotient of $\mathrm{Ind}_{P_0}^G(\sigma)$. 
Let $\hat T$ be the complex torus dual to $T$ . We define $\epsilon_{M,\pi} \in \hat{T}(\mathbb C) $ as follows: For any rational cocharacter $\chi \in X_*(T)$, we evaluate $\sigma \circ \chi $ at a uniformizer $\pi_F$, then we get an element
$\epsilon_{\pi} \in Hom(X_*(T ),\mathbb C^{\times})= X_*(\hat T)\otimes \mathbb C^{\times} =\hat T(\mathbb C)$. 
Thus we have an element $\epsilon_{M,\pi} \in \hat M$ well-defined up to the action of the rational Weyl-group of $T$ in $M$. 

\begin{proposition}

Let $\pi $ be a smooth admissible representation of $GL_n(E)$ ($n \geq 3$) of the following form:
 $$\pi = \mathrm{Ind}_P^{GL_n} (\pi_1(\alpha_1) \otimes \pi_2(\alpha_2) \otimes \ldots \otimes \pi_k (\alpha_k) \otimes \rho  \otimes \pi_k(-\alpha_k) \otimes \ldots \otimes \pi_2(-\alpha_2) \otimes \pi_1(-\alpha_1))$$ such that $\pi_1(\alpha_1) \otimes \pi_2(\alpha_2) \otimes \ldots \otimes \pi_k (\alpha_k) \otimes \rho  \otimes \pi_k(-\alpha_k) \otimes \ldots \otimes \pi_2(-\alpha_2) \otimes \pi_1(-\alpha_1)$ is $\theta$-stable where:
 
\begin{itemize}  

\item $P = M N$ is a $\theta$-stable standard parabolic subgroup of $GL_{n}$ corresponding to the composition $(n_1, n_2, \ldots , n_k , m , n_k, \ldots, n_2, n_1)$ of $n$.
\item $\pi_i(\alpha_i)$ is defined to be $\pi_i \otimes |\mathrm{det} |^{\alpha_i}$ where $\pi$ is an admissible representation of $GL_{n_i}(E)$ and $\alpha_i \in \mathbb C$. 
\item $\pi_i \circ \theta \cong \pi_i$ (Thus $\pi_i (\alpha_i) \circ \theta \cong \pi_i(-\alpha_i)$).

\end{itemize}
  
Let $V_i$ be the underlying vector space of $\pi_i$, let $A_{\theta}$ be an intertwining operator of $\pi$ induced from an intertwining operator of $\pi_1(\alpha_1) \otimes \pi_2(\alpha_2) \otimes \ldots \otimes \pi_k (\alpha_k) \otimes \rho  \otimes \pi_k(-\alpha_k) \otimes \ldots \otimes \pi_2(-\alpha_2) \otimes \pi_1(-\alpha_1)$, defined as follows: 
$$A_{\theta}(v_1 \otimes v_2 \otimes \ldots \otimes v_k  \otimes w  \otimes v'_k \otimes \ldots \otimes v'_2 \otimes v'_1)$$

$$=A_{\theta,1} v'_1 \otimes A_{\theta,2} v'_2 \otimes \ldots \otimes A_{\theta,k} v'_k  \otimes B_{\theta} w  \otimes A_{\theta,k} v_k \otimes \ldots \otimes A_{\theta,2} v_2 \otimes A_{\theta,1} v_1$$ where $A_{\theta,i} \in End_{\mathbb C}(V_i)$ is an intertwining operator of $\pi_i$ and $B_{\theta}$ is an intertwining operator of $\rho$. 

Then we have $$\mathrm{Tr}(C_{\theta c} \phi_{n \alpha s}, \mathrm{Ind}_P^{GL_n} (\pi_1(\alpha_1) \otimes \pi_2(\alpha_2) \otimes \ldots \otimes \pi_k (\alpha_k) \otimes \rho  \otimes \pi_k(-\alpha_k) \otimes \ldots \otimes \pi_2(-\alpha_2) \otimes \pi_1(-\alpha_1)) A_{\theta} ) $$

$$= \mathrm{Tr}(C_{\theta c} \phi_{n \alpha s}, \mathrm{Ind}_P^{GL_n} (\pi_1 \otimes \pi_2 \otimes \ldots \otimes \pi_k \otimes \rho  \otimes \pi_k \otimes \ldots \otimes \pi_2 \otimes \pi_1) A_{\theta} )  $$
\end{proposition}

\begin{proof}
 First, we have $$\mathrm{Tr}(C_{\theta c} \phi_{n \alpha s}, \mathrm{Ind}_P^{GL_n} (\pi_1(\alpha_1) \otimes  \ldots \otimes \pi_k (\alpha_k) \otimes \rho  \otimes \pi_k(-\alpha_k) \otimes \ldots \otimes \pi_1(-\alpha_1)) A_{\theta}  ) $$
 
 $$ = \mathrm{Tr}(C_{\theta c} \phi_{n \alpha s} ^{{(P)}}, (\pi_1(\alpha_1) \otimes  \ldots \otimes \pi_k (\alpha_k) \otimes \rho  \otimes \pi_k(-\alpha_k) \otimes \ldots \otimes \pi_1(-\alpha_1)) A_{\theta}  ) $$ by Proposition 2.2.3. By applying Proposition 2.5.2, we have that the above equality equals

$$\sum\limits_{ (s_i) } q^{C(n_i,s_i)} \mathrm{Tr}( (C_{\theta c}(\phi_{n_1\alpha \alpha_1} \otimes \ldots \otimes  \phi_{m \alpha t} \otimes  \ldots \otimes \phi_{n_1\alpha \alpha_1} ),$$ 
$$(\pi_1(\alpha_1) \otimes  \ldots \otimes \pi_k (\alpha_k) \otimes \rho  \otimes \pi_k(-\alpha_k) \otimes \ldots \otimes \pi_1(-\alpha_1)) A_{\theta})$$

$$=\sum\limits_{ (s_i) } q^{C(n_i,s_i)}( \prod_{i = 1}^k\mathrm{Tr} (C_{\theta c}(\phi_{n_i\alpha s_i} \otimes \phi_{n_i\alpha s_i} ),(\pi_i(s_i) \otimes  \pi_i(-s_i)) (A_{\theta,i} \otimes A_{\theta,i}) ))\mathrm{Tr} (C_{\theta c}\phi_{m\alpha t} , \rho B_{\theta}) $$ 
where the sum ranges over extended partitions $ s = s_1 + \ldots s_k + t + s_k + \ldots+ s_1 $ satisfying the property described in Proposition 2.6.2. 

Let $G_i $ denote $GL_{n_i} \times GL_{n_{ k -i}} = GL_{n_i} \times GL_{n_i}$, and we use $\mathcal{P}^{\theta,i}$ to denote the set of $\theta$-stable standard parabolic subgroups of $G_i$. By applying Proposition 2.2.2, we obtain $$\mathrm{Tr} (C_{\theta c}(\phi_{n_i\alpha s_i} \otimes \phi_{n_i\alpha s_i} ),(\pi_i(\alpha_i) \otimes  \pi_i(-\alpha_i)) (A_{\theta,i} \otimes A_{\theta,i} )) $$

$$= \sum\limits_{P_i \in \mathcal P^{\theta,i}, P_i = M_iN_i}\varepsilon_P \mathrm{Tr}(\hat{\chi}^{\theta}_{N_i}(\phi_{n_i\alpha s_i} \otimes \phi_{n_i\alpha s_i})^{(P)}, (\pi_i(\alpha_i) \otimes  \pi_i(-\alpha_i))_{N_i}(\delta_{G_i, P_i}^{-\frac 12})(A_{\theta,i} \otimes A_{\theta,i} )).$$ 

We fix a $P_i = M_iN_i \in \mathcal P^{\theta}$, let $G_{i,1} \times G_{i,2}$ denote $GL_{n_i} \times GL_{n_{ k -i}} = GL_{n_i} \times GL_{n_i}$ (where $\pi_i(\alpha_i)$ is a representation of $G_{i,1}(E)$ and $\pi_i(-\alpha_i)$ is a representation of $G_{i,2}(E)$). We write $P_{i,1} = N_{i.1} M_{i,1} = P_i \cap G_{i,1}$ and $P_{i,2} = N_{i.2} M_{i,2} = P \cap G_{i,2}$. 

 Let $J_{1,1}, \ldots, J_{t,1}$ (resp. $J_{1,2}, \ldots, J_{t,2}$) be the simple sub-representations of the semi-simplification of $ ({\pi_i})_{N_{i,1}}$ which is fixed by $A_{\theta,i}  $ and unramified. We write $\epsilon_{M_i,J_{l,1}} = \epsilon_{M_i,J_{l,2}}  = (q^{h_{l,1}},\ldots,q^{h_{l,n_l}})$ (also denoted by $\epsilon_{M,J_{l}}$), then $\epsilon_{M,J_l(\pm \alpha_i)} = (q^{h_{l,1} \pm \alpha_i},\ldots,q^{h_{l,n_l}\pm \alpha_i})$.  Note that $\hat{\chi}^{\theta}_N(\phi_{n_i\alpha \alpha_i} \otimes \phi_{n_i\alpha \alpha_i})^{(P)}$ is of the form $\sum_{j \in J} f_{j,1} \otimes f_{j,2}$ for some index set $J$ where $f_{j,1}, f_{j,2} \in \mathcal H^{unr}(GL_{n_i}(E))$ whose Satake transfers are homogeneous of degree $s_i$ (as described in Proposition 2.6.2). Then we have that $$\mathrm{Tr}(\hat{\chi}^{\theta}_N(\phi_{n_i\alpha s_i} \otimes \phi_{n_i\alpha s_i})^{(P)}, (\pi_i(\alpha_i) \otimes  \pi_i(-\alpha_i))_N(\delta_{G_i, P_i}^{-\frac 12})(A_{\theta,i} \otimes A_{\theta,i} ))$$

$$=\sum_{j \in J} \mathrm{Tr}(f_{j,1} \otimes f_{j,2}, (\pi_i(\alpha_i) \otimes  \pi_i(-\alpha_i))_N(\delta_{G_i, P_i}^{-\frac 12})(A_{\theta,i} \otimes A_{\theta,i}) )$$

 We write $J_{ j,1}^{K \cap M_{i,1}(E)} = J_{j,2}^{K \cap M_{i,2}(E)} = \mathbb C v $, we also write $v \otimes v = v_1 \otimes v_2$ where $v_1 \in J_{1, j}^{K \cap M_{i,1}(E)}, v_2 \in J_{2, j}^{K \cap M_{i,2}(E)} $. To prove our proposition, we only need to show that
 
$$\mathrm{Tr}(f_{j,1} \otimes f_{j,2}, (J_{k,1} (\alpha_i) \otimes  J_{k,2}(-\alpha_i)_N(\delta_{G_i, P_i}^{-\frac 12})(A_{\theta,i} \otimes A_{\theta,i}) ) $$
$$= \mathrm{Tr}(f_{j,1} \otimes f_{j,2}, (J_{k,1}\otimes  J_{k,2})_N(\delta_{G_i, P_i}^{-\frac 12})(A_{\theta,i} \otimes A_{\theta,i} ))$$
 for all $i, j ,k$. 

We have $$\mathrm{Tr}(f_{j,1} \otimes f_{j,2}, (J_{k,1} (\alpha_i) \otimes  J_{k,2}(-\alpha_i)_N(\delta_{G_i, P_i}^{-\frac 12})(A_{\theta,i} \otimes A_{\theta,i} )) $$

$$= \mathrm{Tr}(v_1 \times v_2 \mapsto \int_{G_1(E) \times G_2(E)} (f_{j,1} \otimes f_{j,2})(J_{k,1} (\alpha_i) \otimes  J_{k,2}(-\alpha_i))_N(\delta_{G_i, P_i}^{-\frac 12})(A_{\theta,i} \otimes A_{\theta,i}) )(v_2 \otimes v_1 )dg_1 dg_2  $$
(we write $(A_{\theta,i} \otimes A_{\theta,i}) )(v_2 \otimes v_1 )dg_1 dg_2 $ because $ A_{\theta,i}$ is an isomorohism from $ \pi_{i}(-\alpha_i) \circ \theta$ to $ \pi_{i}(\alpha_i)$, and $ A_{\theta,i}$ is also an isomorphism from $ \pi_{i}(\alpha_i) \circ \theta $ to $ \pi_{i}(-\alpha_i)$)

$$= (\mathrm{Tr}(v  \mapsto \int_{G_1(E) } f_{j,1}{J_{k,1} (\alpha_i)}{(\delta_{G_i, P_i}^{-\frac 12}}|_{M_{i,1}(F)})A_{\theta,i} v dg_1) ) \times $$
$$(\mathrm{Tr}(v  \mapsto \int_{G_2(E) } f_{j,2}{J_{k,2} (-\alpha_i)}{(\delta_{G_i, P_i}^{-\frac 12}}|_{M_{i,2}(F)})A_{\theta,i} v dg_2))$$
(recall that we write $P_{i,1} = N_{i.1} M_{i,1} = P \cap G_1$ and $P_{i,2} = N_{i.2} M_{i,2} = P \cap G_2$)
$$=\varepsilon_{A_{\theta,i}} (\mathrm{Tr}(v  \mapsto \int_{G_1(E) } f_{j,1}{J_{k,1} (\alpha_i)}{(\delta_{G_i, P_i}^{-\frac 12}}|_{M_{i,1}(F)}) v dg_1) )\times$$
$$(\mathrm{Tr}(v  \mapsto \int_{G_2(E) } f_{j,2}{J_{k,2} (-\alpha_i)}{(\delta_{G_i, P_i}^{-\frac 12}}|_{M_{i,2}(F)})v dg_2))$$
(where $\varepsilon_{A_{\theta,i}} = \pm 1$ only depends on $A_{\theta,i}$, see Proposition 2.2.2)
$$ = \varepsilon_{A_{\theta,i}} f_{j,1}(\epsilon_{M,J_{l,1}(\alpha_i)} ) f_{j,1}(\epsilon_1) f_{j,1}(\epsilon_{M,J_{l,2}(-\alpha_i)} ) f_{j,1}(\epsilon_2)$$
(We use $\epsilon_1$ and $\epsilon_2$ to denote the Hecke matrix of ${\delta_{G_i, P_i}}^{-\frac 12}|_{M_{i,1}(F)}$ and ${\delta_{G_i, P_i}}^{-\frac 12}|_{M_{i,2}(F)}$ respectively)
$$ = \varepsilon_{A_{\theta,i}} f_{j,1}(q^{h_{l,1} + \alpha_i},\ldots,q^{h_{l,n_l}+ \alpha_i} ) f_{j,1}(\epsilon_1) f_{j,2}(q^{h_{l,1}-\alpha_i},\ldots,q^{h_{l,n_l}- \alpha_i} ) f_{j,1}(\epsilon_2)$$
$$ = \varepsilon_{A_{\theta,i}} q^{\alpha_i s_i}f_{j,1}(q^{h_{l,1} },\ldots,q^{h_{l,n_l}} ) f_{j,1}(\epsilon_1) q^{- \alpha_i s_i} f_{j,1}(q^{h_{l,1} },\ldots,q^{h_{l,n_l}} ) f_{j,1}(\epsilon_2)$$
(because the Satake transforms of $f_{j,1}, f_{j,2}$ are homogeneous of degree $s_i$)
$$ = \varepsilon_{A_{\theta,i}} f_{j,1}(q^{h_{l,1} },\ldots,q^{h_{l,n_l}} ) f_{j,1}(\epsilon_1)  f_{j,1}(q^{h_{l,1} },\ldots,q^{h_{l,n_l}} ) f_{j,1}(\epsilon_2)$$
$$ =\varepsilon_{A_{\theta,i}} (\mathrm{Tr}(v  \mapsto \int_{G_1(E) } f_{j,1}{J_{k,1}}{(\delta_{G_i, P_i}^{-\frac 12}}|_{M_{i,1}(F)}) v dg_1) )(\mathrm{Tr}(v  \mapsto \int_{G_2(E) } f_{j,2}{J_{k,2}}{(\delta_{G_i, P_i}^{-\frac 12}}|_{M_{i,2}(F)})v dg_2)) $$
$$= (\mathrm{Tr}(v  \mapsto \int_{G_1(E) } f_{j,1}{J_{k,1} }{(\delta_{G_i, P_i}^{-\frac 12}}|_{M_{i,1}(F)})A_{\theta,i} v dg_1) )(\mathrm{Tr}(v  \mapsto \int_{G_2(E) } f_{j,2}{J_{k,2} }{(\delta_{G_i, P_i}^{-\frac 12}}|_{M_{i,2}(F)})A_{\theta,i} v dg_2))$$
$$ =  \mathrm{Tr}(f_{j,1} \otimes f_{j,2}, (J_{k,1}\otimes  J_{k,2})_N(\delta_{G_i, P_i}^{-\frac 12})(A_{\theta,i} \otimes A_{\theta,i} )).$$
Thus, we have proved $$\mathrm{Tr}(f_{j,1} \otimes f_{j,2}, (J_{k,1} (\alpha_i) \otimes  J_{k,2}(-\alpha_i)_N(\delta_{G_i, P_i}^{-\frac 12})(A_{\theta,i} \otimes A_{\theta,i}) ) $$
$$= \mathrm{Tr}(f_{j,1} \otimes f_{j,2}, (J_{k,1}\otimes  J_{k,2})_N(\delta_{G_i, P_i}^{-\frac 12})(A_{\theta,i} \otimes A_{\theta,i} )),$$ and the proposition follows. 
\end{proof}

\section{Twisted traces of $\theta$-stable semi-stable rigid representations}

In this section, we compute the twisted compact traces of the function of Kottwitz (which is introduced in $\S$2.1 in \cite{Kot84}, which we also define in $\S4.1$) against certain $\theta$-stable representations of $GL_n(F)$ where $F$ is a $p$-adic field.

\subsection{$\theta$-stable semi-stable rigid representations}

Let $E/\mathbb Q$ be a quadratic extension in which $p$ is inert and unramified. Then $E_p/\mathbb Q_p$ is a quadratic extension such that $p$ remains a uniformizer in $E_p$.  Let $\mathcal O_{E_p}$ be the ring of integer of $E_p$. We use $G_n$ to denote $GL_n(E_p)$. Let $\theta$ be an involution of $G_n$ defined by $g \mapsto A_n^{-1}(\bar g ^t)^{-1}A_n)$, where $A_n$ is the anti-diagonal matrix with entries $1$ and $g \mapsto \bar g^t$ is the conjugation transpose with respect to the non-trivial automorphism of $\mathrm{Gal}(E/\mathbb Q)$. We normalize the $p$-adic norm on $E_p$ such that $|p| = p^{-2}$. 

Let $x, y, n \in \mathbb{Z}_{>0} $ such that $n = xy$ (with $n \geq 3$). We define the representation $\mathrm{Speh}(x, y)$ of $GL_n(E_p)$ to be the Langlands quotient of the representation $$| \mathrm{det} |^{\frac{y-1}2} \mathrm{St}_{GL_x(E_p)} \times | \mathrm{det} |^{\frac{y-3}2} \mathrm{St}_{GL_x(E_p)} 
 \times \ldots \times | \mathrm{det} |^{\frac{1-y}2} \mathrm{St}_{GL_x(E_p)} $$ where the product means unitary parabolic induction from the standard parabolic subgroup of $GL_n$ with $y$ blocks and each block of size $x$. A representation $\pi_p$ of $GL_n$ is called semi-stable rigid if it is unitarizable, and it is isomorphic to a representation of the form $\mathrm{Ind}_P^G (\otimes_{i = 1}^k \mathrm{Speh}(x_i, y)(\varepsilon_i| \cdot |^{e_i}))$,  where $P = MN$ is the parabolic subgroup corresponding to the composition $(yx_a)$ of $n$ and where the tensor product is taken along the blocks $M = \prod_{a = 1}^k G_{yx_a}$ and
 
\begin{itemize}
\item $k \in \{1, 2, \ldots, n\}$; 
\item $\varepsilon_i$ is a unitary unramified character $\varepsilon_i: GL_x \rightarrow \mathbb C^{\times}$ (thus $\varepsilon_k$ is of the form $g \mapsto |\mathrm{det}(g)|^{\mathbf id_i }$ for some $d_i \in \mathbb R$ ) for each $i \in \{1, 2, \ldots, k\}$;
\item	 $e_i$ is a real number $e_i$ in the open interval $(-\frac 12, \frac 12)$ for each $i \in \{1, 2, \ldots, k\}$;
\item positive integers $y, x_1, x_2, \ldots, x_k $ such that $\frac ny = \sum_{ i = 1 }^k x_i$. 
\end{itemize}
We remark that these representations are irreducible (see Theorem 7.5 of \cite{Tad86}).

\begin{definition}
We say that a semi-stable rigid representation is of $\theta$-type if it is of form $\mathrm{Ind}_P^G (\otimes_{i = 1}^k \mathrm{Speh}(x_i, y)(\varepsilon_i))$, where for each $i$, the character $\varepsilon_i: GL_x \rightarrow \mathbb C^{\times}$ is a $\theta$-stable unitary unramified character (thus $\varepsilon_i$ is either the trivial character or the unramified quadratic character). 

\end{definition}

\begin{theorem}
Let $\pi$ be a discrete automorphic representation of $GL_n(\mathbb A_E)$, and we assume that $\pi_p$ has an Iwahori fixed vector (as deinfed in $\S 2.1$). Then $\pi_p$ is a semi-stable rigid representations. 
\end{theorem}

\begin{proof}

See Theorem 2 of \cite{Kret11}. 
\end{proof}

\begin{corollary}
Let $\pi = \mathrm{Ind}_P^G (\otimes_{i = 1}^k \mathrm{Speh}(x_i, y)(| \cdot |^{\mathbf i d_i + e_i}))$ be a $\theta$-stable semi-stable rigid representation of $GL_n(E_p)$, then we have $$\mathrm{Tr}(C_{\theta c} \phi_{n \alpha s},  \mathrm{Ind}_P^G (\otimes_{i = 1}^k \mathrm{Speh}(x_i, y)(| \cdot |^{\mathbf i d_i + e_i})) A_{\theta})  $$
$$ = \pm \mathrm{Tr}(C_{\theta c} \phi_{n \alpha s},  \mathrm{Ind}_P^G (\otimes_{i = 1}^k \mathrm{Speh}(x_i, y)(\varepsilon_i) A'_{\theta})  $$ for some intertwining operator $A'_{\theta}$, where $\varepsilon_i$ is either the trivial character or the unramified quadratic character. 
\end{corollary}

\begin{proof}
We assume that $\pi$ is of the form $\mathrm{Ind}_P^G (\otimes_{i = 1}^k \mathrm{Speh}(x_i, y)(| \cdot |^{\mathbf i d_i + e_i}))$, then $\pi \circ \theta $ is of the form $$\mathrm{Ind}_{\theta(P)}^G (\otimes_{i = 1}^k \mathrm{Speh}(x_{k-i}, y)(| \cdot |^{-\mathbf i d_{k-i}-e_{k-i}})) \cong \mathrm{Ind}_{P}^G (\otimes_{i = 1}^k \mathrm{Speh}(x_{i}, y)(| \cdot |^{-\mathbf i d_{i}-e_{i}})).$$ If there exists $j$ such that $| \cdot |^{\mathbf i d_i + e_i}$ is not $\theta$-stable, then there exists $j'$ such that $x_{j'} = x_{j} $ and $| \cdot |^{-\mathbf i d_j-e_j} = | \cdot |^{\mathbf i d_{j'} + e_{j'}}$, because $$ \mathrm{Ind}_{P}^G (\otimes_{i = 1}^k \mathrm{Speh}(x_{i}, y)(| \cdot |^{-\mathbf i d_{i}-e_{i}})) \cong \mathrm{Ind}_P^G (\otimes_{i = 1}^k \mathrm{Speh}(x_i, y)(| \cdot |^{\mathbf i d_i + e_i}))$$ implies that $$(\mathrm{Speh}(x_{i}, y)(| \cdot |^{-\mathbf i d_{i}-e_{i}}))_{1 \leq i\leq k} = (\mathrm{Speh}(x_{i}, y)(| \cdot |^{\mathbf i d_{i}+e_{i}}))_{1 \leq i\leq k}$$ up to permutation, see Theorem 7.5 of \cite{Tad86}).  

Let $\pi_i$ denote $\mathrm{Speh}(x_i, y)$, and let $s_i$ denote $ \mathbf i d_i + e_i$, after reordering $i$, we may assume $\pi = \mathrm{Ind}_P^G(\pi_1(s_1) \otimes  \ldots \pi_t(s_t) \otimes \rho \otimes \pi_t(-s_t) \otimes \ldots \pi_1(-s_1))$, where $\rho$ is the induced representation of $\mathrm{Speh}(x_i, y)(| \cdot |^{\mathbf i d_i + e_i}))$'s for which $ (| \cdot |^{\mathbf i d_i + e_i}))$ is $\theta$-stable (thus $\rho$ is itself $\theta$- stable). We may choose an intertwining operator $A'_{\theta}$ of $\pi = \mathrm{Ind}_P^G(\pi_1(s_1) \otimes  \ldots \pi_t(s_t) \otimes \rho \otimes \pi_t(-s_t) \otimes \ldots \pi_1(-s_1))$ as described in Proposition 2.7.1, then we have 
$$\mathrm{Tr}(C_{\theta c} \phi_{n \alpha s},  \mathrm{Ind}_P^G (\otimes_{i = 1}^k \mathrm{Speh}(x_i, y)(| \cdot |^{\mathbf i d_i + e_i})) A_{\theta})  $$
$$=\pm \mathrm{Tr}(C_{\theta c} \phi_{n \alpha s},  \mathrm{Ind}_P^G(\pi_1(s_1) \otimes  \ldots \pi_t(s_t) \otimes \rho \otimes \pi_t(-s_t) \otimes \ldots \pi_1(-s_1))A'_{\theta})  $$
$$ = \pm \mathrm{Tr}(C_{\theta c} \phi_{n \alpha s}, \mathrm{Ind}_P^G(\pi_1 \otimes  \ldots \pi_t \otimes \rho \otimes \pi_t \otimes \ldots \pi_1)A'_{\theta}) .  $$
\end{proof}

\subsection{Ring of Zelevinsky and segments}

In this section, we collect some constructions and facts from \cite{Zel80}.
Let $F$ be a p-adic field, we normalize the $p$-adic norm on $F$ so that $|\pi_F| = p^{-r}$, where $\pi_F$ is a uniformizer and $p^r$ is the number of element in the residue field. We use $G_n$ to denote $GL_{n}(F)$. Let $ m,m' \in \mathbb Z_{\geq 1}$, and let $\pi$ (resp. $\pi'$) be a smooth admissible representation of $G_m$ (resp. $G_{m'}$ ), then we write $\pi \times \pi$ for the representation of $G_{m+ m'}$ unitary parabolically
induced from the representation $\pi \otimes \pi'$ of the standard Levi
subgroup consisting of two blocks, one of size $m$, and the other one of size $m'$.
We denote by $\mathcal R$ the direct sum $\bigoplus_{n \in \mathbb Z_{\geq 0}} \mathrm{Groth}(G_n)$, where $G_0$ is defined to be the trivial group. The group $G_0$ has one unique irreducible representation $\sigma_0$ (the trivial representation). The operation direct sum together with $\times$ turns the vector space $\mathcal R$ into a commutative $\mathbb C$-algebra with $\sigma_0$ as unit element. We call it \textbf{ring of Zelevinsky}.

We fix a set $S$. A \textbf{multi-set} on $S$ is defined as a function $\chi: S \rightarrow \mathbb Z_{>0}$. Note that subsets of $S$ may be represented by their characteristic functions, and thus the notion of a multi-set on $S$ generalises that of a subset of $S$. We also write down the multi-set $\chi: S \rightarrow \mathbb Z_{>0}$ as $a= \{\ldots, x, \ldots, x, y, \ldots, y, \ldots \}$ where each element $x \in S$ is repeated $\chi(x)$ times (when given $a= \{\ldots, x, \ldots, x, y, \ldots, y, \ldots \}$, we also denote the associated function by $\chi_a$). And $\chi(x)$ is called a multiplicity of $x$ in $S$. We write $x \in a$ if $\chi_a(x) > 0$, and $a \subset b$ if $\chi_a(x) \leq \chi_b(x)$ for all $x\in S$; the empty multi-set $at$ corresponds to $\chi_a= 0$ . A multi-set $a$ is called finite if $\chi$ has finite support.

We write $\nu$ for the absolute value morphism from $F^{\times} $ to $\mathbb C^{\times}$. A segment $S = \langle x,y\rangle$ (where $x, y \in \mathbb Q$ and $y-x \in \mathbb Z$), is defined as follows: 
\begin{itemize}
\item We define $\langle x,y\rangle$ to be $\{x,x+1,\ldots,y\}$ if $y \geq x$;
\item We define $\langle x,y\rangle$ to be the empty set if $y < x$; 
\end{itemize}
The segment $\{x\}$ has length $1$, the segment $\{x, x + 1\}$ has length $2$, etc. For any segment $\langle x,y\rangle$ with $y \geq x$ we write $\Delta\langle x,y\rangle$ for the induced representation $\nu^x \times \nu^{x+1} \times \ldots \times \nu^y$. We define $\Delta\langle x,y\rangle$ to be $0$ in case $ y \leq x-1$. For any segment $S$ of non-negative length the object $\Delta S$ is a representation of the group $GL_n(F)$, where $n$ is the length of $ S$. We define $Q\langle x,y\rangle$ to be the unique irreducible quotient of $\Delta\langle x,y \rangle$ as introduced in \cite{Zel80}. We define $b(\Delta\langle x, y \rangle)$ (resp. $b (Q\langle x, y \rangle)$) to be $x$ and $e(\Delta\langle x, y \rangle)$ (resp. $e (Q\langle x, y \rangle)$) to be $y$. 

\begin{example}
The Steinberg representation $\mathrm{St}_{G_n}$ can be defined as $Q\langle \frac{1-n}2, \frac{n-1}2 \rangle$. 
\end{example}

Let $S = \langle x,y \rangle$ be a segment, we use $-S$ to denote $\langle -y, -x\rangle $. Let $\mathcal S $ denote the set of all segments. And segments $S, S'$ are called \textbf{linked} if $S \not\subset S', S' \not\subset S$ and $S \cup S'$ is a segment. We say that $S$ \textbf{precedes} $S'$ if $S$ and $S'$ are linked and $b(S')  = b(S) + k$ for some $k \in \mathbb Z_{>0}$. And we have the following theorem: 

\begin{proposition}

Let $S_1, \ldots, S_m$ be segments, then we have

(a) $Q(S_1) \times Q(S_2), \ldots, \times Q(S_m)$ is irreducible if and only if for each $i,j = 1, \ldots, m$, the segments $S_i$ and $S_j$ are not linked. 

(b) Suppose that for each pair of indices $i,j$ such that $i < j$, $S_i$ does not precede $S_j$, then the representation $Q(S_1) \times Q(S_2), \ldots, \times Q(S_m)$ has a unique irreducible quotient, denoted by $Q(S_1, \ldots, S_m)$. 

(c) $Q(S_1, \ldots, S_n)$ is isomorphic to $Q(S'_1, \ldots, S'_m)$ if and only if $m = n$ and there exists $w \in S(n)$ such that $(S_1, \ldots, S_n) = (S'_{w(1)}, \ldots,S'_{w(n)} )$ (we use $S(n)$ to denote the permutation group on $\{1,2,\ldots,n\}$). 
\end{proposition}

\begin{proof}
See chapter 4 and chapter 6 of \cite{Zel80}. 
\end{proof}

\begin{example}
The Speh representation $\mathrm{Speh}(x,y)$ of $G_{xy}$ can be defined as $Q(\langle \frac{1-x}2 + \frac{y-1}2, \frac{x-1}2 + \frac{y-1}2 \rangle, \langle \frac{1-x}2 + \frac{y-3}2, \frac{x-1}2 + \frac{y-3}2  \rangle, \ldots, \langle  \frac{1-x}2 + \frac{1-y}2, \frac{x-1}2 + \frac{1-y}2  \rangle)$. 
\end{example}

Let $\mathcal B$ denote representations of the form $Q(S_1 \times \ldots \times S_m)$ where $S_1, \ldots, S_m \in \mathcal S$ such that for each pair of indices $i,j$ such that $i < j$, $S_i$ does not precede $S_j$. Let $a, b \in \mathcal B$. An elementary operation on $a$ is defined as follows: Replace pair $\{S, S'\}$ of linked segments by the pair $\{S^{\cup} =S \cup S', S^{\cap} =S \cap S'\}$. We write $b < a$ if $b$ may be obtained from $a$ by a chain of elementary operations. By Lemma 6.7 of \cite{Zel80}, if $b < a$ then the number of pairs of linked segments of $b$ is less than that of $a$. And we have the following: 

\begin{proposition}
Let $S_1, \ldots, S_m $ be segments and write $a = Q(S_1, \ldots, S_m)$. 
Then we have $Q(S_1) \times Q(S_2), \ldots, \times Q(S_m) = \sum_{b \in \mathcal B} m(a,b)b$ in $\mathcal R$, where \begin{itemize} 

\item $m(a,b) \geq 0$ and $m(a,b) >0 $ if and only if $b < a$. 

\item $m(a,a) = 1$. 

\end{itemize}
\end{proposition}

\begin{proof}
For the proof, see chapter 7 of \cite{Zel80}. 
\end{proof}

\subsection{Twisted traces of representations of $\theta$-type}

In the following, let $ F = E_p$, and let $\theta$ be an involution of $G_n$ as defined at the beginning of this section. Note that $Q \langle x, y \rangle \circ \theta $ is a quotient of $\Delta\langle x, y \rangle \circ \theta $, and since $\Delta\langle x, y \rangle \circ \theta \cong \Delta\langle -y, -x \rangle$, we have  $Q\langle x, y \rangle \circ \theta \cong  Q\langle -y, -x \rangle  $. We also have $(Q(S_1) \times \ldots \times Q(S_m) ) \circ \theta \cong Q(-S_m)\times \ldots \times Q(-S_1) $, and similarly , we have $Q(S_1, \ldots, S_m) \circ  \theta \cong Q(-S_m, \ldots,-S_1) $. We define $\mathcal S^{\theta}$ to be $\{s \in \mathcal S:  s  = -s\}$, and we define $\mathcal B^{\theta}$ to be $\{b \in \mathcal B:  b \circ \theta  \cong b\}$

 \begin{definition}
 We define the representation $\rho(x,y)$ of $G_{xy}$ as follows: \begin{itemize}
 
 \item $\mathrm{St}_{G_{x + y -1}} \times \mathrm{St}_{G_{x + y-3}} \ldots \times \mathrm{St}_{G_{x-y + 1}}$ if $x \geq y$. 
  
 \item $\mathrm{St}_{G_{x + y -1} }\times  \mathrm{St}_{G_{x + y-3}} \ldots \times \mathrm{St}_{G_{y-x + 1}}$ if $x < y$. 
 
 \end{itemize}
 
 \end{definition}

We note that $ \rho(x,y)$ is irreducible, and we may also write $\rho(x,y)$ as follows:

\begin{itemize}
\item $Q(\langle -\frac{x+y}2 + 1,  \frac{x + y}2-1 \rangle, \langle -\frac{x+y}2 + 3,  \frac{x + y}2-3 \rangle, \ldots, \langle -\frac{y-x}2 ,  \frac{x-y}2  \rangle)$  if $x \geq y$.
\item $Q(\langle -\frac{x+y}2 + 1,  \frac{x + y}2-1 \rangle, \langle -\frac{x+y}2 + 3,  \frac{x + y}2-3 \rangle, \ldots, \langle -\frac{x-y}2 ,  \frac{y-x}2  \rangle)$  if $x < y$.
\end{itemize}

\begin{proposition}
$\rho(x,y) < \mathrm{Speh}(x,y)$ in $(\mathcal B,<)$. 
\end{proposition}

\begin{proof}
Let $S_1, S_2, \ldots, S_n $ denote $\langle \frac{1-x}2 + \frac{y-1}2, \frac{x-1}2 + \frac{y-1}2 \rangle, \langle \frac{1-x}2 + \frac{y-3}2, \frac{x-1}2 + \frac{y-3}2  \rangle, \ldots, \langle  \frac{1-x}2 + \frac{1-y}2, \frac{x-1}2 + \frac{1-y}2 \rangle$ respectively, then we have $ \mathrm{Speh} (x, y ) = Q(S_1, \ldots, S_n)$. 

We define $U_{1,1}$ to be $ S_1$, and we define $(U_{1,k + 1}, I_{1,k+ 1})$ to be $(U_{1, k } \cup S_{k + 1 }, U_{1, k } \cap S_{k + 1})$ for $1 \leq k \leq n-1$, then we get segments $ (U_{1,n}, I_{1,2}, \ldots, I_{1,n}) < (S_1 , \ldots, S_n)$ with $U_{1,n} = \langle -\frac{x+y}2 + 1,  \frac{x + y}2-1 \rangle$. 

Then we define $U_{2,2}$ to be $I_{1,2}$, and we define $(U_{2,k + 1}, I_{2,k+ 1})$ to be $(U_{2, k } \cup I_{2,k + 1 }, U_{2, k } \cap I_{2,k + 1})$ for $2 \leq k \leq n-1$, then we get segments $ (U_{1,n}, U_{2,n}, I_{2,3} \ldots, I_{2,n}) < (U_{1,n}, I_{1,2}, \ldots, I_{1,n})$ with $U_{1,n} = \langle -\frac{x+y}2 + 1,  \frac{x + y}2-1 \rangle, U_{2,n} = \langle -\frac{x+y}2 + 3,  \frac{x + y}2-3 \rangle$. 

Repeat the above process, we obtain $\rho(x,y) < \mathrm{Speh}(x,y)$. 

\end{proof}

 \begin{lemma}
We consider $b \in \mathcal B$ which is $\theta$-stable and satisfies that $b < \mathrm{Speh}(x,y)$, and we write $b = Q( S_{1} \times S_{2} \ldots \times S_{n} )$ with $S_{i} \in \mathcal S$. 
Then we have that $ b = \rho(x,y)$ if and only if there exists no $ i \not = j$ such that $S_i =-S_j$, which occurs if and only if  $S_i = -S_i$ for all $i$, because of the $\theta$-stableness. 

 \end{lemma}

\begin{proof}
The 'only if' part is trivial, so we prove only the 'if' part:

We may assume that $  1-\frac{x + y}2 \in S_1$ and $  \frac{x + y}2-1 \in S_i$. Note that the multiplicities of $1-\frac{x + y}2, \frac{x + y}2-1$ in $S_1 + \ldots + S_n$ (as multiple-set) are $1$. Then by the $\theta$-stableness of $b$, we must have $S_1 =-S_i$, by the assumption, we must then have $1 = i$, thus $S_1 = \langle -\frac{x+ y }2  + 1, \frac{x+ y}2 -1  \rangle$. 

Note that the multiplicities of $3-\frac{x + y}2, \frac{x + y}2-3$ in $S_2 + \ldots + S_n$ (as multiple-set) are all $1$, we consider the remaining $3-\frac{x + y}2, \frac{x + y}2-3$ in $S_2 + \ldots + S_n$. Without loss of generality, we again assume that $  3-\frac{x + y}2 \in S_2$ and $  \frac{x + y}2-3 \in S_j$, repeat the above argument we have $2 = j $ and $S_2 =\langle -\frac{x+ y }2  + 3, \frac{x+ y}2 -3  \rangle$. 

Repeat the argument, we may conclude that $ b = \rho(x,y)$.
\end{proof}

Let $\pi = \mathrm{Ind}_P^G (\otimes_{i = 1}^k \mathrm{Speh}(x_i, y)(\varepsilon_i))$ be a representation of $\theta$-type, as defined in last section. 
We associate a poset $(\mathcal C_{\pi} ,<)$ to $\pi$ defined as follows: 

\begin{itemize}
\item $\mathcal C_{\pi} := \{ b_1 (\varepsilon_1) \times b_2  (\varepsilon_2)\ldots \times b_k  (\varepsilon_k): b_i \in \mathcal B, \text{ and }b_i < \mathrm{Speh}(x_i,y)\} \cup \{ \pi \}$. 
\item $ b_1 (\varepsilon_1) \times b_2  (\varepsilon_2)\ldots \times b_k  (\varepsilon_k) <  b'_1 (\varepsilon_1) \times b'_2  (\varepsilon_2)\ldots \times b'_k  (\varepsilon_k)$ if there exists some $i$ such that $ b_i < b'_i$ and $ b_j \leq b'_j$ for all $j$. 
\end{itemize}
Note that all elements of $\mathcal C_{\pi}$ are irreducible (see $\S 4.2$ of \cite{Zel77}), and $\pi = \mathrm{Ind}_P^G (\otimes_{i = 1}^k \rho(x_i, y)(\varepsilon_i))$ is the minimal element in $(\mathcal C^{\theta}_{\pi},<)$. We define $\mathcal C^{\theta}_{\pi} \subset \mathcal C_{\pi}$ to be $\{\rho  \in \mathcal C_{\pi}: \rho \text{ is }\theta \text{-stable} \}$.

\begin{proposition}
Let $\rho = Q( b_1 (\varepsilon_1) \times b_2  (\varepsilon_2)\ldots \times b_k  (\varepsilon_k)) \in \mathcal C^{\theta}_{\pi}$, then we have $$\mathrm{Tr}_{\theta}(C_{\theta c}\phi_{n\alpha s} ,  \rho ) = \mathrm{Tr}_{\theta} (C_{\theta c}\phi_{n\alpha s} , \rho_{\mathrm{St}}) + \sum\limits_{\rho' \in \mathcal C_{\pi}^{\theta},\rho'<\rho} N_{\rho, \rho'}  \mathrm{Tr}_{\theta}(C_{\theta c}\phi_{n\alpha s} , \rho') $$ where 
\begin{itemize}
\item $ \rho_{\mathrm{St}} $ is of the form $\mathrm{Ind}_P^G(\otimes_i \mathrm{St}_{GL_{n_i}}(\varepsilon_i))$, where $\varepsilon_i$ is either the trivial character or the unramified quadratic character.

 \item $N_{\rho,\rho'} \in \mathbb Z $ and it depends on $\rho$, $\rho'$ and an intertwining operator which we introduced later in the proof (which is constructed from $\rho$). 

\end{itemize}
\end{proposition}

\begin{proof}

We write $b_i = Q( S_{i,1} \times S_{i,2} \ldots \times S_{i,n_i} )$. There are three possibilities for $\rho = b_1 (\varepsilon_1) \times b_2  (\varepsilon_2)\ldots \times b_k  (\varepsilon_k) \in \mathcal C^{\theta}_{\pi}$:

\begin{itemize}
\item Type I: $b_i \cong b_i \circ \theta$ for all $i$, and $S_{i,j} =-S_{i,j}$ for all $i,j$.
\item Type II: $b_i \cong b_i \circ \theta$ for all $i$, but there exist $i,j$ such that $S_{i,j} \not=- S_{i,j}$ .
\item Type III: There exists $i $ such that $b_i \not \cong b_i \circ \theta$ (then for such a $b_i$, there exists $j \not = i$ such that $b_i (\varepsilon_i) \circ \theta \cong b_j (\varepsilon_j)$). 
\end{itemize} 

If $\rho$ is of Type I, then by the Lemma 3.3.1, the proposition holds.

If $\rho$ is of Type II or III, then for all $1\leq i \leq k$ and $1\leq j \leq n_i$, we have $Q(S_{i,j})(\varepsilon_i) = Q (-S_{r,q})(\varepsilon_r) $ for some $ 1\leq r \leq k, 1\leq q \leq n_r$. Also note that $\varepsilon_i$'s are either trivial character or quadratic character, then we may conclude that $\varepsilon_i = \varepsilon_r$. We order $Q(S_{i,j})(\varepsilon_i)$ into the following form :$$Q(A_1)(\alpha_1), \ldots, Q(A_N)(\alpha_N), Q(B_1) (\beta_1), \ldots, Q(B_M)(\beta_m), Q(-A_N)(\alpha_N), \ldots, Q (-A_1)(\alpha_1) $$ where:

\begin{itemize}
\item $2N  +M = \sum_{j = 1}^k n_j$
\item $B_i = -B_i$ for $1 \leq i \leq M$ (thus $B_i$'s are Steinberg representations).
\item $\alpha_i$'s and $\beta$'s are either the trival character of the unramified quadratic character. 
\end{itemize}

In ring of Zelevinsky $ \mathcal R$, we have
 $$\prod_{i = 1}^{k} \prod_{j = 1}^{n_i}  Q(S_{i,j}) (\varepsilon_i) = \prod _{i = 1}^N Q (A_i) (\alpha_i) \cdot \prod_{i = 1}^M  Q(B_i)(\beta_i) \cdot  \prod _{i = 1}^N Q (-A_i) (\alpha_i).$$
 By proposition 3.2.2, we have $$\prod_{i = 1}^{k} \prod_{j = 1}^{n_i}  Q(S_{i,j}) (\varepsilon_i)  = \rho + \sum_{\rho'< \rho} m_{\rho, \rho'} \rho' ,$$ where $m_{\rho}$ is some integer depending on $\rho $ and $\rho'$.
 
We choose an intertwining operator $A_{\theta}$ on $ \prod _{i = 1}^N Q (A_i) (\alpha_i)\cdot \prod_{i = 1}^M \cdot Q(B_i)(\beta_i) \cdot  \prod _{i = 1}^N Q(-A_i) (\alpha_i) $ as described in Proposition 2.7.1. Then we have $$\mathrm{Tr}(C_{\theta c}\phi_{n\alpha s} ,  (\prod _{i = 1}^N Q (A_i) (\alpha_i) \cdot \prod_{i = 1}^M  Q(B_i)(\beta_i) \cdot  \prod _{i = 1}^N Q (-A_i) (\alpha_i))  A_{\theta})$$
$$ = \mathrm{Tr}_{\theta} (C_{\theta c}\phi_{n\alpha s} , \rho) +  \sum\limits_{\rho' < \rho} m_{\rho, \rho'} \mathrm{Tr}_{\theta}(C_{\theta c}
\phi_{n\alpha s} , \rho')   $$
$$= \mathrm{Tr}_{\theta} (C_{\theta c}\phi_{n\alpha s} , \rho) +  \sum\limits_{\rho' < \rho, \rho' \in \mathcal C_{\pi}^{\theta}} m^{\theta}_{\rho, \rho'} \mathrm{Tr}_{\theta}(C_{\theta c}f
\phi_{n\alpha s} , \rho')  +  \sum\limits_{\rho' < \rho, \rho' \not \in \mathcal C_{\pi}^{\theta}} \frac 12 m_{\rho, \rho'} \mathrm{Tr}_{\theta}(C_{\theta c}f
\phi_{n\alpha s} , (\rho' + \rho' \circ \theta)) $$
$$ +   \sum\limits_{\rho' < \rho, \rho' \in \mathcal C_{\pi}^{\theta}} \frac 12 (m_{\rho, \rho'} 
- m'_{\rho, \rho'}) \mathrm{Tr}_{\theta}(C_{\theta c}f
\phi_{n\alpha s} , (\rho' + \rho')) $$
$$ = \mathrm{Tr}_{\theta} (C_{\theta c}\phi_{n\alpha s} , \rho) +  \sum\limits_{\rho' < \rho, \rho' \in \mathcal C_{\pi}^{\theta}} m^{\theta}_{\rho, \rho'} \mathrm{Tr}_{\theta}(C_{\theta c}f
\phi_{n\alpha s} , \rho'), $$
where $m^{\theta}_{\rho, \rho'}$ and $m'_{\rho, \rho'}$ are defined as follows: We define $m'_{\rho, \rho'}$ to be the number of $\rho'$ which is fixed by $A_{\theta}$ and we label those $\rho'$ which is fixed by $A_{\theta}$ as $\rho'_1, \ldots, \rho'_{m'_{\rho,\rho'}}$. We choose an intertwining operator of $\rho'$ (denoted by $B_{\theta}$), and we use $\mathrm{Tr}_{\theta}(C_{\theta c} \phi_{n\alpha s} , \rho')$ to denote $\mathrm{Tr}(C_{\theta c} \phi_{n\alpha s} , \rho'B_{\theta})$. We use $ B_{\theta,1}, \ldots,B_{\theta, m'_{\rho,\rho'}} $ to denote the induced intertwining operator of $A_{\theta}$ on $\rho'_1, \ldots \rho'_{m'_{\rho,\rho'}}$, then $m^{\theta}_{\rho,\rho'}$ is defined by the following equality: $$\sum_{i = 1}^{m'_{\rho,\rho'}} \mathrm{Tr}(C_{\theta c} \phi_{n\alpha s} , \rho'_iB_{\theta, i}) = m^{\theta}_{\rho, \rho'} \mathrm{Tr}_{\theta}(C_{\theta c}f\phi_{n\alpha s} , \rho')$$
(note that $\mathrm{Tr}(C_{\theta c} \phi_{n\alpha s} , \rho'_iB_{\theta, i}) = \pm \mathrm{Tr}_{\theta}(C_{\theta c}f\phi_{n\alpha s} , \rho')$ for all $i$, see Corollary 2.2.1). By the definition of $m'_{\rho, \rho'}$, we have that the term $\sum\limits_{\rho' < \rho, \rho' \in \mathcal C_{\pi}^{\theta}} \frac 12 (m_{\rho, \rho'} 
- m'_{\rho, \rho'}) \mathrm{Tr}_{\theta}(C_{\theta c}f
\phi_{n\alpha s} , (\rho' + \rho'))  $ equals $0$. 

We may also write $Q(A_i)(\alpha_i) =\mathrm{St}_{GL_{m_i}} (c_i)  $ for some $c_i \in \mathbb C$, then $Q(A_i)(-\alpha_i) =\mathrm{St}_{GL_{m_i}} (-c_i) $. Applying proposition 2.7.1, we have 
 $$\mathrm{Tr}(C_{\theta c}\phi_{n\alpha s} ,  (\prod _{i = 1}^N Q (A_i) (\alpha_i) \cdot \prod_{i = 1}^M Q(B_i)(\beta_i) \cdot  \prod _{i = 1}^N Q (-A_i) (\alpha_i))  A_{\theta})$$
 $$ = \mathrm{Tr}(C_{\theta c}\phi_{n\alpha s} ,  (\prod _{i = 1}^N \mathrm{St}_{GL_{m_i}} (c_i) \cdot \prod_{i = 1}^M  Q(B_i)(\beta_i) \cdot  \prod _{i = 1}^N \mathrm{St}_{GL_{m_i}} (-c_i))  A_{\theta})$$
$$ =\mathrm{Tr}(C_{\theta c}\phi_{n\alpha s} ,  (\prod _{i = 1}^N \mathrm{St}_{GL_{m_i}} \cdot \prod_{i = 1}^M  Q(B_i)(\beta_i) \cdot  \prod _{i = 1}^N \mathrm{St}_{GL_{m_i}})  A_{\theta}) $$
Then we have $$\mathrm{Tr}_{\theta} (C_{\theta c}\phi_{n\alpha s} , \rho) +  \sum\limits_{\rho' < \rho, \rho' \in \mathcal C_{\pi}^{\theta}} m^{\theta}_{\rho, \rho'} \mathrm{Tr}_{\theta}(C_{\theta c}f
\phi_{n\alpha s} , \rho') $$
$$ =\mathrm{Tr}(C_{\theta c}\phi_{n\alpha s} ,  (\prod _{i = 1}^N \mathrm{St}_{GL_{m_i}} \cdot \prod_{i = 1}^M  Q(B_i)(\beta_i) \cdot  \prod _{i = 1}^N \mathrm{St}_{GL_{m_i}})  A_{\theta}),  $$ and the proposition follows.

\end{proof}

Let $\pi = \mathrm{Ind}_P^G (\otimes_{i = 1}^k \mathrm{Speh}(x_i, y)(\varepsilon_i))$ be a representation of $\theta$-type. By repeatedly applying Proposition 3.3.2, we conclude the following theorem:

\begin{theorem} 
We have: $$\mathrm{Tr}_{\theta}(C_{\theta c}\phi_{n\alpha s} ,  \pi )  = \sum \limits_{\rho \in \mathcal C_{\pi}^{\theta}}M_{\pi,\rho}  \mathrm{Tr}_{\theta}(C_{\theta c}f_{n\alpha s} , \rho )$$ where $M_{\pi, \rho} \in \mathbb Z $, and $\rho$ is of the form $\mathrm{Ind}_P^G(\otimes_i \mathrm{St}_{GL_{n_i}}(\varepsilon_i))$, such that $\varepsilon_i$ is either the trivial character or the unramified quadratic character. 
\end{theorem}

\begin{remark}
In the next section, we give a point-counting formula in terms of twisted traces of representations of $\theta$-type. And Theorem 3.3.1 reduces the computation of twisted traces of representations of $\theta$-type to the computation of twisted traces of products of Steinberg representations. We perform this reduction because compared with Steinberg representations, Speh representations (or more generally, ladder presentations) have much more complicated Jacquet modules (for Jacquet modules of Ladder representations, see \cite{Kret12}). Additionally, given an intertwining operator $A_{\theta}$ on a representation of $\theta$-type (denoted by $\pi$), it is hard to track how $A_{\theta}$ acts on $\pi_N$. 
\end{remark}

	\section{The basic stratum of Kottwitz varieties}
	In this section, we study the cohomology of the basic stratum of Kottwitz varieties (as introduced in \cite{Kot92}, which are certain compact Shimura varieties associated to division algebras). In conclusion, we derive a point-counting formula for the basic stratum in terms of twisted traces of products of Steinberg representations. 
	
	\subsection{Kottwitz varieties}
 
	Let $D$ be a division algebra over $\mathbb Q $ equipped with an anti-involution $*$. Let $\overline{\mathbb Q} $ be the algebraic closure of $\mathbb Q$ inside $\mathbb C$. Write $F$ for the centre of $D$, and we embed $F$ into $\overline{\mathbb Q}$. We assume that $F$ is a quadratic imaginary field  extension of $\mathbb Q$ and we assume that $*$ induces the complex conjugation on $F$.
	
	Let $n$ be the positive integer such that $n^2$ equals the dimension of $D$ over $F$. Let $G$ be the  $\mathbb Q$-group, such that for each commutative $\mathbb Q$-algebra $R$ the group $G(R)$ is defined as follows:
	
	$$G(R):=\{x \in D \otimes_{\mathbb Q} R: xx^* \in R^{\times}\}$$
	
	The mapping $c: G \rightarrow \mathbb G_{m,\mathbb Q}$, defined by $x\mapsto xx^*$, is referred to as the 'factor of similitudes'. Let $h_0$ be an algebra morphism $h_0: \mathbb C \rightarrow D_{\mathbb R}$ such that $h_0(z)^* = h_0(\overline z)$ for all $z \in \mathbb C$. We assume that the involution of $D_{\mathbb R	}$ defined by $x \mapsto h_0(i)^{-1}xh_0(i)$ ($x\in D_{\mathbb R}$) is positive. We restrict $h_0$ to $\mathbb C^{\times}$ to obtain a morphism $h'$ from the Deligne torus $\mathrm{Res}_{\mathbb C/\mathbb R}\mathbb G_{m,{\mathbb C}}$ to $G_{\mathbb R}$. Let $h$ denote $h'^{-1}$ and let $X$ denote the $G(\mathbb R$)-conjugacy class of $h$, then the pair $(G,X)$ is a PEL-type Shimura datum (\cite{Kot92}). Let $\mathbb S$ denote $\mathrm{Res}_{\mathbb C/\mathbb R}\mathbb G_{m,{\mathbb C}}$. We define isomorphism $r$ from $\mathbb S_{\mathbb C}$ to $\mathbb G_{m,\mathbb C} \times \mathbb G_{m,\mathbb C}$ as follows: for an arbitrary $\mathbb C$-algebra $A$, define $A \otimes_{\mathbb R}\mathbb C \rightarrow  A^2$ by $a \otimes z \mapsto (za,\bar za)$. We define cocharacter $\mu':G_{m,\mathbb C} \rightarrow \mathbb S_{\mathbb C}$ by $r \circ \mu'(z) = (z,1)$ and we define $\mu \in X_*(G)$ to be $\mu'\circ h $. 
 
 Let $E$ denote the reflex field of Shimura datum $(G,X,h)$. We copy Kottwitz's description of the reflex field $E$ (\cite{Kot92}): Let $v_1,v_2$ be the two embeddings of $F$ into $\mathbb C$. We associates number $n_{v_1}$ to $v_1$ and $n_{v_2}$ to $v_2$, such that the group $G(\mathbb R)$ is isomorphic to the group $GU_{\mathbb C/\mathbb R}(n_{v_1} , n_{v_2} )(\mathbb R)$ (as defined in $\S 2.3$). The group $\mathrm{Aut}(\mathbb C/\mathbb Q)$ acts on the set of $\mathbb Z$-valued functions on $\mathrm{Hom}(F, \mathbb C)$ by translations. The reflex field $E$ is the fixed field of the stabilizer subgroup in $\mathrm{Aut}(\mathbb C/\mathbb Q)$ of the function $v \mapsto n_v$. Then we obtain varieties $Sh_K$ defined over the field $E$
	and these varieties represent corresponding moduli problems of Abelian varieties of PEL-type as defined in \cite{Kot92}.

	At p: Let $K \subset G(\mathbb A_f)$ be a compact open subgroup. Let $p$ be a prime number such that $G_{\mathbb Q_p}$ is unramified, and such that $K$ decomposes into a product $K_pK^p$ where $K_p \subset G(\mathbb Q_p)$ is hyperspecial and $K^p \subset G(\mathbb A_f^p)$ is small enough so that $Sh_K(G,X)$ is smooth. Let $Sh_K(G,X)$ be the Shimura variety (defined over $E$). Let $M_K$ be the moduli scheme over $\mathcal O_E \otimes_{\mathbb Z} \mathbb Z_{(p)}$ which represents the corresponding moduli problem of Abelian varieties of PEL-type (as defined in $\S 5$ of \cite{kot92p}, note that $M_{K,E}$ is $|\mathrm{Ker}^1(\mathbb Q, G)|$ copies of $Sh_K(G,X)$). We assume that $ p$ is inert in $F/\mathbb Q$ (the case that $p$ splits in extension $F/\mathbb Q$ is studied in \cite{Kret11} and \cite{Kret15}, also note in \cite{Kret11} and \cite{Kret15} truncated traces are studied, whereas we study twisted truncated traces in this article). Under our assumption, the group $G_{\mathbb Q_p}$ is unramified, and thus quasi-split, then we have $G_{\mathbb Q_p} \cong GU_{F_{\mathfrak p/\mathbb Q_p}}(A_n)$ (as defined in $\S 2.3$, we also use $GU(n)$ to denote $GU_{F_{\mathfrak p/\mathbb Q_p}}(A_n)$ when there is no ambiguity). 
	
	At $\mathbb R$: By the classification of unitary groups over real numbers, we have $G_{\mathbb R} \cong GU_{\mathbb C/\mathbb R}(s, n-s)$ for some $0 \leq  s \leq n $. If $n-s \not = s $, then the reflex field is $F$, and if $n-s = s $, then the reflex field is $\mathbb Q$.

	Let $\zeta$ be an irreducible algebraic representation over $\overline{\mathbb Q}$ of $G_{\overline{\mathbb Q}}$, associated with an $\ell$-adic local system on $\bar S$ (where $\ell \not = p$ is a prime number), denoted by $\mathcal L$. Let $\mathfrak g$ be the Lie algebra of
$G(\mathbb R)$ and let $K_{\infty}$ be the stabilizer subgroup in $G(\mathbb R) $ of the morphism $h$. Let $f_{\infty}$ be
an Euler–Poincar{\'e} function associated to $\zeta$ (as introduced in \cite{Clo85}, see also $\S 3$ of \cite{Kot92}). Note that $f_{\infty}$ has the following property: Let $\pi_{\infty}$ be a $(\mathfrak g, K_{\infty})$-
module occurring as the component at infinity of an automorphic representation $\pi$ of
$G$. Then the trace of $f_{\infty}$ against $\pi_{\infty}$ is equal to
$N_{\infty}\sum_{i=1}^{\infty}(-1)^idim(H^i(\mathfrak g, K_{\infty};\pi_{\infty}\otimes \zeta))$, where $N_{\infty}$ is a certain explicit constant as defined in $\S 3$ of \cite{Kot92}. Let $\ell$ be a prime number different from $p$ and let $\overline{\mathbb Q}_{\ell}$ be an algebraic closure of $\mathbb Q_{\ell}$ with an embedding $\overline{\mathbb Q} \subset \overline{\mathbb Q}_{\ell}$. We write $\mathcal L$ for the $\ell$-adic local system on $M_{K,\mathcal O_{E_{p}}}$ associated to the representation $\zeta \otimes \mathbb Q_{\ell}$ of $G_{\bar{\mathbb Q}_{\ell}}$ .
		
We write $q$ for the number of elements in the residue field of $\mathcal O_{E_p}$. 
 For each geometric point $x \in M_{K, \overline{\mathbb F}_q}$, we consider the isocrystal with additional $G$-structure $b(\mathcal A_x)$ associated to the abelian variety $\mathcal A_x$ corresponding to $x$ (see pp.170-171 of \cite{RaRi96}). The set of isocrystals with additional $G$-structure $\{b(\mathcal A_x): x \in  M_{K, \overline{\mathbb F}_q}(\overline {\mathbb F}_q)\}$ 
 is contained in certain finite set $B(G_{\mathbb Q_p} , \mu) \subset B(G_{\mathbb Q_p})$ (see $\S 6$ of \cite{Kot97}). For $b \in B(G_{\mathbb Q_p} )$, we define $S_b$ to be $\{x \in M_{K, \overline{\mathbb F}_q}(\overline {\mathbb F}_q): b(\mathcal A_x) = b\}$. Then $S_b$ is a subvariety of  $M_{K, \overline{\mathbb F}_q}$ (see $\S 1$ of \cite{Ra02}) and the collection $\{S_b\}_{b\in B(G_{\mathbb Q_p},\mu)}$ of subvarieties of $M_{K, \overline{\mathbb F}_q}$ is called the Newton stratification of $M_{K, \overline{\mathbb F}_q}$. 
We take $b$ to be the basic element of $B(G_{\mathbb Q_p},\mu)$, and call $S_b$ the basic stratum (in this case $S_b$ is closed in $M_{K, \overline{\mathbb F}_q}$, see $\S 1$ of \cite{Ra02}). 

Now let $S_b$ be the basic stratum of $M_{K, \overline{\mathbb F}_q}$, we use $B_K$ to denote the closed subscheme of $M_{K, \mathbb F_q}$ (defined over $\mathbb F_q$), such that $B_{K,\overline{\mathbb F}_q}$ is $S_b$ (See Theorem 3.6 of \cite{RaRi96}). Note that $Sh_{K^pK_p}(G,X)$ is compact in this case, then $M_{K, \mathbb F_q}$ is proper, therefore $B \subset M_{K, \mathbb F_q}$ is proper.
 The Hecke correspondences on $Sh_K(G,X)$ may be restricted to the subvariety $\iota : B_K \rightarrow
  M_{K, \mathbb F_q}$, and we have an action of the algebra $\mathcal H(G(\mathbb A_f)//K)$ on the cohomology spaces $H^i_{et} (B_K , \iota^*\mathcal L)$. Furthermore, the space $H^i (B_K , \iota^*\mathcal L)$ carries an action of the Galois
group $\mathrm{Gal}(\overline{\mathbb F}_{q} /\mathbb F_{q} )$, which commutes with the action of $\mathcal H(G(\mathbb A_f)//K)$.

\begin{remark}
Let $\sigma $ denote the non-trivial element of $\mathrm{Gal}(F/\mathbb Q)$. Recall that we assume that $\sigma$ agrees with $^*$ on $F$.
Note that $G(F)  = \{g \in D \otimes_{\mathbb Q} F: g\iota(g) \in F^{\times}\}$ where $\iota$ is given by  $\iota(d \otimes f) = d^* \otimes f$. 
We have $D\otimes_\mathbb Q F \cong D \times D$ with isomorphism given by $d \otimes f \mapsto (df, d\sigma (f))$. Then $\iota(d \otimes f) = d^* \otimes f$, $\sigma (d \otimes f) = d \otimes \sigma(f)$ induces involution on $D \times D$ given by $\iota(d_1, d_2) = (d_2^*, d_1^*)$ (because $\iota(d \otimes f) = d^* \otimes f \mapsto (d^*f , d^*\sigma (f)) = ((d\sigma(f))^*, (df)^*) )$ and $\sigma(d_1,d_2) = (d_2,d_1)$ (because $\sigma(d \otimes f) = d \otimes \sigma (f) \mapsto (d\sigma(f) , df) $). 

Also note that $G(F) \cong F^ \times  \times GL_n(F)$ with isomorphism given by $(d_1, d_2) \mapsto (d_1d_2^*, d_1)$ (we denote this isomorphism by $\varphi$). Then $ \sigma $ induces an action on $F^{\times} \times D^{\times}$ given by $\sigma (f,d) = (f^*, f(d^*)^{-1})$ (if we write $\varphi(d_1,d_2) = (d_1d_2^*, d_1) = (f,d)$, then we have $\sigma(f,d) = \sigma(d_1d_2^*, d_1)= \sigma(\varphi(d_1,d_2)) $
$ = \varphi(\sigma(d_1,d_2)) = \varphi(d_2,d_1)= (d_2d_1^*, d_2)$
$ =  ((d_1d_2^*)^*, (d_1d_2^*/d_1)^*) = (f^*, f^*(d^*)^{-1})$). 

Similarly, we have $G_{\mathbb Q_p} (F_{\mathfrak p}) \cong F_{\mathfrak p}^{\times} \times GL_n(F_{\mathfrak p})$, and the non-trivial element of $\mathrm{Gal}(F_{\mathfrak p}/\mathbb Q_p)$ (denoted by $f \mapsto \bar f$ for $f \in F_{\mathfrak p}$) induces an action on $F_{\mathfrak p}^{\times} \times GL_n(F_{\mathfrak p})$ given by $(f,g) \mapsto (\bar f, \bar fA_n^{-1}(\bar g ^t)^{-1}A_n)$, where $A_n$ is the anti-diagonal matrix with entries $1$ and $g \mapsto \bar g^t$ is the conjugation transpose. 
\end{remark}

Let $\alpha$ be a positive integer and let $E_{p,\alpha}/E_p$ be an unramified extension of degree $\alpha$. We write $\phi_{\alpha}$ for the characteristic function of the double coset $G(\mathcal O_{E_{p,\alpha} })\mu (p^{-1})G(\mathcal O_{E_{p,\alpha} })$ in $G(E_{p,\alpha})$. The function $\phi'_{\alpha} \in \mathcal H^{\mathrm{unr}}(\mathbb Q_p)$ is obtained from $\phi_{\alpha}$ via base change from $\mathcal H^{\mathrm{unr}}(E_{p,\alpha})$ to $\mathcal H^{\mathrm{unr}}(\mathbb Q_p)$ (as defined as p239 of \cite{Kot86}). We call the functions $\phi'_{\alpha}$ the functions of Kottwitz. 

\begin{proposition}
\begin{enumerate}
    \item If $n-s \not= s$ (in which case the reflex field is $F$), then $\phi'_{\alpha} = f_{n \alpha s}$ (recall that we define $f_{n \alpha s}$ in Definition 2.4.1).
    \item If $ n-s = s$ (in which case the reflex field is $\mathbb Q$) and $\alpha$ is even, then $\phi'_{\alpha} = f_{n \frac \alpha2 s}$. 
    \item If $ n-s = s$ (in which case the reflex field is $\mathbb Q$) and $\alpha$ is odd, then $\phi'_{\alpha} = p^{-\frac \alpha2 s(n-s)}f_{n\alpha s}$ (recall that we define $\varphi_{n \alpha s}$ in Definition 2.4.1).
\end{enumerate}

\end{proposition}

\begin{proof}
Firstly, note that for an arbitrary finite unramified extension $L/\mathbb Q_p$, the group $G(L)$ splits if and only if $[L:\mathbb Q_p]$ is even. If $G(E_{p,\alpha})$ splits, then $\phi'_{\alpha} = \varphi_{n 1 s}$ (recall that we define $\varphi_{n \alpha s}$ in $\S$2.4). For the proof, see Proposition 4.2.1 of \cite{Morel08}, and note that we take the reciprocal of $h'$ when we define Shimura datum at the beginning of this section.

If $G(E_{p,\alpha})$ does not split, the Stake transform of $ \phi'_{\alpha} $ is $$ p^{\alpha s(n-s)/2}\sum_{\sigma \in \{\pm 1\}^k \rtimes S_k} \sigma (X^{\alpha}X_1^{\alpha} \ldots X_s^{\alpha}) =  p^{\alpha s(n-s)/2} X^{\alpha}\sum\limits_{\substack{I \subset \{1,2, \ldots, n\}\\|I| = s}}(X^I)^{\alpha}$$ where $X^I : = \prod_{i \in I}X_{i}$ (see Proposition 4.2.1 of \cite{Morel08}, for the definition of $X, X_i$ see Definition 2.4.1, also recall that we define the action of $\{\pm 1\}^k \rtimes S_k$ on $X^{\alpha} X_1^{\alpha} \ldots X_s^{\alpha}$ at the end of $\S2.3$). Thus
the Stake transform of $ \phi'_{\alpha} $ is $ p^{-\alpha s(n-s)/2} f_{n 1 s}$ if $G(E_{p,\alpha})$ does not split. 

By applying the base change map as defined in $\S$4.2 of \cite{Morel08}, we get our proposition.  
\end{proof}

\subsection{Cohomology of the basic stratum}

Let $f^{\infty p} \in \mathcal H(G(\mathbb A^{p \infty})//K)$ to be an arbitrary Hecke operator, and let $f_{\infty} \in \mathcal H(G(\mathbb R))$ be an Euler–Poincar{\'e} function associated to $\zeta $. Then we have the following proposition (Proposition 9 of \cite{Kret11}):
	
\begin{proposition}
If $n-s \not = s$, the reflex field is $F$, and we define $q = |\mathcal O_F/(p)| = p^2$. We have $$\mathrm{Tr}(f^{\infty p}\times \Phi_{\mathfrak p}^{\alpha}, \sum_{i = 1}^{\infty}(-1)^iH^i_{et}(B_{\bar {\mathbb F}_q}, \iota^*\mathcal L)) = |\mathrm{ker}^1(\mathbb Q, G)| \sum\limits_{x' \in Fix_{\Phi_{\mathfrak p} ^{\alpha}\times f^{\infty p}}(\bar{\mathbb F}_q)}   \mathrm{Tr}(\Phi_{\mathfrak p} ^{\alpha}\times f^{\infty p}, \mathcal L_x) $$ 
$$=  | \mathrm{ker}^1(\mathbb Q, G)| \mathrm{Tr}(C_cf_{n\alpha s} f^{\infty p} f_{\infty} , \mathcal A(G))$$ for $\alpha$ sufficiently large. 

If $n-s = s$, the reflex field is $\mathbb Q$. We have $$\mathrm{Tr}(f^{\infty p}\times \Phi_{\mathfrak p}^{\alpha}, \sum_{i = 1}^{\infty}(-1)^iH^i_{et}(B_{\bar {\mathbb F}_p}, \iota^*\mathcal L)) = | \mathrm{ker}^1(\mathbb Q, G) | \sum\limits_{x' \in Fix_{\Phi_{\mathfrak p} ^{\alpha}\times f^{\infty p}}(\bar{\mathbb F}_p)}   \mathrm{Tr}(\Phi_{\mathfrak p} ^{\alpha}\times f^{\infty p}, \mathcal L_x)  $$
$$= | \mathrm{ker}^1(\mathbb Q, G) |\mathrm{Tr}(C_cf_{n \frac \alpha2 s} f^{\infty p} f_{\infty}, \mathcal A(G))$$ for $\alpha$ even and sufficiently large, and $$\mathrm{Tr}(f^{\infty p}\times \Phi_{\mathfrak p}^{\alpha}, \sum_{i = 1}^{\infty}(-1)^iH^i_{et}(B_{\bar {\mathbb F}_p}, \iota^*\mathcal L)) = | \mathrm{ker}^1(\mathbb Q, G) | \sum\limits_{x' \in Fix_{\Phi_{\mathfrak p} ^{\alpha}\times f^{\infty p}}(\bar{\mathbb F}_p)}   \mathrm{Tr}(\Phi_{\mathfrak p} ^{\alpha}\times f^{\infty p}, \mathcal L_x)  $$
$$= | \mathrm{ker}^1(\mathbb Q, G) |\mathrm{Tr}(C_c(p^{-\frac \alpha2 s(n-s)}f_{n\alpha  s} )f^{\infty p} f_{\infty}, \mathcal A(G))$$ for $\alpha$ odd and sufficiently large.
\end{proposition}

\begin{proof}
We use proposition 9 of \cite{Kret11} and Proposition 4.1.1.
\end{proof}

We assume that $n-s =s$. For $k \in \mathbb Z_{>0}$, let $B_{k,1}, \ldots B_{k,i_k}$ denote the irreducible components of $B_{\mathbb F_{p^k}}$. We use $d_B$ to denote the dimension of $B$, and we use $C_{k,l}$ to denote the number of irreducible components of $B_{k,l,\overline{\mathbb F}_p}$. Then we have the following corollary:

\begin{corollary}
We assume that $n-s = s$, then we have:
\begin{enumerate}
    \item If $|B_{\mathbb F_{p^k}}| \not = 0$ for some $k \in \mathbb Z_{>0,odd}$, then $d_B \geq \frac{s(n-s)}4$. 
    \item If $C_{k,l}$ is odd for some $k \in \mathbb Z_{>0,odd}$ and $1 \leq l \leq i_k$, then $d_B = \frac{s(n-s)}2$. 
\end{enumerate}
\end{corollary}

\begin{proof}
The proof of (1):
We take $k \in \mathbb Z_{>0}$ to be a sufficiently large odd integer such that $|B_{\mathbb F_{p^k}}| \not = 0$. We take $f^{\infty p }$ to be the characteristic function of $K^p$ (denoted by $1_{K^p}$) and we take $\zeta$ to be the trivial representation $G_{\overline{\mathbb Q}}$ (then the associated local system $\mathcal L$ is $\overline{\mathbb Q}_{\ell}$). Applying Proposition 4.2.1, then for an arbitrary odd $\alpha \in \mathbb Z_{>0}$, we have $$|B(\mathbb F_{p^{\alpha k}})| = \mathrm{Tr}(1_{K^p}\times \Phi_{\mathfrak p}^{\alpha k}, \sum_{i = 1}^{\infty}(-1)^iH^i_{et}(B_{\bar {\mathbb F}_p}, \overline{\mathbb Q}_{\ell})) $$
$$=  | \mathrm{ker}^1(\mathbb Q, G) |\mathrm{Tr}(C_c(p^{-\frac {\alpha k}2 s(n-s)}\varphi_{n (\alpha k)  s} )f^{\infty p} f_{\infty}, \mathcal A(G)),$$ 
$$|B(\mathbb F_{p^{2\alpha k}})| = \mathrm{Tr}(1_{K^p}\times \Phi_{\mathfrak p}^{2\alpha k}, \sum_{i = 1}^{\infty}(-1)^iH^i_{et}(B_{\bar {\mathbb F}_p}, \overline{\mathbb Q}_{\ell})) =  | \mathrm{ker}^1(\mathbb Q, G) |\mathrm{Tr}(C_c\varphi_{n (\alpha k)  s} f^{\infty p} f_{\infty}, \mathcal A(G)).$$ 
Then we may conclude that $|B(\mathbb F_{p^{2\alpha k}})| = p^{\frac {\alpha k}2 s(n-s)}  |B(\mathbb F_{p^{\alpha k}})|$. By Corollary 6.4 of \cite{Mustata}, there exists $c \in \mathbb R_{>0}$ such that $ |B(\mathbb F_{p^{\alpha}})| < c p^{\alpha d_B} $ for all odd $\alpha \in \mathbb Z_{>0}$. Then for all odd $\alpha \in \mathbb Z_{>0}$, we have $$ p^{\frac {\alpha k}2 s(n-s)} < p^{\frac {\alpha k}2 s(n-s)}  |B(\mathbb F_{p^{\alpha k}})| = |B(\mathbb F_{p^{2\alpha k}})| < c p^{2\alpha k d_B} ,$$ then we may conclude that $d_B > \frac{s(n-s)}4$. 

The proof of (2): We assume that $C_{k,l}$ is odd and we take $k \in \mathbb Z_{>0}$ to be a sufficiently large odd integer such that $C_{k,l} | k $. We take $f^{\infty p }$ to be the characteristic function of $K^p$ (denoted by $1_{K^p}$) and we take $\zeta$ to be the trivial representation $G_{\overline{\mathbb Q}}$. Repeat the proof of (1), we have $|B(\mathbb F_{p^{2\alpha k}})| = p^{\frac {\alpha k}2 s(n-s)}  |B(\mathbb F_{p^{\alpha k}})|$ for all odd $\alpha \in \mathbb Z_{>0}$. By Corollary 6.4 and Proposition 6.7 of \cite{Mustata}, there exists $c_1, c_2 \in \mathbb R_{>0}$ such that $ c_1 p^{\alpha d_B}< |B(\mathbb F_{p^{\alpha}})| < c_2 p^{\alpha d_B} $ for all odd $\alpha \in \mathbb Z_{>0}$. Then for all odd $\alpha \in \mathbb Z_{>0}$, we have $$ c_1 p^{\frac {\alpha k}2 s(n-s) + \alpha k d_B} < p^{\frac {\alpha k}2 s(n-s)}  |B(\mathbb F_{p^{\alpha k}})| = |B(\mathbb F_{p^{2\alpha k}})| < c_2 p^{2\alpha k d_B} ,$$ therefore $c_1 p^{\frac {\alpha k}2 s(n-s) } < c_2 p^{\alpha k d_B}$ for all odd $\alpha \in \mathbb Z_{>0}$ and we have $d_B \geq \frac{s(n-s)}2$. 

For all odd $\alpha \in \mathbb Z_{>0}$, we also have $$ c_1 p^{2\alpha k d_B} < |B(\mathbb F_{p^{2\alpha k}})|  =  p^{\frac {\alpha k}2 s(n-s)}  |B(\mathbb F_{p^{\alpha k}})| <  c_2 p^{\frac {\alpha k}2 s(n-s) + \alpha k d_B},$$ therefore $ c_1 p^{\alpha k d_B} < c_2 p^{\frac {\alpha k}2 s(n-s) }$ for all odd $\alpha \in \mathbb Z_{>0}$, and we have $d_B \leq \frac{s(n-s)}2$. Then we may conclude that $d_B = \frac{s(n-s)}2$.
\end{proof}

\begin{remark}
Assume $n-s = s$ and $s(n-s)$ is odd. By (2) of Corollary 4.2.1, we have that $C_{k,l}$ is even for all odd $k \in \mathbb Z_{>0}$.
\end{remark}
 
Let $D^{\times}$ denote the algebraic group over $\mathbb Q$ given by $D^{\times} (A) = (D\otimes_{\mathbb Q}A)^{\times}$, and let $F^{\times}$ denote the algebraic group over $\mathbb Q$ given by $F^{\times} (A) = (F\otimes_{\mathbb Q}A)^{\times}$. 
	
\begin{proposition}
Let $H $ denote $\mathrm{Res}_{\mathbb Q}^F(F^{\times}\times D^{\times})$ (as an algebraic group over $\mathbb Q$). We suppose that $ f= f_{ep}^{G} \otimes f^{\infty} \in \mathcal H(G({\mathbb A}))$ and $\varphi = \varphi_{ep}^{H, \theta} \otimes \varphi^{\infty} \in  \mathcal H(H({\mathbb A}))$ are associated in the sense of base change (i.e., the functions $f_v$ and $\varphi_v$ are associated for all places $v$, as defined in $\S 3.1 $ of \cite{Lab99}), where $ f_{ep}^{G}$ is an Euler-Poincare function on $G(\mathbb R)$ associated to the trivial representation, and $ \varphi_{ep}^{H, \theta}$ is a Lefschetz function for $ \theta$ on $H(\mathbb R)$ (as introduced in p.120 of \cite{Clo99}). Then we have $$\mathrm{Tr} (r (f), \mathcal A(G)) = \mathrm{Tr}(R(\varphi)A_\theta  , \mathcal A( F^{\times} \times D^{\times} ))$$ where $r, R$ are right regular representations, and $A_{\theta}$ is defined by $\alpha \mapsto \alpha \circ \theta$ for $\alpha \in \mathcal A  (F^{\times} \times D^{\times} ))$, and $\theta$ is the involution arising from the Galois action (see Remark 4.1.1).
\end{proposition}
	
\begin{proof}
We follow the argument given in the proof of Theorem A.3.1 in \cite{Clo99}:: 

 Let $G^*, H^*$ be the quasi-split inner form of $G, H$ respectively. Note that $$H(\mathbb R) = (F \otimes_{\mathbb Q} \mathbb R)^{\times} \times (D  \otimes_{\mathbb Q} \mathbb R)^{\times} = \mathbb C^{\times } \times GL_n(\mathbb C), $$ then we have $H^1(\mathbb R, (H_{\mathbb R})_{SC}) = H^1(\mathbb R,\mathrm{Res}_{\mathbb R}^{\mathbb C}(SL_{n,\mathbb R}))$. By Shapiro's Lemma, we have $H^1(\mathbb R,\mathrm{Res}_{\mathbb R}^{\mathbb C}(SL_{n,\mathbb R}))= 1$. By applying Theorem A.1.1 of \cite{Clo99}, we have that $\varphi_{ep}^{H,\theta}$ is stabilizing (as defined Definition 3.8.2 of \cite{Lab99}). After multiplying $\varphi^{\theta}_{ep}$ by some $c \in \mathbb R_{>0}$, we may choose $\varphi^{\theta}_{ep}$ and $f_{ep}$ such that $\varphi^{\theta}_{ep}$ and $f_{ep}$ are associated (see Corollary A1.2 of \cite{Clo99}). 
  
 Also note that traces formula for $G$ and $H$ only contain elliptic terms on the geometric side, then we have
 $$T_e^{H, \theta} (\varphi) =  \mathrm{Tr}(R(\varphi)A_\theta  , \mathcal A( F^{\times} \times D^{\times} )) \text{ and } T_e^{G} (f) =  \mathrm{Tr} (r (f), \mathcal A(G))$$
 where $T_e^{H, \theta} $ is the elliptic of the geometric side of the twisted trace formula for $H$ (as defined in 4.1 of \cite{Lab99}) and $T_e^{G} $ is the elliptic of the geometric side of the trace formula for $G$ (as defined in 9.1 of \cite{Kot86}). 
Applying Lemma 1.9.7 in \cite{Lab99}, we have that $\{\infty\}$ is $(H^*, G^*)$-essential (as defined in Definition 1.9.5 of \cite{Lab99}). 
Let $f^{G^*} \in G^*(\mathbb A)$ be a Langlands-Shelstad transfer of $f$ (see 9.3 of \cite{Kot86}). By applying Theorem 4.3.4 of \cite{Lab99}, we obtain $ T_e^{H, \theta} (\varphi) =  a(H,\theta, G^*)  ST_e^{G^*}(f^{G^*})$ where $ST_e^{G^*}$ is the elliptic part of the stable trace formula for $G^*$ (as defined in 9.2 of \cite{Kot86}), and $a(H,\theta, G^*) )$ as defined in $\S$4.3 of \cite{Lab99}. By Proposition 4.3.2 in \cite{Lab99}, we have $a(H,\theta, G^*) = 1$. Then we have $$ \mathrm{Tr}(R(\varphi)A_\theta  , \mathcal A( F^{\times} \times D^{\times} ))  =T_e^{H, \theta}(\varphi) = ST_e^{G^*}(f^{G^*})$$

For $\gamma_0 \in G(\mathbb Q)$, we have that $\mathfrak R(I_0/\mathbb Q)$ is trivial (see $\S 2$ of \cite{Kot92}), where $I_0$ is the centraliser of $\gamma_0$ in G and $\mathfrak R(I_0/\mathbb Q)$ is defined in $\S 4$ of \cite{Kot86}. Then $G$ satisfies the conditions listed before Lemma A2.1 in the appendix. Then by applying Lemma A2.1, we may conclude that $ T_e^{G} (f) = ST_e^{G^*}(f^{G^*})$. Then we have $$ \mathrm{Tr}(R(\varphi)A_\theta  , \mathcal A( F^{\times} \times D^{\times} ))=T_e^{H, \theta}(\varphi) = ST_e^{G^*}(f^{G^*}) = T_e^{G} (f) =  \mathrm{Tr} (r (f), \mathcal A(G)) $$
\end{proof} 

\begin{remark}
Let $K_{\infty}$ be a maximal compact subgroup of $H(\mathbb R)$, and let $\mathfrak h$ be the Lie algebra of $H(\mathbb R)$. Let $\pi \subset \mathcal A( F^{\times} \times D^{\times} )$ be an automorphic representation such that $\mathrm{Tr}(R(\varphi)A_\theta  , \pi) \not = 0$, then $\mathrm{Tr}_{\theta}(\varphi^{\theta}_{ep},\pi_{\infty} ) = \sum_{i  =1}^{\infty} \mathrm Tr(\theta, H^i(\mathfrak h, K_{\infty}; \pi_{\infty}))\not = 0 $ (see p.120 of \cite{Clo99}), which implies that $\pi_{\infty}$ is cohomological. 
\end{remark} 
We fix the following datum:

\begin{itemize}

\item Local system: We take $\zeta$ to be the trivial representation $G_{\overline{\mathbb Q}}$, then the associated local system $\mathcal L$ is $\overline{\mathbb Q}_{\ell}$. 
\item At $p$: If $ n -s \not = s $, we take $C_{\theta c}  \varphi_{n \alpha s} \in \mathcal H(H(\mathbb Q_{p}))$ (with $\varphi_{n \alpha s}$ as defined in $\S$2.4), and $ C_c f_{n \alpha s }\in \mathcal H(G(\mathbb Q_{p}))$. Then $C_{\theta c}  \varphi_{n \alpha s}$ and $ C_c f_{n \alpha s }$ are associated (see the proof of Proposition 10 of \cite{Kret11}).

If $ n -s  = s $ and $\alpha $ is odd, we take $C_{\theta c}  (p^{-\frac{\alpha}2 s(n-s)}\varphi_{n \alpha s} )$ 
$\in \mathcal H(H(\mathbb Q_{p}))$ and $ C_c (p^{-\frac{\alpha}2 s(n-s)}f_{n \alpha s })$ 
$\in \mathcal H(G(\mathbb Q_{p}))$. If $ n -s  = s $ and $\alpha $ is even, we take $C_{\theta c}  \varphi_{n \frac\alpha2 s} \in \mathcal H(H(\mathbb Q_{p}))$ and $ C_c f_{n \frac\alpha2 s }\in \mathcal H(G(\mathbb Q_{p}))$

\item Away from $p$ and infinite: We take $1 _{K^{p}} \in \mathcal H(G(\mathbb A^{p \infty}))$ to be the characteristic function of $K^p$, then there exists $ \varphi^{p\infty} \in \mathcal H(H(\mathbb A^{p \infty}))$ associated to $1 _{K^{p \infty}}$ if $K^p$ is small enough (for unramified places, see Proposition 3.3.2 in \cite{Lab99}, for ramified places, see the proof of Theorem A5.2 in \cite{Clo99})
\item At $\infty$: We take $\varphi_{ep}^{\theta} \in \mathcal H(H(\mathbb A_{ \infty}))$ to be a Lefschetz function for $ \theta$, and $ f_{ep} \in \mathcal H(G(\mathbb A_ \infty)$ to be an Euler-Poincare function associated to the trivial representation, after multiplying $\varphi^{\theta}_{ep}$ by some $c \in \mathbb R_{>0}$, we may choose $\varphi^{\theta}_{ep}$ and $f_{ep}$, such that $\varphi^{\theta}_{ep}$ and $f_{ep}$ are associated (see Corollary A1.2 of \cite{Clo99}). 
\end{itemize}

By applying Proposition 4.2.1 and Proposition 4.2.2, we obtain:

\begin{corollary}
\begin{enumerate}
    \item If $n-s \not = s$, we define $q = |\mathcal O_F/(p)| = p^2$, then for an arbitrary $\alpha \in \mathbb Z_{>0}$ sufficiently large, we have $$|B(\mathbb F_{q^{\alpha}})| = \mathrm{Tr}(1 _{K^{p}}\times \Phi_{\mathfrak p}^{\alpha}, \sum_{i = 1}^{\infty}(-1)^iH^i_{et}(B_{\bar {\mathbb F}_q}, \iota^*\mathcal L))$$ 
    
    $$= | \mathrm{ker}^1(\mathbb Q, G)|\mathrm{Tr}_{\theta}(C_{\theta c}  \varphi_{n \alpha s}\varphi^{p\infty} \varphi_{ep}^{\theta}, \mathcal A( F^{\times} \times D^{\times} )).$$
    \item  If $n-s = s$, then for $\alpha \in \mathbb Z_{>0}$ even and sufficiently large, we have $$|B(\mathbb F_{p^{\alpha}})| =  \mathrm{Tr}(1 _{K^{p}}\times \Phi_{\mathfrak p}^{\alpha}, \sum_{i = 1}^{\infty}(-1)^iH^i_{et}(B_{\bar {\mathbb F}_p}, \iota^*\mathcal L)) $$
    
   $$ = | \mathrm{ker}^1(\mathbb Q, G)| \mathrm{Tr}_{\theta}(C_{\theta c}  \varphi_{n \frac\alpha2 s}\varphi^{p\infty} \varphi_{ep}^{\theta}, \mathcal A( F^{\times} \times D^{\times} )). $$
    \item  If $n-s = s$, then for $\alpha \in \mathbb Z_{>0}$ odd and sufficiently large, we have $$|B(\mathbb F_{p^{\alpha}})| =  \mathrm{Tr}(1 _{K^{p}}\times \Phi_{\mathfrak p}^{\alpha}, \sum_{i = 1}^{\infty}(-1)^iH^i_{et}(B_{\bar {\mathbb F}_p}, \iota^*\mathcal L)) $$
    $$=| \mathrm{ker}^1(\mathbb Q, G)| \mathrm{Tr}_{\theta}(C_{\theta c}  (p^{-\frac \alpha2 s(n-s)}\varphi_{n \alpha s})\varphi^{p\infty} \varphi_{ep}^{\theta}, \mathcal A( F^{\times} \times D^{\times} )). $$
    
\end{enumerate}

\end{corollary}

 \subsection{A point-counting formula}

we establish the following point-counting formula in this section:

\begin{proposition}
\begin{enumerate}
    \item If $n-s \not = s$, then for an arbitrary $\alpha \in \mathbb Z_{> 0}$ sufficiently large, we have $$|B(\mathbb F_{q^{\alpha}})| =| \mathrm{ker}^1(\mathbb Q, G)| \sum_{\pi} N_{\pi} \mathrm{Tr}_{\theta}(C_{\theta c} \phi_{n \alpha s}, \pi) ,$$ where $\pi $'s are representations of $GL_n(F_p)$ of the form $\mathrm{Ind}_P^G(\otimes_i \mathrm{St}_{GL_{n_i}}(\varepsilon_i))$, such that $\varepsilon_i$ is either the trivial character or the unramified quadratic character and $$N_{\pi} = \pm \sum\limits_{\substack{\chi \otimes \tau \subset \mathcal A(F^{\times} \times D^{\times}): \\ \chi \otimes \tau \text{ is } \theta \text{-stable,} \\  \tau_p \text{ is of }\theta \text{ type}, \\ \pi\leq \tau_p  }}  \mathrm{Tr}_{\theta}( \varphi_{ep}^{\theta},(\chi \otimes \tau)_{\infty} )  \mathrm{Tr}_{\theta}(\varphi^{p\infty} ,(\chi \otimes \tau )^{p \infty}) M_{\tau_p, \pi}$$ with the partial ordering as defined before Proposition 3.3.2, and $M_{\tau_p, \pi}$ as introduced in Theorem 3.3.1, and $$\mathrm{Tr}_{\theta}(\varphi^{\theta}_{ep},(\chi \otimes \tau)_{\infty} ) = \sum_{i  =1}^{\infty} \mathrm Tr(\theta, H^i(\mathfrak h, K_{\infty}; (\chi \otimes \tau)_{\infty}))$$ (see Remark 4.2.1).  
    
    \item If $n-s = s$, then for an arbitrary even $\alpha \in \mathbb Z_{> 0}$ which is sufficiently large, we have $$|B(\mathbb F_{p^{\alpha}})| = | \mathrm{ker}^1(\mathbb Q, G)|\sum_{\pi} N_{\pi} \mathrm{Tr}_{\theta}(C_{\theta c} \phi_{n \frac \alpha2 s}, \pi) $$ with $N_{\pi}$ as defined above, and $\pi $'s are representations of $GL_n(F_p)$ of the form $\mathrm{Ind}_P^G(\otimes_i \mathrm{St}_{GL_{n_i}}(\varepsilon_i))$, such that $\varepsilon_i$ is either the trivial character or the unramified quadratic character.  

    \item If $n-s = s$, then for an arbitrary odd $\alpha \in \mathbb Z_{> 0}$ which is sufficiently large, we have $$|B(\mathbb F_{p^{\alpha}})| = | \mathrm{ker}^1(\mathbb Q, G)|\sum_{\pi} N_{\pi} p^{-\frac\alpha2 s(n-s)} \mathrm{Tr}_{\theta}(C_{\theta c} \phi_{n  \alpha s}, \pi) $$ with $N_{\pi} $ as defined above and $\pi $'s are representations of $GL_n(F_p)$ of the form $\mathrm{Ind}_P^G(\otimes_i \mathrm{St}_{GL_{n_i}}(\varepsilon_i))$, such that $\varepsilon_i$ is either the trivial character or the unramified quadratic character.  
\end{enumerate}

\end{proposition}

\begin{proof}
We assume that $ n-s \not = s$ (the proof for the case $n-s = s$ is almost identical): 

Take $\alpha \in \mathbb Z_{>0}$ sufficiently large, by applying Corollary 4.2.1, we have $$|B(\mathbb F_{q^{\alpha}})| =| \mathrm{ker}^1(\mathbb Q, G)| \mathrm{Tr}_{\theta}(C_{\theta c}  \varphi_{n \alpha s}\varphi^{p\infty} \varphi_{ep}^{\theta}, \mathcal A( F^{\times} \times D^{\times} )).$$

$$ = \sum\limits_{\rho \subset \mathcal A( F^{\times} \times D^{\times} )} \mathrm{Tr}_{\theta}(C_{\theta c}  \varphi_{n \alpha s}\varphi^{p\infty} \varphi_{ep}^{\theta},\rho )) $$ where the sum ranges over $\theta$-stable irreducible subrepresentations of $F^{\times}(\mathbb A) \times D^{\times}(\mathbb A)$ in $ \mathcal A( F^{\times} \times D^{\times} ) $. 
And  $$  \mathrm{Tr}_{\theta}(C_{\theta c}  \varphi_{n \alpha s}\varphi^{p\infty} \varphi_{ep}^{\theta},\rho ))   =  \mathrm{Tr}_{\theta}( \varphi_{ep}^{\theta},\rho_{\infty} ))  \mathrm{Tr}_{\theta}(\varphi^{p\infty} ,\rho ^{p \infty})) \mathrm{Tr}_{\theta}(C_{\theta c}  \varphi_{n \alpha s},\rho_p ).    $$
We write $\rho = \chi \otimes \tau \subset \mathcal A( F^{\times}) \otimes \mathcal A(D^{\times} )$, then $\rho_p  = \chi_p \otimes \tau_p $ is a representation of $F_p^{\times} \times GL_n(F_p)$. 

We use $\theta_1$ to denote the involution of $F^{\times}_p \times GL_n(F_p)$ defined by $(f,g) \mapsto (\bar f, \bar fA_n^{-1}(\bar g ^t)^{-1}A_n)$ where $A_n$ is the anti-diagonal matrix with entries $1$ and $g \mapsto \bar g^t$ is the conjugation transpose (which arises from the Galois action, as described in Remark 4.1.1). And we use $\theta_2$ to denote the involution of $F^{\times}_p \times GL_n(F_p)$ defined by $(f,g) \mapsto (\bar f, A_n^{-1}(\bar g ^t)^{-1}A_n)$. Because of the $\theta$-stableness, we must have $(\chi_p \otimes \tau_p) \circ \theta_1 \cong \chi_p \otimes \tau_p $. Also, note that for an arbitrary $(f,g) \in F_p^{\times} \times GL_n(F_p)$, we have $(\chi_p \otimes \tau_p) \circ \theta_1(f,g) = (\chi_p (\bar f) \otimes \tau_p(\bar fA_n^{-1}(\bar g ^t)^{-1}A_n)$, therefore, if $(\chi_p \otimes \tau_p) \circ \theta_1 \cong \chi_p \otimes \tau_p $, then $\tau_p$ must have trivial central character, and we have $(\chi_p \otimes \tau_p) \circ \theta_1(f,g) = (\chi_p (\bar f) \otimes \tau_p(\bar fA_n^{-1}(\bar g ^t)^{-1}A_n) = (\chi_p (\bar f) \otimes \tau_p(A_n^{-1}(\bar g ^t)^{-1}A_n) = (\chi_p \otimes \tau_p) \circ \theta_2(f,g)$. In the following, we also use $\theta$ to denote the involution of $GL_k(F_p)$ ($k \geq 0$) defined by $ g \mapsto A_n^{-1}(\bar g ^t)^{-1}A_n$, then we have $$ \mathrm{Tr}_{\theta}(C_{\theta c}  \varphi_{n \alpha s},\rho_p)  = \mathrm{Tr}_{\theta_2}(C_{\theta c}  \varphi_{n \alpha s},\chi_p \otimes \tau_p)) =   \mathrm{Tr}_{\theta_1}(C_{\theta c}  \varphi_{n \alpha s},\chi_p \otimes \tau_p )) = \mathrm{Tr}_{\theta} (X^{\alpha},\chi_p )\mathrm{Tr}_{\theta} (\phi_{n \alpha s},\tau_p ). $$ with $\phi_{n\alpha s}$ as defined in $\S$2.4. 
 
We apply the Jacquet-Langlands correspondence (cf.VI.1 of \cite{HarrisTaylor01}) to $\tau$ (denoted by $JL(\tau)$), then $JL(\tau)$ is a discrete automorphic representation of $GL_n(\mathbb A_F)$. By applying Theorem 3.1.1, we have that $\tau_p \cong JL(\tau)_p$ is a semi-stable rigid representation (we have the isomorphism because $p$ is an unramified place). Since $\chi_p$ is one-dimensional, we have $\mathrm{Tr}_{\theta} (X^{\alpha},\chi_p ) = \pm 1$ (see pp.511-512 of \cite{Kret11}), and therefore $$\mathrm{Tr}_{\theta} (X^{\alpha},\chi_p )\mathrm{Tr}_{\theta} (\phi_{n \alpha s},\tau_p ) =  \pm \mathrm{Tr}_{\theta} (\phi_{n \alpha s},\tau_p ). $$ 

By applying Corollary 3.1.1, we obtain $\mathrm{Tr}_{\theta} (\phi_{n \alpha s},\tau_p ) = \mathrm{Tr}_{\theta} (\phi_{n \alpha s},\tau'_p ) $ where $\tau_p$ is a semi-stable rigid representation of $\theta$-type (as defined in Definition 3.1.1). Then apply Theorem 3.3.1, we can conclude our theorem when $\alpha$ is sufficiently large. We will later remove the restriction on $\alpha$ in $\S$5.2.  
\end{proof}

\section{Twisted traces of products of Steinberg representations}

By Proposition 4.3.1, we may reduce the computation of $|B(q^{\alpha})|$ to the computation of twisted traces of products of Steinberg representations, in this section, we compute twisted traces of products of Steinberg representations. 

\subsection{An explicit description intertwining operators}

In the following, we fix a quadratic extension of $p$-adic fields $F/\mathbb Q_p$ such that $p$ remains prime in $F$, let $\sigma$ be the non-trivial element of $\mathrm{Gal}(F/\mathbb Q_p)$. For $k\in \mathbb Z_{>0}$, we use $\theta$ to denote the involution of $GL_k(F)$ given by $g \mapsto A_n^{-1}(\bar g ^t)^{-1}A_n$, where $A_n$ is the anti-diagonal matrix with entries $1$ and $g \mapsto \bar g^t$ is the conjugation transpose.  And we fix an irreducible admissible representation $\pi$ of $GL_n(F)$ of the form  $\mathrm{Ind}_P^G(\otimes_i \mathrm{St}_{GL_{n_i}}(\varepsilon_i))$, where $\varepsilon_i$ is either the trivial character or the unramified quadratic character. We embed $S_n$ into $GL_n$ as follows: We send $ \sigma $ to $ (a_{ij})$ where $a_{ij} = 1$ if and only if $\sigma(i) = j$, and $a_{ij} = 0$ otherwise. The action of $S_n$ on $GL_n$ is given by $w \cdot g = w^{-1} g w$ for $w \in S_n$ and $g \in GL_n$. 

We use $B_{k} = M_{0,k}N_{0,k} \subset  GL_{k}$ to denote the Borel group of upper triangular matrices. We standardise parabolic subgroups of $GL_k$ with respect to $B_k$.

Note that the permutation group on $\{1,2,\ldots,n\}$ (denoted by $S_n$) is generated by transposes $(1,2), \ldots, (n-1, n)$ (which we denote by $\tau_1, \ldots, \tau_{n-1}$), for $w \in S_n$, we define $\ell (w) $ to be $\min\{k: w = w_1w_2\ldots w_k: w_i \in \{\tau_1, \ldots, \tau_{n-1}\}\} $. For $i, j \in \{1,2,\ldots,n\}$, we say that $(i,j)$ is an inversion of $w$ if $i < j $ and $w(i) > w(j)$, we use $inv(w)$ to denote the total number of inversions of $w$, then we have the following proposition (for the proof, see Lemma 2.1 of \cite{Wildon}):

\begin{proposition}
For every $w \in S_n$, we have $inv(w) = \ell(w)$. 
\end{proposition}

Let $\lambda = (\lambda_1, \ldots, \lambda_k)$ be a partition of $n$, we use $S_{\lambda} \subset S_n$ to denote $S_{\lambda_1} \times \ldots \times S_{\lambda_k}$. Let $\lambda$ and $\mu$ be two partitions of $n$, then we have the following proposition:

\begin{proposition}
For an arbitrary double cosets $S_{\lambda}g S_{\mu} \subset S_n$, there exists a unique $g_{min} \in S_{\lambda}g S_{\mu}$ such that $\ell(g_{min})$ is minimal among elements in $S_{\lambda}g S_{\mu}$. 
\end{proposition}
\begin{proof}
See V.4.6 of \cite{Renard10}.
\end{proof}

We call $g_{min} \in S_{\lambda}g S_{\mu}$ the minimal length representative of $S_{\lambda}g S_{\mu}$. 

Let $ \pi = \mathrm{Ind}_P^G(\otimes_i \mathrm{St}_{GL_{n_i}}(\varepsilon_i))$ be a representation of $\theta$-type of $GL_n(F)$, where $\varepsilon_i$ is either the trivial character or the unramified quadratic character. We embed $ \mathrm{St}_{GL_{n_i}}(\varepsilon_i)$ into $\nu(\varepsilon_i)^{\frac {n_i-1}2} \times\nu(\varepsilon_i)^{\frac {n_i-3}2} \ldots \times  \nu(\varepsilon_i)^{\frac {1-n_i}2}$ (as introduced in $\S3.3$), then we embed $ \pi $ into an unramified principal series, which we denote by $\mathrm{Ind}_B^{GL_n}(\sigma)$. Let $A_{\theta}$ be an intertwining operator on $\mathrm{Ind}_{B_n}^{GL_n}(\sigma)$ given by $f \mapsto f \circ \theta$ for $f \in \mathrm{Ind}_{B_n}^{GL_n}(\sigma)$. Then we get an intertwining operator on $ \pi = \mathrm{Ind}_P^G(\otimes_i \mathrm{St}_{GL_{n_i}}(\varepsilon_i))$, which we also denote by $A_{\theta}$. Let $Q= LU$ be a standard parabolic subgroup of $GL_n$ then we have the following explicit description of $\pi_U$ (which is based on Bernstein-Zelevinsky filtration, see VI.5.2 of \cite{Renard10} or 2.2-2.3 of \cite{Nguyen04}) and the induced intertwining operator on $\pi_U$ (see 2.2-2.3 of \cite{Nguyen04}):

\begin{proposition}
In the ring of Zelevinsky $\mathcal R$, we have $$ J_	U(\mathrm{Ind}_P^G(\otimes_i \mathrm{St}_{GL_{n_i}}(\varepsilon_i)))= \bigoplus_{w \in G_{P,Q}  \subset S_n} \mathrm{Ind}_{L \cap w^{-1}P}^L (w \circ J_{w Q \cap M} (\otimes_i \mathrm{St}_{GL_{n_i}}(\varepsilon_i))$$ where $G_{P,Q}  \subset S_n$ is the set of minimal length representatives of $S_P\setminus S_n/ S_Q$

The induced intertwining operator $A_{\theta}$ acts on $$\bigoplus\limits_{w \in G_{P,Q}  \subset S_n} \mathrm{Ind}_{L \cap w^{-1}P}^L (w \circ  J_{w Q \cap M}(\otimes_i \mathrm{St}_{GL_{n_i}}(\varepsilon_i)))$$ as follows: the intertwining operator $A_{\theta}$ maps $$ \mathrm{Ind}_{L \cap w^{-1}P}^L (w \circ  J_{w Q \cap M}(\otimes_i \mathrm{St}_{GL_{n_i}}(\varepsilon_i)))$$ onto $$\mathrm{Ind}_{L \cap (w^{\theta})^{-1}P}^L (w^{\theta} \circ  J_{w^{\theta} Q \cap M}(\otimes_i \mathrm{St}_{GL_{n_i}}(\varepsilon_i))) $$ where $w^{\theta}$ is the minimal length representative in $S_P \theta (w) S_Q$ and $\theta(w)$ is defined to be $A_n^{-1} w A_n$ (recall that $A_n$ is the anti-diagonal matrix with entries $1$).
\end{proposition}

\subsection{Computation of twisted traces}
In this section, we compute the twisted traces of products of Steinberg representations and remove the $\alpha$ sufficiently large condition in Proposition 4.3.1. Let $\pi = \mathrm{Ind}_P^G(\otimes_i \mathrm{St}_{GL_{n_i}}(\varepsilon_i))$ be a representation of $\theta$-type of $GL_n(F)$, where $\varepsilon_i$ is either the trivial character or the unramified quadratic character, we compute $\mathrm{Tr}(C_{\theta c} \phi_{n\alpha s }, \pi A_{\theta})$ where $A_{\theta}$ was constructed before Proposition 5.1.3.  

First, by applying Proposition 2.2.3, we obtain $$\mathrm{Tr}_{\theta}(C_{\theta c} \phi_{n \alpha s}, \pi) = \sum\limits_{Q \in \mathcal P^{\theta}, Q = LU}\varepsilon_{Q,\theta} \mathrm{Tr}_{\theta}(\hat{\chi}^{\theta}_U  \phi_{n \alpha s}^{(Q)}, \pi_U(\delta_{Q}^{-\frac 12})A_{\theta}).$$  By proposition 5.1.3, we have $$\mathrm{Tr}_{\theta}(\hat{\chi}^{\theta}_U  \phi_{n \alpha s}^{(Q)}, \pi_U(\delta_{Q}^{-\frac 12})A_{\theta})  = \mathrm{Tr}(\hat{\chi}^{\theta}_U  \phi_{n \alpha s}^{(Q)},  \bigoplus_{w \in G_{P,Q}  \subset S_n} \mathrm{Ind}_{L \cap w^{-1}P}^L (w \circ  J_{w Q \cap M}(\otimes_i \mathrm{St}_{GL_{n_i}}(\varepsilon_i))) A_{\theta})$$ 
$$ =  \sum_{w \in G^{\theta}_{P,Q}  \subset S_n} \mathrm{Tr}_{\theta}(\hat{\chi}^{\theta}_U  \phi_{n \alpha s}^{(Q)},   \mathrm{Ind}_{L \cap w^{-1}P}^L (w \circ  J_{w Q \cap M}(\otimes_i \mathrm{St}_{GL_{n_i}}(\varepsilon_i)))$$ where $G^{\theta}_{P,Q}  \subset S_n$ is defined to be $\{w \in G_{P,Q}  \subset S_n: w^{\theta} = w, w S_{Q} \cap S_P = \{id_{S_n}\}\}$. We have the condition $w^{\theta} = w$ because $A_{\theta}$ maps $ \mathrm{Ind}_{L \cap w^{-1}P}^L (w \circ  J_{w Q \cap M}(\otimes_i \mathrm{St}_{GL_{n_i}}(\varepsilon_i))) $ to  $\mathrm{Ind}_{L \cap (w^{\theta})^{-1}P}^L (w^{\theta} \circ  J_{w^{\theta} Q \cap M}(\otimes_i \mathrm{St}_{GL_{n_i}}(\varepsilon_i))) $ where $w^{\theta}$ is the minimal length representative in $S_P \theta (w) S_Q$. And we have the condition $w^{-1} S_{Q} w \cap S_P = \{id_{S_n}\}$, because $\hat{\chi}^{\theta}_U  \phi_{n \alpha s}^{(Q)}$ is unramified, and $ \mathrm{Ind}_{L \cap w^{-1}P}^L (w \circ  J_{w Q \cap M}(\otimes_i \mathrm{St}_{GL_{n_i}}(\varepsilon_i)))$ is ramified if $w Q \cap M \not = M_0$, where $w Q \cap M  =M_0$ if and only if $w^{-1} S_{Q} w \cap S_P = \{id_{S_n}\}$. 

Let $(q^{t_{\pi,1}}, \ldots, q^{t_{\pi,n}})$ denote the Hecke matrix of $J_{N_0}(\otimes_i \mathrm{St}_{GL_{n_i}}(\varepsilon_i)) = \otimes_i (\delta_{B_{n_i}}^{\frac 12}(\varepsilon_i))$, for $w \in G^{\theta}_{P,Q}$ we have: 
$$\mathrm{Tr}_{\theta}(\hat{\chi}^{\theta}_U  \phi_{n \alpha s}^{(Q)},   \mathrm{Ind}_{L \cap w^{-1}P}^L (w \circ  J_{w Q \cap M}(\otimes_i \mathrm{St}_{GL_{n_i}}(\varepsilon_i)))) =   \mathrm{Tr}_{\theta}(\hat{\chi}^{\theta}_U  \phi_{n \alpha s}^{(Q)},   \mathrm{Ind}_{M_0}^L (w \circ  (\otimes_i \delta_{B_{n_i}}^{\frac 12}))$$

$$ =  \mathrm{Tr}_{\theta}(\hat{\chi}^{\theta}_U  \phi_{n \alpha s}^{(Q)},   w \circ  (\otimes_i \delta_{B_{n_i}}^{\frac 12}(\varepsilon_i))) = \varepsilon_{Q, w} \mathcal S(\hat{\chi}^{\theta}_U  \phi_{n \alpha s}^{(Q)}) (q^{t_{\pi,w(1)}}, \ldots, q^{t_{\pi,w(n)}}) $$ where $ \varepsilon_{Q, w} = \pm 1$ depends on $Q$ and $w$ (See Proposition 2.2.2). We have then the following proposition:

\begin{proposition}
$$\mathrm{Tr}_{\theta}(C_{\theta c} \phi_{n \alpha s}, \pi)  = \sum\limits_{Q \in \mathcal P^{\theta}, Q = LU}\sum\limits_{w \in G_{P,Q}^{\theta}}\varepsilon_{Q,\theta}  \varepsilon_{Q, w}  \mathcal S(\hat{\chi}^{\theta}_U  \phi_{n \alpha s}^{(Q)}) (q^{t_{\pi,w(1)}}, \ldots, q^{t_{\pi,w(n)}}) $$
\end{proposition}

Applying Proposition 4.3.1, we obtain the following theorem (we later remove the sufficiently large condition for $\alpha$ in Proposition 4.3.1): 
\begin{theorem}
\begin{enumerate}
    \item If $n-s \not = s$, then for an arbitrary $\alpha \in \mathbb Z_{> 0}$, we have $$|B(\mathbb F_{q^{\alpha}})| =| \mathrm{ker}^1(\mathbb Q, G)| \sum_{\pi} N_{\pi} \mathrm{Tr}_{\theta}(C_{\theta c} \phi_{n \alpha s}, \pi)$$ 
$$=   | \mathrm{ker}^1(\mathbb Q, G)| \sum_{\pi}  \sum\limits_{Q \in \mathcal P^{\theta}, Q = LU}\sum\limits_{w \in G_{P,Q}^{\theta}}N_{\pi} \varepsilon_{Q,\theta}  \varepsilon_{Q, w}  \mathcal S(\hat{\chi}^{\theta}_U  \phi_{n \alpha s}^{(Q)}) (q^{t_{\pi,w(1)}}, \ldots, q^{t_{\pi,w(n)}}) $$ with $N_{\pi}$ as defined in Proposition 4.3.1, and $\pi $'s are representations of $GL_n(F_p)$ of the form $\mathrm{Ind}_P^G(\otimes_i \mathrm{St}_{GL_{n_i}}(\varepsilon_i))$, such that $\varepsilon_i$ is either the trivial character or the unramified quadratic character (recall that $\varepsilon_{Q,w} = \pm 1$ is introduced before Proposition 5.2.1 and $\varepsilon_{Q, \theta} = \pm 1$ is introduced in Proposition 2.2.3). 

\item If $n-s = s$, then for an arbitrary even $\alpha \in \mathbb Z_{> 0}$, we have $$|B(\mathbb F_{p^{\alpha}})| =| \mathrm{ker}^1(\mathbb Q, G)| \sum_{\pi} N_{\pi} \mathrm{Tr}_{\theta}(C_{\theta c} \phi_{n \frac \alpha2 s}, \pi) $$ 
$$   =   | \mathrm{ker}^1(\mathbb Q, G)| \sum_{\pi}  \sum\limits_{Q \in \mathcal P^{\theta}, Q = LU}\sum\limits_{w \in G_{P,Q}^{\theta}}N_{\pi} \varepsilon_{Q,\theta}  \varepsilon_{Q, w}  \mathcal S(\hat{\chi}^{\theta}_U  \phi_{n \frac{\alpha}2 s}^{(Q)}) (q^{t_{\pi,w(1)}}, \ldots, q^{t_{\pi,w(n)}}) $$ with $N_{\pi}$ as defined in Proposition 4.3.1 and $\pi $'s are representations of $GL_n(F_p)$ of the form $\mathrm{Ind}_P^G(\otimes_i \mathrm{St}_{GL_{n_i}}(\varepsilon_i))$, such that $\varepsilon_i$ is either the trivial character or the unramified quadratic character. 

\item If $n-s = s$, then for an arbitrary odd $\alpha \in \mathbb Z_{> 0}$, we have $$|B(\mathbb F_{p^{\alpha}})| =| \mathrm{ker}^1(\mathbb Q, G)| \sum_{\pi} N_{\pi} p^{-\frac \alpha2 s(n-s)}\mathrm{Tr}_{\theta}(C_{\theta c} \phi_{n  \alpha s}, \pi) $$ 
$$   =   | \mathrm{ker}^1(\mathbb Q, G)| \sum_{\pi}  \sum\limits_{Q \in \mathcal P^{\theta}, Q = LU}\sum\limits_{w \in G_{P,Q}^{\theta}}N_{\pi}  p^{-\frac \alpha2 s(n-s)} \varepsilon_{Q,\theta}  \varepsilon_{Q, w}  \mathcal S(\hat{\chi}^{\theta}_U  \phi_{n \frac{\alpha}2 s}^{(Q)}) (q^{t_{\pi,w(1)}}, \ldots, q^{t_{\pi,w(n)}}) $$ with $N_{\pi}$ as defined in Proposition 4.3.1, and $\pi $'s are representations of $GL_n(F_p)$ of the form $\mathrm{Ind}_P^G(\otimes_i \mathrm{St}_{GL_{n_i}}(\varepsilon_i))$, such that $\varepsilon_i$ is either the trivial character or the unramified quadratic character. 
\end{enumerate} 
\end{theorem}
\begin{proof}

We may assume that $ n-s \not = s$, the proof for the case $n-s = s$ is almost identical: 

We use $n_i$ to denote $dim_{\overline{\mathbb Q}_{\ell}}(H_{et}^i(B_{\bar{\mathbb F}_q},\overline{\mathbb Q}_{\ell}))$, and we use $e_{i,0}, \ldots, e_{i,n_i}$ to denote the eigenvalues of $Frob_p$ acting on $H_{et}^i(B_{\bar{\mathbb F}_q},\overline{\mathbb Q}_{\ell})$ respectively. Note that $\mathrm{Gal}(\overline{\mathbb F}_q/\mathbb F_q)$ is an abelian group, then the eigenvalues of $Frob^{\alpha}_p$ acting on $H_{et}^i(B_{\bar{\mathbb F}_q},\overline{\mathbb Q}_{\ell})$ are $e^{\alpha}_{i,0}, \ldots, e^{\alpha}_{i,n_i}$. Then we have $$|B(\mathbb F_{q^{\alpha}})| = \sum\limits_i \sum_{j = 0}^{n_i } (-1)^ie^{\alpha}_{i,j}.$$ By applying Proposition 4.3.1 and Proposition 5.2.1, we have $$|B(\mathbb F_{q^{\alpha}})| =| \mathrm{ker}^1(\mathbb Q, G)| \sum_{\pi}  \sum\limits_{Q \in \mathcal P^{\theta}, Q = LU}\sum\limits_{w \in G_{P,Q}^{\theta}}N_{\pi} \varepsilon_{Q,\theta}  \varepsilon_{Q, w}  \mathcal S(\hat{\chi}^{\theta}_U  \phi_{n \alpha s}^{(Q)}) (q^{t_{\pi,w(1)}}, \ldots, q^{t_{\pi,w(n)}}) $$ for $\alpha$ sufficiently large. Also note that the above equality equals $$| \mathrm{ker}^1(\mathbb Q, G)| \sum_{\pi}  \sum\limits_{Q \in \mathcal P^{\theta}, Q = LU}\sum\limits_{w \in G_{P,Q}^{\theta}}N_{\pi} \varepsilon_{Q,\theta}  \varepsilon_{Q, w}  \mathcal S(\hat{\chi}^{\theta}_U  \phi_{n 1 s}^{(Q)}) (q^{\alpha t_{\pi,w(1)}}, \ldots, q^{\alpha t_{\pi,w(n)}}). $$ Therefore for $\alpha$ sufficiently large, we have $$|B(\mathbb F_{q^{\alpha}})|= | \mathrm{ker}^1(\mathbb Q, G)|\sum\limits_i \sum_{j = 0}^{n_i } (-1)^ie^{\alpha}_{i,j}$$ 
$$ =| \mathrm{ker}^1(\mathbb Q, G)|  \sum_{\pi}  \sum\limits_{Q \in \mathcal P^{\theta}, Q = LU}\sum\limits_{w \in G_{P,Q}^{\theta}}N_{\pi} \varepsilon_{Q,\theta}  \varepsilon_{Q, w}  \mathcal S(\hat{\chi}^{\theta}_U  \phi_{n 1 s}^{(Q)}) (q^{\alpha t_{\pi,w(1)}}, \ldots, q^{\alpha t_{\pi,w(n)}}).$$ Also note that $ \mathcal S(\hat{\chi}^{\theta}_U  \phi_{n 1 s}^{(Q)})$ are polynomial functions, thus, we conclude that the identity holds for all $\alpha \in \mathbb Z_{>0}$. 

\end{proof}

\begin{remark}
If $\alpha s$ is even, then we have that $$\mathrm{Tr}_{\theta}(C_{\theta c} \phi_{n \alpha s}, \mathrm{Ind}_P^G(\otimes_i \mathrm{St}_{GL_{n_i}}(\varepsilon_i))) =\mathrm{Tr}_{\theta}(C_{\theta c} \phi_{n \alpha s}, \mathrm{Ind}_P^G(\otimes_i \mathrm{St}_{GL_{n_i}})), $$
by considering the Hecke matrix of $J_{N_0}(\mathrm{Ind}_P^G(\otimes_i \mathrm{St}_{GL_{n_i}}(\varepsilon_i)))$ and $J_{N_0}(\mathrm{Ind}_P^G(\otimes_i \mathrm{St}_{GL_{n_i}}))$ respectively. 
\end{remark}

\subsection{Combinatorics of double cosets of symmetric groups}
	In last section, we define $G^{\theta}_{P,Q}  \subset S_n$ is defined to be $\{w \in G_{P,Q}  \subset S_n: w^{\theta} = w, w^{-1} S_{Q} w \cap S_P = \{id_{S_n}\}\}$ where $P = MN$ is the standard parabolic subgroup corresponding partition $(\lambda_1, \ldots, \lambda_s)$ and $Q = LU$ is the standard parabolic subgroup corresponding partition $(\mu_1, \ldots, \mu_t)$, $G_{P,Q}  \subset S_n$ is the set of minimal length representatives of $S_P\setminus S_n/ S_Q \subset S_n$ and $w^{\theta}$ is the minimal length representative in $S_P \theta (w) S_Q$ (recall that we define $ \theta (w) $ to be $A_n^{-1} wA_n$ where $A_n$ is the anti-diagonal matrix with entries being $1$).  
	
Thus $G^{\theta}_{P,Q} $ consists of $w \in S_n$ satisfying the following properties:

\begin{enumerate}
 \item $w$ is a minimal length representative in $S_P\setminus S_n/ S_Q$. 
 \item $w^{-1} S_{Q} w\cap S_P = \{id_{S_n}\}$.
 \item $w$ is the minimal length representative in $S_P \theta (w) S_Q$. 
\end{enumerate}
In the following, we give a combinatorics interpretation of the conditions (1),(2), and (3) above.

\begin{remark} If we write a computer program to compute the point-counting formula in Theorem 5.2.1, then the most challenging part is to determine which $w \in S_n$ belongs to $G^{\theta}_{P,Q} $. With our combinatorial interpretation of the conditions (1), (2), and (3) above, we have a more effective method to determine $G^{\theta}_{P,Q}$ algorithmically. 
\end{remark}  

Let $\alpha = (\alpha_1, \ldots, \alpha_s)$ be a partition of $n$. The \textbf{diagram of $\alpha$} is the set $\{(i, j) : 1 \leq i \leq s, 1 \leq j \leq \alpha_i\}$, which we denote by $T^{(\alpha_1, \ldots, \alpha_s)}$,
which we represent by a Young diagram in the usual way. For example $T^{(4,1,3)}$ is represented by
$$
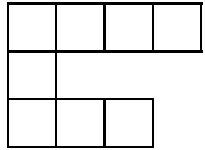
\begin{ytableau}
\empty &\empty &\empty&\empty \\
\empty  \\
\empty & \empty &\empty
\end{ytableau}$$ 
We refer to the elements of a Young diagram as \textbf{nodes}
Let $\alpha = (\alpha_1, \ldots, \alpha_s), \beta = (\beta_1, \ldots, \beta_t)$ be two partitions of $n$. \textbf{An $\alpha$-tableau of type $\beta$} is a function from the nodes of $\beta$ to $\mathbb Z_{>0}$ which takes each value $i \in \{1, \ldots , s\}$ exactly $\beta_i$ times, which we denote by $T^{(\alpha_1, \ldots, \alpha_s)}_{(\beta_1,\ldots,\beta_t)}$. We represent tableaux
by drawing the diagram of $\alpha$, and then writing the image of each node inside the corresponding box. For example $T_{(2,4,2)}^{(4,1,3)}$ is represented by
$$\begin{ytableau}
1 &1 &2&2 \\
2  \\
2 & 3 &3
\end{ytableau}$$ 
We label $T^{(\alpha_1, \ldots, \alpha_s)}$ as follows: We map $(i,j)$ to $\sum_{l = 1}^{\alpha_{i-1}} \alpha_l + j$ (we define $\alpha_0 = 0$), for example $T^{(4,1,3)}$ is labelled as follows: 
$$\begin{ytableau}
_1 &_2 &_3&_4 \\
_5  \\
_6 & _7 &_8
\end{ytableau}$$
Then $T^{(\alpha_1, \ldots, \alpha_s)}_{(\beta_1,\ldots, \beta_t)}$ may also be regard as a function from $\{1,2,\ldots,n\}$ to $\{1,\ldots,t\}$. 
We define the action of $S_n$ on $T^{(\alpha_1, \ldots, \alpha_s)}_{(\beta_1,\ldots, \beta_t)}$ by "place permutation", that is, if we regard $T^{(\alpha_1, \ldots, \alpha_s)}_{(\beta_1,\ldots, \beta_t)}$ as a function from $\{1,2,\ldots,n\}$ to $\{1,\ldots,t\}$, then given a $\sigma \in S_n$, $\sigma T^{(\alpha_1, \ldots, \alpha_s)}_{(\beta_1,\ldots, \beta_t)}$ is defined by $$\sigma T^{(\alpha_1, \ldots, \alpha_s)}_{(\beta_1,\ldots, \beta_t)} (k) = T^{(\alpha_1, \ldots, \alpha_s)}_{(\beta_1,\ldots, \beta_t)} (\sigma^{-1}k )$$ For example, $(1,2,3)(6,7) T_{(2,4,2)}^{(4,1,3)}$ is represented by 
$$\begin{ytableau}
2 &1 &1&2 \\
2  \\
3 & 2 &3
\end{ytableau}$$ 

\textbf{Combinatorics interpretation of conditions 1, 2 and 3:} We first provide a combinatorics interpretation of $w$ is a minimal length representative in $S_P\setminus S_n/ S_Q$ using Young tableau, following \cite{Wildon} (recall that $P$ corresponds partition $(\lambda_1,\ldots,\lambda_s)$ and $Q$ corresponds to partition $(\mu_1,\ldots,\mu_t)$). 

\begin{proposition} Let $w$ be an element of $S_n$. Then $w$ is a minimal length representative in $S_P\setminus S_n/ S_Q$ if and only if:
\begin{itemize}
    \item $w T^{(\lambda_1, \ldots, \lambda_s)}_{(\mu_1,\ldots, \mu_t)}$ has weakly increasing rows; i.e., $w T^{(\lambda_1, \ldots, \lambda_s)}_{(\mu_1,\ldots, \mu_t)}(i, j ) \leq w T^{(\lambda_1, \ldots, \lambda_s)}_{(\mu_1,\ldots, \mu_t)}(i , j+1)  $ for all $i,j$, 
    \item if $i < j$ and the $ith$ and $jth$ positions of $w T^{(\lambda_1, \ldots, \lambda_s)}_{(\mu_1,\ldots, \mu_t)}$ are equal, then $w(i) < w(j)$.
\end{itemize}

\end{proposition} 

\begin{proof}
See Theorem 4.1 of \cite{Wildon}
\end{proof}

Given a finite set $A$, let $S_A$ denote the permutation group of elements of $A$. When $A = \emptyset$, we define $S_A$ to be the trivial group with one element. We define $\Lambda_j= \sum_{l = 1}^{j} \lambda_l $ and $\Gamma_j= \sum_{l = 1}^{j} \mu_l $ (we also define $\Lambda_0  = 0$ and $\Gamma_0 = 0$), then $$S_P \cong S_{\{1,\ldots,\Lambda_1\}} \times S_{\{\Lambda_1 +1,\ldots,\Lambda_2\}} \times \ldots\times S_{\{\Lambda_{s-1} +1,\ldots,\Lambda_s\}}$$ and we have $$w^{-1} S_P w = S_{\{w^{-1}(1),\ldots,w^{-1(}\Lambda_1)\}} \times S_{\{w^{-1}(\Lambda_1 +1),\ldots,w^{-1}(\Lambda_2\})} \times \ldots\times S_{\{w^{-1}(\Lambda_{s-1} +1),\ldots,w^{-1}(\Lambda_s)\}}$$
and $$S_Q \cap w^{-1}S_P w = \prod\limits_{1 \leq i \leq s, 1 \leq j \leq t} S_{\{w^{-1}(\Lambda_{i-1} +1),\ldots,w^{-1}(\Lambda_{i})\}} \cap S_{\{\Delta_{j-1} +1,\ldots,\Delta_{j}\}}$$
Then we have that $S_Q \cap w^{-1}S_P w = \{id_{S_n}\}$ iff $|\{w^{-1}(\Lambda_{i-1} +1),\ldots,w^{-1}(\Lambda_{i})\} \cap \{\Delta_{j-1} +1,\ldots,\Delta_{j}\}|= 0 $ or $1$ for all $i , j$. In terms of Young tableau, we have that $S_Q \cap w^{-1}S_P w = \{id_{S_n}\}$ if and only if $w T^{(\lambda_1, \ldots, \lambda_s)}_{(\mu_1,\ldots, \mu_t)}$ has distinct elements in each rows. Applying Proposition 5.3.1, we have the following proposition:

\begin{corollary} Given $w \in S_n$, we have that $w$ satisfies conditions 1 and 2 (introduced at the beginning of this section) if and only if:

\begin{itemize}
    \item $w T^{(\lambda_1, \ldots, \lambda_s)}_{(\mu_1,\ldots, \mu_t)}$ has strictly increasing rows; i.e., $w T^{(\lambda_1, \ldots, \lambda_s)}_{(\mu_1,\ldots, \mu_t)}(i, j ) < w T^{(\lambda_1, \ldots, \lambda_s)}_{(\mu_1,\ldots, \mu_t)}(i , j+1) $ for all $i,j$,
    \item if $i < j$ and the $ith$ and $jth$ positions of $w T^{(\lambda_1, \ldots, \lambda_s)}_{(\mu_1,\ldots, \mu_t)}$ are equal, then $w(i) < w(j)$.
\end{itemize}
 
\end{corollary}

Now we look at condition 3: For $w \in S_n$, we define $R^{(\lambda_1, \ldots, \lambda_s)}_{(\mu_1,\ldots, \mu_t)} (w,i)$ to be the set of elements in the i-th row of $\theta(w)T^{(\lambda_1, \ldots, \lambda_s)}_{(\mu_1,\ldots, \mu_t)} $, that is, $$R^{(\lambda_1, \ldots, \lambda_s)}_{(\mu_1,\ldots, \mu_t)} (w,i) = \{a: a = w T^{(\lambda_1, \ldots, \lambda_s)}_{(\mu_1,\ldots, \mu_t)}(i, j ) \text{ for some j}\}.$$ By the construction of $w T^{(\lambda_1, \ldots, \lambda_s)}_{(\mu_1,\ldots, \mu_t)}$, we have that $w$ is the minimal length representative in $S_P \theta (w) S_Q$ if and only if $T^{(\lambda_1, \ldots, \lambda_s)}_{(\mu_1,\ldots, \mu_t)}$ and $\theta(w)T^{(\lambda_1, \ldots, \lambda_s)}_{(\mu_1,\ldots, \mu_t)} $ have the same elements at each row, i.e., $R^{(\lambda_1, \ldots, \lambda_s)}_{(\mu_1,\ldots, \mu_t)} (w,i)  = R^{(\lambda_1, \ldots, \lambda_s)}_{(\mu_1,\ldots, \mu_t)} (\theta(w),i) $ for all $i$. In summary, we have the following:

\begin{proposition}
Let $w \in S_n$. Then we have that $w \in G^{\theta}_{P,Q} $ if and only if:
\begin{itemize}
\item $w T^{(\lambda_1, \ldots, \lambda_s)}_{(\mu_1,\ldots, \mu_t)}$ has increasing rows, i.e., $w T^{(\lambda_1, \ldots, \lambda_s)}_{(\mu_1,\ldots, \mu_t)}(i, j ) < w T^{(\lambda_1, \ldots, \lambda_s)}_{(\mu_1,\ldots, \mu_t)}(i , j+1)  $ for all $i,j$. 
\item if $i < j$ and the $ith$ and $jth$ positions of $w T^{(\lambda_1, \ldots, \lambda_s)}_{(\mu_1,\ldots, \mu_t)}$ are equal, then $w(i) < w(j)$.
\item $R^{(\lambda_1, \ldots, \lambda_s)}_{(\mu_1,\ldots, \mu_t)} (w,i)  = R^{(\lambda_1, \ldots, \lambda_s)}_{(\mu_1,\ldots, \mu_t)} (\theta(w),i) $ for all $i$.
\end{itemize}
\end{proposition}

\newpage

\appendix

\section{The elliptic part of the geometric side of the trace formula}

In this appendix, we prove a lemma on the elliptic part of the geometric side of the trace formula. 

\subsection{Central elements in some algebraic groups}

we first study central elements in some algebraic groups.

\begin{proposition}
Let $G$ be an algebraic group over a field $F$, and let $H$ be an inner form of $G$, then we have $Z(G) \cong Z(H)$. 
\end{proposition}
\begin{proof}
Let $\phi: G_{\bar F} \rightarrow H_{\bar F}$ be an isomorphism over $\bar F$, then we have an isomorphism $\phi|_{Z(G_{\bar F})}: Z(G_{\bar F}) \rightarrow Z(H_{\bar F})$ between commutative algebraic groups. For an arbitrary $\sigma \in \mathrm{Gal}(\bar F/F)$, we have that $\phi^{-1} \circ \phi^{\sigma}$ is an inner isomorphism, therefore $(\phi|_{Z(G_{\bar F})})^{-1} \circ (\phi|_{Z(G_{\bar F})})^{\sigma}$ is an identity. Then we have that $\phi|_{Z(G_{\bar F})} = (\phi|_{Z(G_{\bar F})})^{\sigma}$ for all $\sigma \in \mathrm{Gal}(\bar F/F)$. Therefore, the morphism $\phi|_{Z(G_{\bar F})}: Z(G_{\bar F}) \rightarrow Z(H_{\bar F})$ is defined over $F$, and the lemma follows. 
 \end{proof}

We embed $GL_n(\mathbb A_f)$ into $\mathbb A_f^{n^2}$. For $N \in \mathbb Z_{>0}$, we use $GL_n(\hat{\mathbb Z}[\frac 1N])$ to denote $GL_n(\mathbb A_f) \cap (\hat{\mathbb Z}[\frac 1N]) ^{n^2}$. 

\begin{lemma}
Let $G$ be an affine algebraic group over $\mathbb Q$ such that $G(\mathbb R)$ is compact, let $K$ be a compact open subset of $G(\mathbb A)$, then $K \cap G(\mathbb Q ) \subset G(\mathbb Q) \subset G(\mathbb R) $ is discrete in $G(\mathbb R)$. 
\end{lemma}

\begin{proof}
Let $K_f$ denote the projection of $K$ onto $G(\mathbb A_f)$, then $K_f$ is a compact open subset of $G(\mathbb A_f)$. Also note that $K \cap \mathbb Q \subset K_f \cap G(\mathbb Q)$, thus, it is enough to prove that  $K_f \cap G(\mathbb Q) \subset G(\mathbb Q) \subset G(\mathbb R)$ is discrete in $G(\mathbb R)$.

 For all $g \in K_f$, there exists a compact open subgroup $K_{f,g} \subset g^{-1}K_f \subset G(\mathbb A_f)$, then $\{gK_{f,g}: g \in K_f\}$ is am open covering of $K_f$. By the compactness of $K_f$, there exists a finite subset of $\{gK_{f,g}: g \in K_f\}$ which is a covering of $K_f$. Let $\{g_1L_1, \ldots, g_sL_s\} \subset \{gK_{f,g}: g \in K_f\}$ be a finite sub-covering of $K_f$.

We fix an embedding $G \hookrightarrow GL_n$ for some $n \in \mathbb Z_{>0}$, then $L_i$ is a compact subgroup of $ GL_n(\mathbb A_f)$. By applying Lemma 2.4.7 of \cite{Loi19}, for each $i$, there exists a compact open subgroup $M_i \subset GL_n(\mathbb A_f)$ containing $L_i$. Then we only need to prove that $(\bigcup_{i=1}^s g_i M_i) \cap G(\mathbb Q) $ is discrete in $G(\mathbb R)$. 

Let $i \in \{1,2, \ldots , s\}$, we first prove that $g_i M_i \cap G(\mathbb Q)$ is discrete in $G(\mathbb R) $: Firstly, we have that $g_i M_i \cap GL_n(\mathbb Q)$ is compact in $GL_n(\mathbb R)$, because there exists an $N \in \mathbb Z_{>0}$ such that $g_i M_i \subset GL_n(\hat{\mathbb Z}[\frac 1N]) $, and $GL_n(\hat{\mathbb Z}[\frac 1N]) \cap GL_n(\mathbb Q) = GL_n(\mathbb Z[\frac 1N])$ is discrete in $GL_n(\mathbb R)$. We can now conclude that  $g_i M_i \cap G(\mathbb Q)$ is discrete in $G(\mathbb R) $ because of the following commutative diagram, where arrows are all continuous:

$$\begin{tikzcd}
g_i M_i \cap G(\mathbb Q) \arrow[rr, hook] \arrow[d, hook] &  & G(\mathbb R) \arrow[d, hook] \\
g_i M_i \cap GL_n(\mathbb Q) \arrow[rr, hook]              &  & GL_n(\mathbb R)             
\end{tikzcd}$$
Because $G(\mathbb R)$ is compact, then we may conclude that $g_i M_i \cap G(\mathbb Q)$ is finite. Therefore $(\bigcup_{i=1}^s g_i M_i) \cap G(\mathbb Q) $ is discrete in $G(\mathbb R)$. 
\end{proof}

\begin{remark}
Let $G$ be an affine algebraic group over $\mathbb Q$, and let $K$ be a compact open subgroup of $G(\mathbb A)$ (we do not assume the compactness of $G(\mathbb R)$). Repeating the above proof, we conclude that $K \cap G(\mathbb Q) $ is discrete in $G(\mathbb R)$. 
\end{remark}

In the following, let $G$ be an algebraic group over $\mathbb Q$, we assume that $Z(G)$ is isogenous over $\mathbb Q$ to a torus of the form $T_s \times T_a$ where $T_s$ splits and $T_a$ is anisotropic over $\mathbb R$. Then we have the following proposition:

\begin{proposition}
Let $K$ be a compact open subset of $G(\mathbb A)$, then $K \cap (Z(G)(\mathbb Q))$ is finite. 
\end{proposition}

\begin{proof}
Let $\phi: Z(G) \rightarrow T_s \times T_a$ be an isogeny over $\mathbb Q$.  First, note that $$ K \cap  (Z(G)(\mathbb Q)) = K  \cap  (Z(G)(\mathbb A)) \cap  (Z(G)(\mathbb Q)) \subset K_f  \cap  (Z(G)(\mathbb A_f)) \cap  (Z(G)(\mathbb Q)),$$ where $K_f$ is the projection of $K$ onto $G(\mathbb A_f)$. Let $K' $ denote $K_f  \cap  (Z(G)(\mathbb A_f))$, then $K'$ is a compact open subset of $Z(G)(\mathbb A_f)$, and it is enough to prove that $K' \cap Z(G)(\mathbb Q)$ is finite. 

 For all $g \in K'$, there exists a compact open subgroup $K_g \subset g^{-1}K' \subset Z(G)(\mathbb A_f)$, then $\{gK_g: g \in K'\}$ is an open covering of $K'$. By the compactness of $K'$, there exists a finite subset of $\{gK_g: g \in K'\}$, which is a covering of $K'$.

Then we may assume $K' = gL$, where $g \in K' \subset Z(G)(\mathbb A_f)$ and $L$ is a compact open subgroup of $Z(G)(\mathbb A_f)$. Then $\phi(gL ) = \phi(g)\phi(L)$ where $\phi(L)$ is a compact subgroup of $(T_s \times T_a)(\mathbb A_f)$. Note that there exists a compact open subgroup $M \subset (T_s \times T_a)(\mathbb A_f) $ such that $\phi(L) \subset M$ (see Lemma 2.4.7 of \cite{Loi19}). Then we only need to prove that  $\phi(g) M \cap (T_s \times T_a)(\mathbb Q)$ is finite, because $\phi$ is an isogeny and $ \phi(gL \cap (Z(G)(\mathbb Q))) \subset \phi(g)M \cap  ((T_s \times T_a) (\mathbb Q))$.

First, we show that $\phi(g) M \cap T_s(\mathbb Q) $ is finite. We may write $ T_s = \mathbb G_m^r$ for some $r \in \mathbb Z_{\geq 0}$, then $\phi(g) M \cap \mathbb G_m^r (\mathbb A_f)  \subset \mathbb G_{m, \mathbb Z}^r(\hat {\mathbb Z} [\frac 1N]) $ for some $N \in \mathbb Z_{>0}$. Then we have $$\phi(g) M \cap \mathbb G_m^r (\mathbb Q) = (\phi(g) M \cap \mathbb G_m^r (\mathbb A_f)) \cap \mathbb G_m^r (\mathbb Q) \subset  (\mathbb G_{m, \mathbb Z}^r(\hat {\mathbb Z} [\frac 1N]) \cap \mathbb G_m^r (\mathbb Q)). $$
And we have $\mathbb G_{m, \mathbb Z}^r(\hat {\mathbb Z} [\frac 1N]) \cap \mathbb G_m^r (\mathbb Q) = \mathbb G_{m, \mathbb Z}^r( \mathbb Z [\frac 1N])$ is finite, thus $\phi(g) M \cap T_s(\mathbb A) $ is finite.

Now we prove that $\phi(g) M \cap T_a(\mathbb Q) $ is finite. This follows because $\phi(g) M \cap T_a(\mathbb Q) $ is discrete in $T_a(\mathbb R)$ (see Lemma A.1.1), and $T_a(\mathbb R)$ is compact. 
\end{proof}

\subsection{An application to trace formula}

    In this section, we deduce a lemma on trace formula. Let $G$ be a connected reductive group over $\mathbb Q$. We assume the following:
\begin{itemize}
\item The centre $Z(G)$ is isogenous over $\mathbb Q$ to a torus of the form $T_s \times T_a$, where $T_s$ splits and $T_a$ is anisotropic over $\mathbb R$ (the condition before Proposition A.1.2). 
\item For all elliptic semi-simple elements $\gamma \in G(\mathbb Q)$, we assume that $\mathfrak R(G_{\gamma}/\mathbb Q)$ is trivial, where $G_{\gamma}$ is the centraliser of $\gamma$ in $G$ and $\mathfrak R(G_{\gamma}/\mathbb Q)$ is defined in $\S 4$ of \cite{Kot86}. 
\item The derived group $G_{der}$ of $G$ is simply connected, and $G_{der}$ has no $E_8$ factors (then $G_{\lambda}$ is connected for $\lambda \in G(\mathbb Q)$, see \cite{Kot82}; and $\tau(G_{der} ) = 1$, where $\tau(G_{der} )$ is the Tamagawa number of $G_{der}$, see \cite{Kot88}). 
\end{itemize}

Let $G^*$ be the quasi-split inner form of $G$, then we have the following lemma:
 
\begin{lemma}
For an arbitrary $f \in \mathcal H(G(\mathbb A))$, let $f^{G^*} \in \mathcal H(G^*(\mathbb A))$ denote the transfer of $f$. Then we have $T^G_e(f) = ST^{G^*}_e(f^{G^*})$, where $T_e(f) $ is the elliptic part of the geometric side of the trace formula for $G$ (as defined in 9.2 of \cite{Kot86}), and $ST_e(f^{G^*}) $ is the elliptic part of the geometric side of the stable trace formula for $G^*$ (as defined in 9.2 of \cite{Kot86}).
\end{lemma}

\begin{proof}
Let $\psi: G_{\overline {\mathbb Q}} \rightarrow G_{\overline {\mathbb Q}} ^{*}$ be an inner twisting. For each $\gamma_0 \in G(\mathbb Q)$, we choose $\gamma_0^* \in G^*(\mathbb Q)$ such that $\psi(\gamma_0)$ is stable conjugate to $\gamma_0^*$ (by $\S 6$ of \cite{Kot82}, such a choice is always possible. Then we have that $G^*_{\gamma^*_0}$ is an inner form of $G_{\gamma_0}$ with a canonical inner twisting given by $\psi|_{G_{\gamma_0, \overline{\mathbb Q}}}$ (see Lemma 5.8 of \cite{Kot82}). By applying Proposition A1.1, we have that $\psi|_{Z(G)}$ (resp. $\psi|_{Z(G_{\gamma_0})}$) gives isomorphism between $Z(G)$ and $Z(G^*)$( resp. $Z(G_{\gamma_0})$ and $Z(G^*_{\gamma^*_0})$), then by the definition of $\mathfrak R(G_{\gamma_0}/\mathbb Q)$ (see 4.6 of \cite{Kot86}), we have $\mathfrak R(G^*_{\gamma^*_0}/\mathbb Q) \cong \mathfrak R(G_{\gamma_0}/\mathbb Q) $ is trivial.  

For an arbitrary $\gamma_0 \in Z(G)(\mathbb Q)$, we have that $\psi(\gamma_0)\in Z(G^*)(\mathbb Q)$ (see Proposition A.1.1). By 9.3.1 in \cite{Kot86}, we have $SO_{\psi(\gamma_0)}(f^{G^*}) = e(\gamma_0) O_{\gamma_0} (f)$ with $e(\gamma)$ defined in 9.3 of \cite{Kot86} for $\gamma \in G(\mathbb A)$. For $\gamma \in G(\mathbb Q)$, by assumption, we know that $G_{\gamma}$ is connected, then $G_{\gamma, v}$ is connected for all places $v$ (because $G_{\gamma}$ is connected implies that $G_{\gamma} $ is geometrically connected, see Proposition 1.14 in \cite{Mil17}). Then we have $e(\gamma) = e(G_{\gamma}) = 1$ (see \cite{Kot83} and $e(G_{\gamma})$ is also defined in \cite{Kot83}). Therefore, we have
 $$f^{G}(\gamma_0) = O_{\gamma_0}(f) = SO_{\psi(\gamma_0)}(f^{G^*}) = f^{G^*}(\psi(\gamma_0)).$$ Applying Proposition A.1.1 again, we obtain 

\begin{align}
 \sum_{\gamma_0 \in Z(G^*)(\mathbb Q)} f^{G^*} (\gamma_0) =  \sum_{\psi(\gamma_0) \in Z(G)(\mathbb Q)} f(\psi(\gamma_0))  =   \sum_{\gamma_0 \in Z(G)(\mathbb Q)} f(\gamma_0)
\end{align}
 and by Proposition A.1.2, the sums are all finite sums. 

We use $E(G)$ (or $E(G^*)$) to denote the stable conjugacy class of elliptic semi-simple elements in $G(\mathbb Q)$ (or $G^*(\mathbb Q)$), and we use $E^*(G)$ (or $E^*(G^*)$) to denote the stable conjugacy class of non-central elliptic semi-simple elements in $G(\mathbb Q)$ (or $G^*(\mathbb Q)$). Then we have that $E(G) = E^*(G) \cup (Z(G)(\mathbb Q)), E(G^*) = E^*(G^*) \cup (Z(G^*)(\mathbb Q))$. By 9.1 and 9.6.5 in \cite{Kot86}, we have 
$$T^G_e(f)-  \sum_{\gamma_0 \in Z(G)(\mathbb Q)} \tau(G) f(\gamma_0)  = \sum_{\gamma_0 \in E^*(G^*)} \tau_1(G) \sum_{\gamma} e(\gamma) O_{\gamma}(f)$$ where $\tau_1(G) = \tau(G)/\tau(G_{SC}) = \tau(G)$ and $\gamma$ runs over a set of representatives for the $G(\mathbb A)$-conjugacy classes in $G(\mathbb A)$ contained in the $G(\overline{\mathbb A})$-conjugacy class of $\psi^{-1}(\gamma_0)$. By 9.3.1 in \cite{Kot86}, we have that $\sum_{\gamma} e(\gamma) O_{\gamma}(f)$ in the above formula is equal to $SO_{\gamma_0}(f^{G^*})$ (note that $SO_{\gamma_0}(f^{G^*}) = 0$ if $\gamma_0$ is not associated to any elements in $G(\mathbb Q)$). Therefore, we have $$T^G_e(f) =   \sum_{\gamma_0 \in Z(G)(\mathbb Q)} \tau(G) f(\gamma_0)  +  \sum_{\gamma_0 \in E^*(G^*)} \tau_1(G) \sum_{\gamma} e(\gamma) O_{\gamma}(f)$$
$$ =   \sum_{\gamma_0 \in Z(G^*)(\mathbb Q)} \tau(G^*) f^{G^*}(\gamma_0)  +  \sum_{\gamma_0 \in E^*(G^*)} \tau(G^*) SO_{\gamma_0} (f^{G^*})$$
(we use (1) and $\tau_1(G) = \tau(G)/\tau(G_{SC}) = \tau(G) = \tau(G^*)$, see \cite{Kot88})
$$=  \sum_{\gamma_0 \in Z(G^*)(\mathbb Q)} \tau(G^*) SO_{\gamma_0} (f^{G^*}) +  \sum_{\gamma_0 \in E^*(G^*)} \tau(G^*) SO_{\gamma_0} (f^{G^*})$$
$$ = \sum_{\gamma_0 \in E(G^*)} \tau(G^*) SO_{\gamma_0} (f^{G^*})  = ST_e^{G^*}(f^{G^*})$$ with $ST_e^{G^*}(f^{G^*})$ as defined in 9.2 in \cite{Kot86}.
\end{proof}
\clearpage
\addcontentsline{toc}{section}{References}
\renewcommand\refname{References}
\bibliographystyle{plain}
\bibliography{basic_stratum}
	
\end{document}